\documentclass[letterpaper,reqno,11pt,oneside,final]{amsart}
\pdfoutput=1
\usepackage{amsmath,amsthm,amsfonts,amssymb}
\usepackage{enumitem,comment}
\usepackage[T1]{fontenc}
\usepackage{times}
\usepackage{mathtools}
\usepackage{mathrsfs}
\usepackage{cases}
\usepackage{bm}
\usepackage[dvipsnames]{xcolor}
\usepackage{microtype}
\usepackage[totalheight=9.6in,totalwidth=6.15in,centering]{geometry}
\usepackage{graphics,graphicx}
\usepackage{tocvsec2}
\usepackage[extension=pdf,bookmarksopen]{hyperref}
\usepackage[color,notref,notcite]{showkeys}

\usepackage[style=alphabetic,sorting=nyt,natbib=true,maxnames=99,isbn=false,doi=false,url=false,giveninits=true,hyperref=auto,arxiv=abs,backend=bibtex]{biblatex}
\AtEveryBibitem{%
 \clearlist{language}}
\DeclareFieldFormat[article,inbook,incollection,inproceedings,patent,thesis,unpublished]{title}{#1\isdot}
\renewbibmacro{in:}{%
 \ifentrytype{article}{}{%
   \printtext{\bibstring{in}\intitlepunct}}}
\defbibheading{apa}[\refname]{\section*{#1}}
\DeclareFieldFormat{sentencecase}{\MakeSentenceCase{#1}} \renewbibmacro*{title}{%
 \ifthenelse{\iffieldundef{title}\AND\iffieldundef{subtitle}}{}
 {\ifthenelse{\ifentrytype{article}\OR\ifentrytype{inbook}%
     \OR\ifentrytype{incollection}\OR\ifentrytype{inproceedings}%
     \OR\ifentrytype{inreference}} {\printtext[title]{%
       \printfield[sentencecase]{title}%
       \setunit{\subtitlepunct}%
       \printfield[sentencecase]{subtitle}}}%
   {\printtext[title]{%
       \printfield[titlecase]{title}%
       \setunit{\subtitlepunct}%
       \printfield[titlecase]{subtitle}}}%
   \newunit}%
 \printfield{titleaddon}}
\newcommand{\citelink}[2]{\hyperlink{cite.\therefsection @#1}{#2}}

\definecolor{darkblue}{rgb}{0.13,0.13,0.39}
\hypersetup{colorlinks=true,urlcolor=darkblue,citecolor=darkblue,linkcolor=darkblue,%
  pdftitle=The KPZ fixed point,pdfauthor={K. Matetski, J. Quastel, D. Remenik}}

\DeclareSymbolFont{eulerletters}{U}{eur}{m}{n}
\newcommand{\funnyp}{\mathscr{P}}

\mathtoolsset{showmanualtags,showonlyrefs}

\renewcommand{\labelitemi}{\raisebox{0.075em}{\textendash}}

\newtheorem{thm}{Theorem}[section] 
\newtheorem{lem}[thm]{Lemma}
\newtheorem{theorem}[thm]{Theorem}
\newtheorem{prop}[thm]{Proposition}

\theoremstyle{definition}
\newtheorem{rem}[thm]{Remark}
\newtheorem{defn}[thm]{Definition}
\newtheorem{ex}[thm]{Example}

\newcommand{\fh}{\mathfrak{h}}
\newcommand{\fg}{\mathfrak{g}}
\newcommand{\ff}{\mathfrak{f}}
\newcommand{\fb}{\mathbf{b}}
\newcommand{\xx}{X}

\newcommand{\R}{R}
\newcommand{\Qt}{Q^*}

\newcommand{\I}{\uptext{i}}
\newcommand{\pp}{\mathbb{P}}

\newcommand{\ee}{\mathbb{E}}
\newcommand{\rr}{\mathbb{R}}

\newcommand{\zz}{\mathbb{Z}}
\newcommand{\aip}{\mathcal{A}}

\newcommand{\Bt}{{\mathcal{A}}_{2\to 1}}

\newcommand{\cD}{\mathcal{D}}
\newcommand{\cR}{\mathcal{R}}
\newcommand{\cm}{\mathcal{M}}

\newcommand{\p}{\partial}
\newcommand{\uno}[1]{\mathbf{1}_{#1}}
\newcommand{\ep}{\varepsilon}
\newcommand{\eps}{\varepsilon}
\newcommand{\vs}{\vspace{6pt}}
\newcommand{\wt}{\widetilde}

\newcommand{\qand}{\quad\text{and}\quad}
\newcommand{\qqand}{\qquad\text{and}\qquad}
\newcommand{\mQ}{\bar{Q}}

\newcommand{\fw}{\mathbf{w}}

\newcommand{\ts}{\hspace{0.1em}}
\newcommand{\tts}{\hspace{0.05em}}
\newcommand{\tsm}{\hspace{-0.1em}}
\newcommand{\ttsm}{\hspace{-0.05em}}

\newcommand{\inv}[1]{\frac{1}{#1}}

\newcommand{\itwopii}[1]{\frac{1}{(2\pi\I)^{#1}}}

\DeclareMathOperator\arctanh{arctanh}

\makeatletter
\newcommand\RedeclareMathOperator{%
  \@ifstar{\def\rmo@s{m}\rmo@redeclare}{\def\rmo@s{o}\rmo@redeclare}%
}
\newcommand\rmo@redeclare[2]{%
  \begingroup \escapechar\m@ne\xdef\@gtempa{{\string#1}}\endgroup
  \expandafter\@ifundefined\@gtempa
     {\@latex@error{\noexpand#1undefined}\@ehc}%
     \relax
  \expandafter\rmo@declmathop\rmo@s{#1}{#2}}
\newcommand\rmo@declmathop[3]{%
  \DeclareRobustCommand{#2}{\qopname\newmcodes@#1{#3}}%
}
\@onlypreamble\RedeclareMathOperator
\makeatother
\newcommand{\uptext}[1]{\text{\upshape{#1}}}
\DeclareMathOperator{\epi}{\uptext{epi}}
\DeclareMathOperator{\oepi}{\epi}

\DeclareMathOperator{\hypo}{\uptext{hypo}}

\DeclareMathOperator{\UC}{\uptext{UC}}
\DeclareMathOperator{\LC}{\uptext{LC}}
\DeclareMathOperator{\Ai}{\uptext{Ai}}
\DeclareMathOperator{\sgn}{\uptext{sgn}}
\DeclareMathOperator{\tr}{\uptext{tr}}
\RedeclareMathOperator{\det}{\mathop{\uptext{det}}}
\RedeclareMathOperator{\ker}{\mathop{\uptext{ker}}}
\RedeclareMathOperator{\exp}{\mathop{\uptext{exp}}}
\RedeclareMathOperator{\log}{\mathop{\uptext{log}}}
\RedeclareMathOperator*{\lim}{\mathop{\uptext{lim}}}
\RedeclareMathOperator*{\sup}{\mathop{\uptext{sup}}}
\RedeclareMathOperator*{\limsup}{\mathop{\uptext{lim\hspace{1pt}sup}}}
\RedeclareMathOperator*{\liminf}{\mathop{\uptext{lim\hspace{1pt}inf}}}
\RedeclareMathOperator*{\max}{\mathop{\uptext{max}}}
\RedeclareMathOperator*{\inf}{\mathop{\uptext{inf}}}
\RedeclareMathOperator*{\arctanh}{\mathop{\uptext{arctanh}}}
\RedeclareMathOperator*{\min}{\mathop{\uptext{min}}}
\RedeclareMathOperator*{\cos}{\mathop{\uptext{cos}}}
\RedeclareMathOperator*{\sin}{\mathop{\uptext{sin}}}
\RedeclareMathOperator*{\arg}{\mathop{\uptext{arg}}}
\RedeclareMathOperator{\Re}{\uptext{Re}}
\RedeclareMathOperator{\Im}{\uptext{Im}}
\renewcommand{\d}{\mathrm{d}}

\newcommand{\TASEP}{\uptext{TASEP}}

\newcommand{\cw}{\mathcal{W}}
\newcommand{\sW}{\mathsf{W}}
\newcommand{\sK}{\mathsf{K}}

\newcommand{\sQ}{\mathsf{N}}

\newcommand{\sI}{\mathsf{I}}

\newcommand{\Ml}{N}
\newcommand{\oM}{\overline{\Ml}}
\newcommand{\SM}{\mathcal{S}}
\newcommand{\SN}{\bar{\mathcal{S}}}

\newcommand{\fT}{\mathbf{S}}
\newcommand{\ft}{\mathbf{t}}
\newcommand{\fs}{\mathbf{s}}
\newcommand{\s}{\fs}
\renewcommand{\b}{\mathbf{b}}
\newcommand{\fx}{\mathbf{x}}
\newcommand{\fy}{\mathbf{y}}
\newcommand{\fz}{\mathbf{z}}
\newcommand{\fu}{\mathbf{u}}
\newcommand{\fa}{\mathbf{a}}
\newcommand{\fA}{\mathbf{A}}

\newcommand{\fB}{\mathbf{B}}
\newcommand{\fK}{\mathbf{K}}
\newcommand{\fI}{\mathbf{I}}
\newcommand{\ftau}{\bm{\tau}}

\newcommand{\gga}{\bar{\bm{\alpha}}}
\newcommand{\g}{\bar{\bm{\gamma}}}
\newcommand{\fsigma}{\bm{\sigma}}

\renewcommand{\P}{\chi}

\let\oldFootnote\footnote
\newcommand\nextToken\relax
\renewcommand\footnote[1]{%
    \oldFootnote{#1}\futurelet\nextToken\isFootnote}
\newcommand\isFootnote{%
    \ifx\footnote\nextToken\textsuperscript{,}\fi}

\def\dash---{\kern.16667em---\penalty\exhyphenpenalty\hskip.16667em\relax}

\numberwithin{equation}{section}

\let\oldmarginpar\marginpar
\renewcommand\marginpar[1]{\-\oldmarginpar[\raggedleft\footnotesize #1]%
  {\raggedright{\small\textsf{#1}}}}

\allowdisplaybreaks[2]

\begin{document}

\maxtocdepth{subsection}

\title{The KPZ fixed point}

\date{October 11, 2020}

\author{Konstantin Matetski} \address[K.~Matetski]{
  Department of Mathematics\\
  University of Toronto\\
  40 St. George Street\\
  Toronto, Ontario\\
  Canada M5S 2E4} \email{matetski@math.toronto.edu}

\author{Jeremy Quastel} \address[J.~Quastel]{
  Department of Mathematics\\
  University of Toronto\\
  40 St. George Street\\
  Toronto, Ontario\\
  Canada M5S 2E4} \email{quastel@math.toronto.edu}

\author{Daniel Remenik} \address[D.~Remenik]{
  Departamento de Ingenier\'ia Matem\'atica and Centro de Modelamiento Matem\'atico (UMI-CNRS 2807)\\
  Universidad de Chile\\
  Av. Beauchef 851, Torre Norte, Piso 5\\
  Santiago\\
  Chile} \email{dremenik@dim.uchile.cl}

\begin{abstract}  
An explicit Fredholm determinant formula is derived for the multipoint distribution of the height function of the totally asymmetric simple exclusion process (TASEP) with arbitrary right-finite initial condition.
The method is by solving the biorthogonal ensemble/non-intersecting path representation found by \cite{sasamoto,borFerPrahSasam}.
The resulting kernel involves transition probabilities of a random walk forced to hit a curve defined by the initial data. 

In the KPZ 1:2:3 scaling limit the formula leads in a transparent way to a Fredholm determinant formula, in terms of analogous kernels based on Brownian motion, for the transition probabilities of the scaling invariant Markov process at the centre of the KPZ universality class. 
The formula readily reproduces known special self-similar solutions such as the Airy$_1$ and Airy$_2$ processes.
The process takes values in real valued functions which look locally like Brownian motion, and is H\"older $1/3-$ in time.

Both the KPZ fixed point and TASEP are shown to be \emph{stochastic integrable systems} in the sense that the time evolution of their transition probabilities can be linearized through a new \emph{Brownian scattering transform} and its discrete analogue.
\end{abstract}

\maketitle
\tableofcontents

\section{The KPZ universality class}

All models in the one dimensional \emph{Kardar-Parisi-Zhang (KPZ) universality class} (random growth models, last passage percolation and directed polymers, random stirred fluids) have an analogue of the height function $h(t,x)$ (free energy, integrated velocity) which is conjectured to converge at large time and length scales ($\ep\searrow 0$), under the KPZ 1:2:3 scaling  
\begin{equation}\label{123}
\ep^{1/2} h(\ep^{-3/2} t, \ep^{-1} x) - C_\ep t, 
\end{equation}
to a universal fluctuating field $\fh(t,x)$ which does not depend on the particular model, but \emph{does} depend on the initial data class.  
Since many of the models are Markovian, the invariant limit process, the \emph{KPZ fixed point}, will be as well.
The purpose of this article is to describe this Markov process, and how it arises from certain microscopic models.

The KPZ fixed point should not be confused with the \emph{Kardar-Parisi-Zhang equation} \cite{kpz},
\begin{equation}\label{KPZ}
\partial_t h = \lambda(\partial_xh)^2    + \nu \partial_x^2h + \sigma \xi
\end{equation}
with $\xi$ a space-time white noise, which is a canonical continuum equation for random growth,  lending its name to the class.
One can think of the space of models in the class as having a trivial, Gaussian fixed point, the \emph{Edwards-Wilkinson fixed point}, given by \eqref{KPZ} with $\lambda=0$ and the 1:2:4 scaling $\ep^{1/2} h(\ep^{-2} t, \ep^{-1} x) - C_\ep t$, and the non-linear \emph{KPZ fixed point}, conjecturally given by sending $\nu\searrow 0$ in \eqref{KPZ} with $\sigma=\nu^{1/2}$.
The KPZ equation is just one of these many models, but it does play a distinguished role as the (conjecturally) unique heteroclinic orbit between the two fixed points.
The KPZ equation can be obtained from certain microscopic models in the \emph{weakly asymmetric} or \emph{intermediate disorder} limits \cite{berGiaco,akq2,mqrScaling,corwinTsai,corwinNica,corwinShenTsai} (which are not equivalent, see \cite{HQ}).
Since some of these models are partially solvable (in particular the asymmetric simple exclusion process, through the work of Tracy and Widom [\citelink{tracyWidomASEP1}{TW08-09}]), exact one point distributions are known for the KPZ equation for special initial data \cite{acq}.
These issues of the universality of the KPZ equation and its distributions comprise the \emph{weak KPZ universality conjecture}.
\nocite{tracyWidomASEP2,tracyWidomASEP1,tracyWidomASEP3}

However, the KPZ equation is \emph{not} invariant under the KPZ 1:2:3 scaling \eqref{123}, which is expected to send it, along
with all other models in the class, to the true universal (strong coupling, long time) fixed point.  
In modelling, for example, edges of bacterial colonies, forest fires, or spread of genes, the non-linearities or noise are often not weak, and it is really the fixed point that should be used in approximations and not the KPZ equation.
However, progress has been hampered by a complete lack of understanding of the time evolution of the fixed point itself.   
Essentially all one had was fixed time distributions of a few special self-similar solutions, the \emph{Airy processes}.

Under the KPZ 1:2:3 scaling \eqref{123} the coefficients of \eqref{KPZ} transform as $(\nu,\sigma^2)\mapsto \ep^{1/2}(\nu,\sigma^2)$.  
A naive guess would then be that the fixed point is nothing but the vanishing viscosity ($\nu\searrow 0$) solution of the Hamilton-Jacobi equation 
\begin{equation}
\partial_t h = \lambda(\partial_xh)^2    + \nu \partial_x^2h 
\end{equation} given (for $\lambda>0$) by the Hopf-Lax formula
\begin{equation}\label{eq:invbur}
h(t,x) = \sup_y\{ -\tfrac1{4\lambda t}(x-y)^2 + h_0(y)\}.
\end{equation}
It is \emph{not}. One of the key features of the class is a stationary solution consisting of (non-trivially) time dependent Brownian motion height functions (or discrete versions of it).
But Brownian motions are not invariant for Hopf's formula (see \cite{frachebourgMartin} for the computation).
Our story has a stronger parallel in the dispersionless ($\nu\to 0$) limit of the (integrated) Korteweg-de Vries (KdV) equation
\begin{equation}\label{}
\partial_t h = \lambda(\partial_xh)^2    + \nu \partial_x^3h. 
\end{equation}
Brownian motions \emph{are} invariant for all $\nu$ (at least in the periodic case \cite{QV}).
But as far as we are aware, the zero dispersion limit has only been done on a case by case basis, with no general formulas.
One can imagine that the various schemes lead to different weak solutions of the ill-posed Hamilton-Jacobi equation $\partial_t h = \lambda(\partial_xh)^2$, with only the vanishing viscosity solution being characterized so far, through the entropy condition in its various manifestations.
However, in our situation, where $h(t,x)$ is locally Brownian in $x$, it is far from clear that the notion of weak solution can have any meaning whatsoever.

The KPZ fixed point is given by a variational formula (see Thm.\,\ref{thm:airyvar}), analogous to \eqref{eq:invbur}, but with  a residual forcing noise, the \emph{Airy sheet}.
Unfortunately, our techniques do not allow us to characterize this noise.
Instead, \emph{we obtain a complete description of the Markov field $\fh(t,x)$ itself through the exact calculation of its transition probabilities}.
These transition probabilities are given in \eqref{eq:twosided-ext} and define the invariant Markov process.

The \emph{strong KPZ universality conjecture} (still wide open) is that this fixed point is the limit under the scaling \eqref{123} for any model in the class, loosely characterized by having: 1. Local dynamics; 2.  Smoothing mechanism;  3.  Slope dependent growth rate (lateral growth);  4.  Space-time random forcing with rapid decay of correlations.   Alternatively, convergence to the fixed point can be taken as the  definition of the KPZ universality class.

Universal fixed points are a theme in probability and statistical physics:  2d critical Ising, SLE, Liouville quantum gravity/Brownian map, the Brownian web, and the continuum random tree have offered asymptotic descriptions for huge classes of models.  
In general, these have been obtained as non-linear transformations of Brownian motions or Gaussian free fields, and their description relies to a large degree on symmetry - often conformal invariance. 

In the case of $\phi^4_d$ \cite{simon-Pphi2}, the main tool is perturbation theory.
Even the recent theory of regularity structures \cite{hairerReg}, which makes sense of the KPZ equation \eqref{KPZ}, does so by treating the non-linear term as a kind of perturbation of the linear equation.

In our case, we have a non-perturbative two-dimensional field theory with a skew symmetry, and a solution should not in principle even be
expected.  What saves us is the one-dimensionality of the fixed time problem, and the fact that several discrete models in the class have an explicit description using non-intersecting paths.  Here we work with TASEP, obtaining a complete description of the transition probabilities in a form which allows us to pass transparently to the 1:2:3 scaling limit\footnote{\label{foot:variants}The method works for several variants of TASEP which also have a representation through biorthogonal ensembles, such as discrete time TASEPs and PushASEP, see \cite{mqr-variants}.}.
In a sense, a recipe for the solution of TASEP has existed since the work of \cite{sasamoto}, who discovered a highly non-obvious representation in terms of non-intersecting paths which can in turn be studied using biorthogonal ensembles \cite{borFerPrahSasam}.
However, the biorthogonalization was only implicit, and one had to rely on exact solutions for a couple of special initial conditions to obtain the asymptotic Tracy-Widom distributions $F_{\text{GUE}}$ and $F_{\text{GOE}}$  \cite{tracyWidom,tracyWidom2}, the Baik-Rains distribution $F_{\text{BR}}$ \cite{baikRains}, and their spatial versions, the Airy processes \cite{johanssonShape,johansson,sasamoto,borFerPrahSasam,bfp,baikFerrariPeche}.
In this article, motivated by the probabilistic interpretation of the path integral forms of the kernels in the Fredholm determinant formulas for these processes, and exploiting the skew time reversibility of TASEP, we are able to obtain a general formula in which the TASEP kernel is given by a transition probability of a random walk forced to hit the initial data.

\smallskip
We end this introduction with an outline of the paper and a brief summary of our results.
Sec.\,\ref{subsec:bioth} recalls and solves the biorthogonal representation of TASEP, motivated by the path integral representation, which is derived in the form we need it in Appx.\,\ref{app:proofPathIntTASEP}.
The biorthogonal functions appearing in the resulting Fredholm determinants are then recognized as hitting probabilities in Sec.\,\ref{subsec:hitting}, which allows us to express the kernels in terms of expectations of functionals involving a random walk forced to hit the initial data.
The determinantal formulas for TASEP with arbitrary right-finite initial conditions are in Thm.\,\ref{thm:tasepformulas}.  
In Sec.\,\ref{sec:123}, we pass to the KPZ 1:2:3 scaling limit to obtain determinantal formulas for transition probabilities of the KPZ fixed point, which is defined formally in Def. \ref{def:fixpt}.
This limit is computed using right-finite initial TASEP data, but since we have exact formulas, we can obtain  a very strong estimate (Lem. \ref{cutofflemma}) on the propagation speed of information which allows us to show there is no loss of generality in doing so.
We then work in Sec.\,\ref{sec:tightandmarkov} to show that the Chapman-Kolmogorov equations hold.
This is done by obtaining a uniform bound on the local H\"older $\beta<1/2$ norm of the approximating Markov fields.
The proof is in Appx.\,\ref{sec:reg}. 
Sec.\,\ref{sec:invariant} opens with the introduction of the Brownian scattering transform, which is the main ingredient in our Fredholm determinant formulas for the KPZ fixed point, while Sec.\,\ref{sec:fixedpt} gives
the general formulas for the transition probabilities of the KPZ fixed point (in \eqref{eq:twosided-ext} and Prop. \ref{prop:pathint-fixedpt}); readers mostly interested in the physical implications may wish to skip directly there.   
The rest of Sec.\,\ref{sec:invariant} gives the key properties of the KPZ fixed point: regularity in space and time and local Brownian behavior, various symmetries, variational formulas in terms of the Airy sheet, and equilibrium space-time covariance; we also show how to recover some of the classical Airy processes from our formulas.
Secs.\,\ref{sec:123} and \ref{sec:invariant} are done at the level of pointwise convergence of kernels.
The convergence of the kernels is upgraded to trace class in Appx.\,\ref{app:hs-estimates}, where the remaining technical details are filled in.

\smallskip
So, in a  sense, everything follows once one is able to explictly biorthogonalize TASEP.
We begin there.

\section{TASEP} 
\label{sec:TASEP}

The \emph{totally asymmetric simple exclusion process} (TASEP) consists of particles with positions\linebreak $\cdots <\xx_t(2)<\xx_t(1)< \xx_t(0)< \xx_t(-1)<\xx_t(-2)< \cdots$
on $\zz\cup\{-\infty,\infty\}$ performing totally asymmetric nearest neighbour random walks with exclusion: Each particle independently attempts jumps to the neighbouring site to the right at rate $1$, the jump being allowed only if that site is unoccupied  (see \cite{ligg1} for the non-trivial fact that the process with an infinite number of particles makes sense).
Placing a (necessarily infinite) number of particles at $\pm\infty$ allows for left- or right-finite data with no change of notation, the particles at $\pm\infty$ playing no role in the dynamics.  
We follow the standard practice of ordering particles from the right; for right-finite data the rightmost particle is labelled $1$, unless indicated otherwise.
Let
\begin{equation}
\xx^{-1}_t(u) = \min \{k \in \zz : \xx_t(k) \leq u\}
\end{equation}
denote the label of the rightmost particle which sits to the left of, or at, $u$ at time $t$. 
The \emph{TASEP  height function} associated to $X_t$ is given for $z\in\zz$ by 
\begin{equation}\label{defofh}
 h_t(z) = -2\tsm\left(\xx_t^{-1}(z-1) - \xx_0^{-1}(-1) \right) - z,
\end{equation}  which fixes $h_0(0)=0$.

The height function is a simple random walk path $h_t(z+1) = h_t(z) +\hat{\eta}_t(z)$ with $\hat{\eta}_t(z)=1$ if there is a particle at $z$ at time $t$ and $-1$ if there is no particle at $z$ at time $t$.
The dynamics of $h_t$ is that local max's become local min's at rate $1$; i.e. if $h_t(z) = h_t(z\pm 1) +1$ then $h_t(z)\mapsto h_t(z)-2$ at rate $1$, the rest of the height function remaining unchanged.
One can think of independent rate one Poisson processes, one for each site $z\in \zz$.
At the jump time of the Poisson process at $z$, we check to see if the height function has a local max there.
If it is we flip it to a local min.
We can also easily extend the height function to a continuous function of $x\in \rr$ by linearly interpolating between the integer points. The evolution of an initial height function $h$ is a deterministic function of the underlying Poisson processes, providing a coupling of the evolution from different initial conditions, which preserves the partial order $h\preceq\tilde h$ if $h(x)\leq\tilde h(x)$ for all $x$.

\subsection{Biorthogonal ensembles}
\label{subsec:bioth}

TASEP was first solved by \citet{MR1468391} using the coordinate Bethe ansatz.   
He showed that the transition probability for $N$ particles has a determinantal form,
\begin{equation}\label{eqGreen}
\pp (X_t( 1)= x_1,\ldots,X_t(N)=x_N)=\det(F_{i-j}(t,x_{N+1-i}-X_0(N+1-j)))_{1\leq i,j\leq N}
\end{equation}
with
\begin{equation}\label{eqFn}
F_{n}(t,x)=\frac{(-1)^n}{2\pi \I} \oint_{\Gamma_{0,1}} \d w\,\frac{(1-w)^{-n}}{w^{x-n+1}}e^{t(w-1)},
\end{equation}
where $\Gamma_{0,1}$ is any positively oriented simple loop which includes $w=0$ and $w=1$.
To mesh with our convention of infinitely many particles, we can place particles $X_0(j)$, $j\le 0$ at $\infty$ and $X_0(j)$, $j>N$ at $-\infty$.
Remarkable as it is, this formula is not conducive to asymptotic analysis where we want to consider the later positions of $M\ll N$ of the particles; one has to find an effective way to sum over the positions of the other $N-M$ particles and, at the same time, to get rid of the dependence in $N$ (which needs to go to infinity) of the dimension of the determinant.
This was overcome by \cite{sasamoto,borFerPrahSasam}, who were able to rewrite the right hand side of \eqref{eqGreen} in terms of a certain Lindstr\"om-Gessel-Viennot/Karlin-McGregor scheme \cite{karlinMcGregor,gesselViennot} involving a (signed) non-intersecting line ensemble, and from that obtain the desired probabilities implicitly from the following biorthogonalization problem.

First for a fixed vector $a\in\rr^m$ and indices $n_1<\dotsc<n_m$ we introduce the functions
\begin{equation}\label{eq:defChis}
\chi_a(n_j,x)=\uno{x>a_j},\qquad\bar\chi_a(n_j,x)=\uno{x\leq a_j},
\end{equation}
which we also regard as multiplication operators acting on the space $\ell^2(\{n_1,\dotsc,n_m\}\times\zz)$ (and later on $L^2(\{\fx_1,\dotsc,\fx_m\}\times\rr)$).
We will use the same notation if $a$ is a scalar, writing 
\begin{equation}\label{eq:defChisScalar}
\chi_a(x)=1-\bar\chi_a(x)=\uno{x>a}.
\end{equation}

\begin{thm}[\cite{borFerPrahSasam}] \label{thm:BFPS}
Suppose that TASEP starts with particles labeled $1,2,\dotsc$ (so that, in particular, there is a rightmost particle)\footnote{We are assuming here that $X_0(1)<\infty$ (and thus $X_0(j)<\infty$ for all $j>1$ too); particles at $-\infty$ are allowed.}\footnote{The \cite{borFerPrahSasam} result is stated only for initial conditions with finitely many particles, but the extension to right-finite (infinite) initial conditions is straightforward because, given fixed indices $n_1<n_2<\dotsm<n_m$, the distribution of $\big(X_t(n_1),\dotsc,X_t(n_m)\big)$ does not depend on the initial positions of the particles with indices beyond $n_m$.} and let $n_1,\dotsc,n_m$ be distinct positive integers.
Then for $t>0$ we have
\begin{equation}\label{eq:extKernelProbBFPS}
  \pp\!\left(X_t(n_j)>a_j,~j=1,\dotsc,m\right)=\det\!\left(I-\bar\chi_aK_t\bar\chi_a\right)_{\ell^2(\{n_1,\dotsc,n_m\}\times\zz)},
\end{equation}
where $\det$ is the Fredholm determinant (see \eqref{FD} for the definition),
\begin{equation}\label{eq:Kt}
K_t(n_i,x_i;n_j,x_j)=-Q^{n_j-n_i}(x_i,x_j)\uno{n_i<n_j}+\sum_{k=1}^{n_j}\Psi^{n_i}_{n_i-k}(x_i)\Phi^{n_j}_{n_j-k}(x_j),
\end{equation}
and where\footnote{We have conjugated the kernel $K_t$ from \cite{borFerPrahSasam} by $2^x$ for later convenience (see \eqref{eq:path-int-kernel-TASEPgem} and the discussion following it). The additional $X_0(n-k)$ in the power of 2 in the $\Psi^n_k$'s has also been added for convenience, and is allowed because it just means that the $\Phi^n_k$'s have to be multiplied by $2^{X_0(n-k)}$.}
\begin{equation}
  Q(x,y)=\frac{1}{2^{x-y}}\uno{x>y} 
\end{equation}
and\ts\footnote{Note that, from \eqref{eqFn} and \eqref{eq:defPsi}, $\Psi^n_k(x)=(-1)^k2^{X_0(n-k)-x}F_{-k}(t,x-X_0(n-k))$ for $k\geq0$.}, for $k\leq n-1,$
\begin{equation}\label{eq:defPsi}
\Psi^n_k(x)=\frac1{2\pi\I}\oint_{\Gamma_0}\d w\,\frac{(1-w)^k}{2^{x-X_0(n-k)}w^{x+k+1-X_0(n-k)}}e^{t(w-1)},
\end{equation}
where $\Gamma_0$ is any positively oriented simple loop including the pole at $w=0$ but not the one at $w=1$.
The functions $\Phi_k^{n}(x)$, $k=0,\ldots,n-1$, are defined implicitly by 
\begin{enumerate}[label={\normalfont (\arabic{*})}]
\item The biorthogonality relation $\sum_{x\in\zz}\Psi_k^{n}(x)\Phi_\ell^{n}(x)=\uno{k=\ell}$;
\label{ortho}
\smallskip
\item  $2^{-x}\Phi^n_k(x)$ is a polynomial of degree at most $n-1$ in $x$ for each $k$.\label{poly}
\end{enumerate} 
\end{thm}

The initial data appear in a simple way in the $\Psi_k^n$, which can be computed explicitly.  $Q^m$ is easy,
\begin{equation}
  Q^{m}(x,y)=\frac{1}{2^{x-y}}\binom{x-y-1}{m-1}\uno{x\ge y+m},\label{eq:Qpow}
\end{equation}
and moreover, as operators on $\ell^2(\zz)$,  $Q$ and $Q^m$ are invertible:
\begin{equation}\label{eq:Qinv}
  Q^{-1}(x,y)=2\cdot\uno{x=y-1}-\uno{x=y},\qquad Q^{-m}(x,y)=(-1)^{y-x+m}2^{y-x}\binom{m}{y-x}.
\end{equation}
It is not hard to check \cite[Eq. 3.22]{borFerPrahSasam} that for all $m,n\in \zz$
\begin{equation}
Q^{n-m}\Psi^n_{n-k}=\Psi^m_{m-k}\label{eq:QmnPsi}
\end{equation}
so, in particular, $\Psi^n_{k}=Q^{-k}\Psi^{n-k}_0$.
We introduce
\begin{equation}\label{def:rt}
\R_t= e^{-\frac{t}2( I + \nabla^-)},\qquad \nabla^-f(x) = f(x)-f(x-1),
\end{equation}  
which can also be defined through its integral kernel (valid for all $t\in\rr$)
\begin{equation}\label{eq:kernelrt}
R_t(x,y)=e^{-t}\frac{t^{x-y}}{2^{x-y}(x-y)!}\uno{x\geq y}=\frac{1}{2\pi \I} \oint_{\Gamma_0} \tsm \d w \,\frac{e^{t(w-1)}}{2^{x-y} w^{x - y +1}}.
\end{equation}
Observe that $\Psi^{n}_0=\R_t\delta_{X_0(n)}$ with $\delta_y(x)=\uno{x=y}$.
$Q$ and $R_t$ commute, because the kernels $Q(x,y)$ and $\R_t(x,y)$ only depend on $x-y$, and thus we obtain the decomposition
\begin{equation}\label{eq:Psi-n-chgd}
  \Psi^n_{k}=\R_tQ^{-k}\delta_{X_0(n-k)}.
\end{equation}

The $\Phi_k^n$, on the other hand, are defined only implicitly through \ref{ortho} and \ref{poly}.
Only for a few special cases of initial data (step, see e.g. \cite{dimers}; and periodic \cite{borFerPrahSasam,bfp,bfs}) were they known, and hence only for those choices asymptotics could be performed in the TASEP and related cases, leading to the Tracy-Widom $F_{\text{GUE}}$ and $F_{\text{GOE}}$ one-point distributions, and then later to the Airy processes for multipoint distributions.

We are now going to solve for the $\Phi_k^n$ for any initial data.
Let us explain how this can be done starting just from the solution for step initial data ($X_0(i)=-i$, $i\geq1$).
The derivation is based on two main ingredients.
The first is a \emph{path integral} version of the extended kernel formula \eqref{eq:extKernelProbBFPS}  for the TASEP finite dimensional distributions (see Appx.\,\ref{app:proofPathIntTASEP} for the proof):
\begin{multline}
\pp\!\left(X_t(n_j)>a_j,~j=1,\dotsc,m\right)
\\=\det\!\big(I-K^{(n_m)}_{t}(I-Q^{n_1-n_m}\P_{a_1}Q^{n_2-n_1}\P_{a_2}\dotsm Q^{n_m-n_{m-1}}\P_{a_m})\big)_{L^2(\zz)}\label{eq:path-int-kernel-TASEPgem}
\end{multline}
for $n_1<n_2<\dotsm<n_m$, where $K^{(n)}_t=K_t(n,\cdot;n,\cdot)$.
Such formulas were first obtained in \cite{prahoferSpohn} for the Airy$_2$ process (see also \cite[App. A]{prolhacSpohn}), and later extended to the Airy$_1$ process in \cite{quastelRemAiry1} and then to a very wide class of processes in \cite{bcr}.
The key is to recognize the kernel $Q(x,y)$ as the transition probabilities of a random walk (which is why we conjugated the \cite{borFerPrahSasam} kernel by $2^x$) and then $\P_{a_1}Q^{n_2-n_1}\P_{a_2}\dotsm Q^{n_m-n_{m-1}}\P_{a_m}(x,y)$ as the probability that this walk goes from $x$ to $y$ in $n_m-n_1$ steps, staying above $a_1$ at time $n_1$, above $a_{2}$ at time $n_2$, etcetera.

The second ingredient is the \emph{skew time reversibility of TASEP}.  From the description just after \eqref{defofh}, it is clear that the evolution rule for the height function backwards in time is the same as that of minus the height function forward in time.
We use it in the form\footnote{Evolve an initial height function $f$ forward through a realization of the Poisson processes from time $0$ to $t$ and call the result $h_{0\to t}(f)$ (see the paragraph right after \eqref{defofh}). Evolve $g$ backwards through the same realization from time $t$ to $0$ to obtain $h_{t\to 0}(g)$. Both maps preserve the partial order $h\preceq\tilde h$, so we have, for each fixed realization of the Poisson processes, $h_{0\to t}(f)\preceq g\Longleftrightarrow f\preceq h_{t\to 0}(g)$. On the other hand, the standard time reversibility property for TASEP says that $h_{t\to 0}(g)$ has the same distribution as $-h_{0\to t}(-g)$. This proves \eqref{timerev}.}
\begin{equation}\label{timerev}
\pp_{f}\big(h_{t}(x)\le g(x),~\,x\in\zz\big) = \pp_{-g}\big(h_{t}(x)\le -f(x),~\,x\in \zz\big),
\end{equation}
the subscript indicating the initial data.

Suppose we have the solution \eqref{eq:Kt} for step initial data centered at $x_0$, which means $h_0(x)$ is the peak $-|x-x_0|$.  
The multipoint distribution at time $t$ is given by \eqref{eq:path-int-kernel-TASEPgem}, but we can use \eqref{timerev} to reinterpret it as the one point distribution of $h_t$ at  $x_0$, starting from an initial condition built out of a series of $m$ peaks centered at $n_1,\dotsc,n_m$ with heights $-a_1,\dotsc,-a_m$.
From this we can guess a formula for the multipoint distributions by extending the resulting kernel in the usual way, as in \eqref{eq:Kt}.
This last step is not fully justified at this stage, but we can use the resulting formula to simply guess the form of the biorthogonal functions $\Phi^n_k$, based on the representation of the kernel in \eqref{eq:path-int-kernel-TASEPgem} in terms of the hitting probability for a random walk.  
Thm.\,\ref{thm:BFPS} is then set up perfectly, because it allows us to easily prove that the guess is correct.

This gives us our key result.

\begin{thm}\label{thm:h_heat}
Fix $0\le k <n$ and consider particles at $\xx_0(1)>\xx_0(2)>\dotsm>\xx_0(n)$. 
Let $h^n_k(\ell, z)$ be the unique solution to the initial--boundary value problem for the backwards heat equation
\mathtoolsset{showonlyrefs=false}
\begin{subnumcases}{\label{bhe}}
(\Qt)^{-1}h^n_k(\ell,z)=h^n_k(\ell+1,z) &  $\ell<k,\,z \in \zz$;\label{bhe1}\\ 
h^n_k(k,z)=2^{z-X_0(n-k)} & $z \in \zz$;\label{bhe2}\\ 
h^n_k(\ell,X_0(n-\ell))= 0 & $\ell<k$.\label{bhe3}
\end{subnumcases} 
\mathtoolsset{showonlyrefs=true}
Then the functions $\Phi^n_k$ from Thm.\,\ref{thm:BFPS} are given by
\begin{equation}\label{eq:defPhink}
  \Phi^n_k(z) = (\R_t^{*})^{-1}h^n_k(0,\cdot)(z)=\sum_{y\in \zz} h^{n}_k(0,y)\R_t^{-1}(y,z).
\end{equation}
Here $Q^*(x,y)=Q(y,x)$ is the kernel of the adjoint of $Q$ (and likewise for $\R_t^*$).
\end{thm}

\begin{rem} 
It is not true in general that $\Qt h^{n}_k(\ell+1,z)=h^{n}_k(\ell,z)$.
In fact, $\Qt h^{n}_{k}(k,z)$ is divergent.
\end{rem}

\begin{proof}
The existence and uniqueness of solutions of \eqref{bhe1}--\eqref{bhe3} is an elementary consequence of the fact that $\ker (Q^*)^{-1}$ has dimension $1$ and it is spanned by the function $2^z$, which allows us to march forwards from the initial condition $h^n_k(k,z)=2^{z-X_0(n-k)}$ uniquely solving the boundary value problem $h^n_k(\ell,X_0(n-\ell))= 0$ at each step\footnote{As a linear operator, $Q^*$ acts on $\ell^1(\zz)$, and it is there that $Q^*$ is invertible, with inverse $(Q^*)^{-1}$ defined by \eqref{eq:Qinv}. However, \eqref{bhe1}--\eqref{bhe3} are being solved in the space of all sequences, in which the matrix $(Q^*)^{-1}$ does have a non-trivial kernel.}.

Before turning to the proof of \eqref{eq:defPhink} we need to prove that $2^{-x}h^{n}_{k}(0,x)$ is a polynomial of degree at most $k$. We proceed by induction.
Note first that, by \eqref{bhe2}, $2^{-x}h^{n}_{k}(k,x)$ is a polynomial of degree 0.
Assume now that $\tilde h^n_k(\ell,x)\coloneqq2^{-x}h^{n}_{k}(\ell,x)$ is a polynomial of degree at most $k-\ell$ for some $0<\ell\leq k$.
By \eqref{bhe1} and \eqref{eq:Qinv} we have
\begin{equation}\label{eq:tildeh-eq}
  \tilde h^{n}_{k}(\ell,y)=2^{-y}(\Qt)^{-1}h^{n}_{k}(\ell-1,y)=\tilde h^{n}_{k}(\ell-1,y-1)-\tilde h^{n}_{k}(\ell-1,y).
\end{equation}
Taking $x\geq X_0(n-\ell+1)$ and summing \eqref{eq:tildeh-eq} gives $\tilde h^{n}_{k}(\ell-1,x)=-\sum_{y=X_0(n-\ell+1)+1}^{x}\tilde h^{n}_{k}(\ell,y)$ thanks to \eqref{bhe3}, which by the inductive hypothesis is a polynomial of degree at most $k-\ell+1$ in $x$.
Similarly, taking $x<X_0(n-\ell+1)$ we get $\tilde h^{n}_{k}(\ell-1,x)=\sum_{y=x+1}^{X_0(n-\ell+1)}\tilde h^{n}_{k}(\ell,y)$, which again is a polynomial of degree at most $k-\ell+1$.
The two polynomials are the same, and thus the claim follows.

Now we check the biorthogonality condition \ref{ortho}. 
Using \eqref{eq:Psi-n-chgd} we get
\begin{align}
	\sum_{z\in\zz}\Psi^n_\ell(z)\Phi^n_k(z)
	&=\sum_{z_1,z_2\in\zz}\sum_{z\in\zz}\R_t(z,z_1)Q^{-\ell}(z_1,X_0(n-\ell))h^{n}_{k}(0,z_2)\R_t^{-1}(z_2,z)
	\\&=\sum_{z\in\zz}Q^{-\ell}(z,X_0(n-\ell))h^{n}_{k}(0,z)
	=(\Qt)^{-\ell}h^{n}_{k}(0,X_0(n-\ell)),
\end{align}
where in the first equality we have used the decay of $R_t$ and the fact that $2^{-x}h^n_k(0,x)$ is a polynomial together with the fact that the $z_1$ sum is finite to apply Fubini.
For $\ell\leq k$, we use the boundary condition $h^n_k(\ell,X_0(n-\ell))= \uno{\ell=k}$, which is both \eqref{bhe2} and \eqref{bhe3}, to get
\begin{equation}
(\Qt)^{-\ell}h^{n}_{k}(0,X_0(n-\ell))=h^{n}_{k}(\ell,X_0(n-\ell))=\uno{k=\ell}.
\end{equation}
For $\ell>k$, we use \eqref{bhe1}, \eqref{bhe2}, and $2^z\in\ker{(\Qt)^{-1}}$:
\begin{equation}
(\Qt)^{-\ell}h^{n}_{k}(0,X_0(n-\ell))=(\Qt)^{-(\ell-k-1)}(\Qt)^{-1}h^{n}_{k}(k,X_0(n-\ell))=0.
\end{equation} 

To finish the proof we need to show that $\Phi^n_k$ satisfies condition \ref{poly} of Thm.\,\ref{thm:BFPS}.
By \eqref{eq:kernelrt} we have $2^{-x}\Phi^n_k(x)=\sum_{y\geq0}\frac{e^{t}}{y!}(-t)^y\ts2^{-(x+y)}h^{n}_{k}(0,x+y)$.
It is enough to note then that, since $2^{-z}h^{n}_{k}(0,z)$ is a polynomial of degree at most $k$, this sum is absolutely convergent and is a polynomial of degree at most $k$ in $x$ as well.
\end{proof}

\subsection{Representation of the kernel as a hitting probability} 
\label{subsec:hitting}

Let
\begin{equation}
	G_{0,n}(z_1,z_2)=\sum_{k=0}^{n-1}Q^{n-k}(z_1,X_0(n-k))h^{n}_{k}(0,z_2)\label{eq:defG}
\end{equation}
with $h^n_k$ the solution of \eqref{bhe}.
Then from Thms. \ref{thm:BFPS} and \ref{thm:h_heat} and using \eqref{eq:Psi-n-chgd} we have
\begin{equation}\label{eq:Kt-decomp}
	K_t(n_i,\cdot;n_j,\cdot)=-Q^{n_j-n_i}\uno{n_i<n_j}+\R_t Q^{-n_i}G_{0,n_j}\R_t^{-1}.
\end{equation}
  
Below the ``curve'' $\big(X_0(n-\ell)\big)_{\ell=0,\dotsc,n-1}$, the functions $h^n_k(\ell,z)$ have an important physical interpretation.
$\Qt(x,y)$ are the transition probabilities of a random walk $B_m^*$ with Geom$[\frac12]$ jumps (strictly) to the right\footnote{We use the notation $B^*_m$ to distinguish it from the walk $B_m$ with transition matrix $Q$ which will appear shortly.\label{footnote:walks}}.
For $0\leq\ell\leq k\leq n-1$, define stopping times
\begin{equation}
\tau^{\ell,n}=\min\{m\in\{\ell,\dotsc,n-1\}\!:\,B^*_m> X_0({n-m})\},
\end{equation}
with the convention that $\min\emptyset=\infty$. 
Then for $z\leq X_0(n-\ell)$ we have \noeqref{eq:hnk-prob}
\begin{equation}\label{eq:hnk-prob}
h^n_k(\ell,z)=\pp_{B^*_{\ell-1}=z}\big(\tau^{\ell,n}=k\big).
\end{equation}
This can be proved by checking that, with this definition, $h^n_k(\ell,z)$ satisfies \eqref{bhe2} and \eqref{bhe3} while for $z\leq X_0(n-\ell-1)$ it also satisfies \eqref{bhe1} and it is given by $2^z$ times a polynomial in $z$ of degree at most $n-1$; the conclusion now follows from the fact, shown in the proof of Thm.\,\ref{thm:h_heat}, that $2^{-z}h^n_k(\ell,z)$ is a polynomial of degree at most $n-1$.

From the memoryless property of the geometric distribution we have for all $z\leq X_0(n)$ that \noeqref{eq:memoryless}
\begin{equation}\label{eq:memoryless}
\pp_{B^*_{-1}=z}\big(\tau^{0,n}=k,\,B^*_k=y\big)=2^{X_0(n-k)-y}\ts\pp_{B^*_{-1}=z}\big(\tau^{0,n}=k\big),
\end{equation}
and as a consequence we get, for $z_2\leq X_0(n)$,
\begin{equation}
\begin{split}
  G_{0,n}(z_1,z_2)&=\sum_{k=0}^{n-1}\pp_{B^*_{-1}=z_2}\big(\tau^{0,n}=k\big)(\Qt)^{n-k}(X_0(n-k),z_1)\\
  &=\sum_{k=0}^{n-1}\sum_{z>X_0(n-k)}\pp_{B^*_{-1}=z_2}\big(\tau^{0,n}=k,\,B^*_k=z\big)(\Qt)^{n-k-1}(z,z_1)\\
  &=\pp_{B^*_{-1}=z_2}\big(\tau^{0,n}<n,\,B^*_{n-1}=z_1\big),
\end{split}\label{eq:G-formula}
\end{equation}
which is \emph{the probability for the walk starting at $z_2$ at time $-1$ to end up at $z_1$ after $n$ steps, having hit the curve $\big(X_0(n-m)\big)_{m=0,\dotsc,n-1}$ in between}.

The next step is to obtain an expression along the lines of \eqref{eq:G-formula} which holds for all $z_2$, and not just $z_2\leq X_0(n)$.  
We begin by observing that for each fixed $y_1$ and $n\ge 1$, $2^{-y_2}Q^n(y_1,y_2)$ extends in $y_2$ to a  polynomial $2^{-y_2}\mQ^{(n)}(y_1,y_2)$ of degree $n-1$ with 
\begin{equation}\label{eq:QExt}
\mQ^{(n)}(y_1,y_2)= \frac{1}{2\pi \I} \oint_{\Gamma_0} \d v\,\frac{(1+v)^{y_1 - y_2 -1}}{2^{y_1-y_2} v^n}= \frac{(y_1 - y_2 - 1)_{n-1}}{2^{y_1 - y_2} (n-1)!},
\end{equation}
where $(x)_k = x (x-1) \cdots (x-k+1)$ for $k > 0$ and $(x)_0=1$ is the \emph{Pochhammer symbol}. 
Note that
\begin{equation}
\mQ^{(n)}(y_1,y_2)=Q^n(y_1,y_2)\qquad\text{for}\quad\,y_1-y_2\geq1.\label{eq:QQext}
\end{equation}
Using \eqref{eq:Qinv} and \eqref{eq:QExt}, we have 
\begin{equation}\label{eq:QinvQext}
Q^{-1}\mQ^{(n)}=\mQ^{(n)}Q^{-1}=\mQ^{(n-1)}\quad\text{for}\quad n>1,\qquad\text{but}\quad Q^{-1}\mQ^{(1)}=\mQ^{(1)}Q^{-1}=0.
\end{equation}
Note also that $\mQ^{(n)}\mQ^{(m)}$ is divergent, so the $\mQ^{(n)}$ are no longer a group like the $Q^n$.
Let 
\begin{equation}
\tau= \min\{ m\ge 0: B_m> X_0(m+1)\},\label{eq:deftau}
\end{equation}
where $B_m$ is now a random walk with transition matrix $Q$ (that is, $B_m$ has Geom$[\tfrac{1}{2}]$ jumps strictly in the negative direction).
Using this stopping time and the extension of $Q^m$ we obtain:

\begin{lem}\label{lem:G0n-formula}
For the kernel defined in \eqref{eq:defG} and all $z_1,z_2\in\zz$ we have
\begin{equation}\label{eq:G0n-formula34}
G_{0,n}(z_1,z_2) = \ee_{B_0=z_1}\!\left[ \mQ^{(n - \tau)}(B_{\tau}, z_2)\uno{\tau<n}\right].
\end{equation}
\end{lem}

\begin{proof}
For $z_2\leq X_0(n)$, \eqref{eq:G-formula} can be written as
\begin{multline}\label{eq:G0napp1}
G_{0,n}(z_1,z_2)=\pp_{B^*_{-1}=z_2}\big(\tau^{0,n}\leq n-1,\,B^*_{n-1}=z_1\big)=\pp_{B_{0}=z_1}\big(\tau\leq n-1,B_{n}=z_2\big)\\
 =\sum_{k=0}^{n-1}\sum_{z>X_0(k+1)}\pp_{B_{0}=z_1}\big(\tau=k,\,B_{k}=z\big)Q^{n-k}(z,z_2)
 =\ee_{B_0=z_1}\!\left[Q^{n-\tau}\big(B_{\tau},z_2\big)\uno{\tau<n}\right].
\end{multline}
We claim that the right hand side of \eqref{eq:G0n-formula34} equals $G_{0,n}(z_1,z_2)$ for all $z_2\leq X_0(n)$.
To see this, note from the last equality in \eqref{eq:G0napp1} that we only need to check $\P_{X_0(k+1)}\mQ^{(k+1)}\bar\P_{X_0(n)}=\P_{X_0(k+1)}{Q^{k+1}}\bar\P_{X_0(n)}$ for $k=0,\dotsc,n-1$ which, since $X_0(k+1)-X_0(n)\geq n-k-1$, follows from \eqref{eq:QQext}.

To complete the proof, recall that we showed that $2^{-z_2}h^n_k(0,z_2)$ is a polynomial of degree at most $k$ in $z_2$, so from \eqref{eq:defG} we have that $2^{-z_2}G_{0,n}(z_1,z_2)$ satisfies the same (for every fixed $z_1$).
It is straightforward to check that the right hand side of \eqref{eq:G0n-formula34} also satisfies this (because $\mQ^{(m)}(z_1,z_2)$ does), and thus since it coincides with $G_{0,n}(z_1,z_2)$ at infinitely many $z_2$'s, we deduce the equality in \eqref{eq:G0n-formula34}.  
\end{proof}

\subsection{Formulas for TASEP with right-finite initial data} 
\label{subsec:TASEPformulas}

Let\footnote{$(\SN_{-t,n})^*$ should be thought of as a version of $\SM_{-t,n}$ (analytic in $z_2-z_1$) made from the other pole in the contour integral in \eqref{eqFn}. In fact, if one changes variables $w\longmapsto1-w$ in \eqref{def:sm} and then changes $n$ to $-n$ then one gets \eqref{def:sn} (with $z_1$ and $z_2$ interchanged), except that the integration is along a loop enclosing only $1$ instead of only $0$. Our choice of taking an adjoint in the definition in \eqref{def:sm} is made just for later convenience.}, for $n\geq1$,
\begin{align}
 \SM_{-t,-n}(z_1,z_2) &= (e^{-\frac{t}2 \nabla^-}\!Q^{-n})^*(z_1,z_2)  = \frac{1}{2\pi\I} \oint_{\Gamma_0}\d w\, \frac{(1-w)^{n}}{2^{z_2-z_1} w^{n +1 + z_2 - z_1}}e^{t(w-1/2)},\label{def:sm}\\
 \SN_{-t,n} (z_1,z_2) &=  \mQ^{(n)}e^{\frac{t}2 \nabla^-} (z_1,z_2)= \frac{1}{2 \pi \I} \oint_{\Gamma_{0}} \d w\,\frac{(1-w)^{z_2-z_1 + n - 1}}{2^{z_1-z_2} w^{n}} e^{t(w-1/2)};\label{def:sn}
\end{align}
the contour integral formulas come from \eqref{eq:Qinv}, \eqref{eq:kernelrt}, and \eqref{eq:QExt}.
As before, $\Gamma_0$ is a simple counterclockwise loop around $0$ not enclosing $1$.
Define also, for $n\geq0$,
\begin{equation}\label{eq:sepi23}
{\SN}_{-t,n}^{\oepi(X_0)}(z_1,z_2) = \ee_{B_0=z_1}\!\left[ \SN_{-t,n - \tau}(B_{\tau}, z_2)\uno{\tau<n}\right].
\end{equation}
The superscript $\oepi(X_0)$ refers to the fact that $\tau$ (defined in \eqref{eq:deftau}) is the hitting time of the strict epigraph\footnote{The \emph{strict epigraph} of a discrete curve $\big(g(m)\big)_{m\geq0}$ is the set $\oepi(g)=\big\{(m,y)\!:m\geq0,\,y>g(m)\big\}$ (see also Sec.\,\ref{UC}).} of the curve $\big(X_0(k+1)\big)_{k=0,\dotsc,n-1}$ by the random walk $B_k$.

\begin{rem}  
$M_m= \SN_{-t,n - m}(B_m, z_2)$ is \emph{not} a martingale, because $Q\mQ^{(n)}$ is divergent.  
So one cannot apply the optional stopping theorem to evaluate \eqref{eq:sepi23}.
The right hand side of  \eqref{eq:sepi23} is only finite because the curve $\big(X_0(k+1)\big)_{k=0,\dotsc,n-1}$ cuts off the divergent sum.
\end{rem}

We are now in position to state the general solution of TASEP with right-finite initial data. 

\begin{thm}{\bf (TASEP formula for right-finite initial data)}\label{thm:tasepformulas}
\enspace Assume that the TASEP initial condition $X_0$ satisfies $X_0(j)=\infty$ for all $j\le 0$.
Then for any distinct positive integers $n_1,\dotsc,n_m$ and $t\geq0$, 
\begin{equation}\label{eq:extKernelProb}
  \pp\!\left(X_t(n_j)>a_j,~j=1,\dotsc,m\right)=\det\!\left(I-\bar\chi_aK^\TASEP_t\bar\chi_a\right)_{\ell^2(\{n_1,\dotsc,n_m\}\times\zz)},
\end{equation}
where $K^\TASEP_t$ is the operator on $\ell^2(\{n_1,\dotsc,n_m\}\times\zz)$ with kernel given by
\begin{equation}\label{eq:Kt-2}
K^\TASEP_t(n_i,\cdot;n_j,\cdot)=-Q^{n_j-n_i}\uno{n_i<n_j}+(\SM_{-t,-{n}_i})^*{\SN}_{-t,n_j}^{\oepi(X_0)}.
\end{equation}
The path integral version \eqref{eq:path-int-kernel-TASEPgem} (with $K^{(n)}_t=K^\TASEP_t(n,\cdot;n,\cdot)$) also holds.  
\end{thm}

\noindent
\begin{minipage}{\textwidth}
\begin{rem}\label{rem:afterTASEPformula}
\leavevmode
\begin{enumerate}[label=\arabic*.,itemsep=3pt,leftmargin=30pt]
\item By shifting the indices of the particles, the theorem allows us to write a formula for any right-finite initial data $X_0$ with $X_0(j)=\infty$ for $j\leq\ell$, any $\ell\in\zz$.
In fact, defining the shift operator
\begin{equation}\label{eq:shift}
\theta_\ell\tts g(u)=g(u+\ell),
\end{equation}
we have the trivial identity
\begin{equation}\label{eq:transInv}
  \pp_{X_0}\big(X_t(n_j)>a_j,~j=1,\dotsc,m\big)=\pp_{\theta_{\ell}X_0}\big(X_t(n_j-\ell)>a_j,~j=1,\dotsc,m\big).
\end{equation}
\item Note that, by definition, $\SN_{-t,n_j}^{\oepi(X_0)}(y,z)=\SN_{-t,n_j}(y,z)$  for $y>X_0(1)$, so \eqref{eq:Kt-2} can also be written as 
\begin{multline}
\mbox{}\hskip0.32in K^\TASEP_t(n_i,\cdot;n_j,\cdot)\\
=-Q^{n_j-n_i}\uno{n_i<n_j}+(\SM_{-t,-{n}_i})^*\P_{X_0(1)}\SN_{-t,{n}_j}+(\SM_{-t,-{n}_i})^*\bar{\P}_{X_0(1)}{\SN}_{-t,n_j}^{\oepi(X_0)}.\label{eq:Kt-2-alt}
\end{multline}
\end{enumerate}
\end{rem}
\end{minipage}

\begin{proof}
Consider first right-finite initial data. 
If $X_0(1)<\infty$ then we are in the setting of the above sections and formulas \eqref{eq:extKernelProb}--\eqref{eq:Kt-2} follow directly from the above definitions together with \eqref{eq:Kt-decomp} and Lem. \ref{lem:G0n-formula}.
If $X_0(i)=\infty$ for $i=1,\dotsc,\ell$ and $X_0(\ell+1)<\infty$ then it is enough to consider $n_j>\ell$ for $j=1,\dotsc,m$, and then from \eqref{eq:transInv} we have $\pp_{X_0}\tsm\big(X_t(n_j)>a_j,~j=1,\dotsc,m\big)=\det\!\big(I-\bar\chi_aK^{(\ell)}_t\bar\chi_a\big)_{\ell^2(\{n_1,\dotsc,n_m\}\times\zz)}$ with $K^{(\ell)}_t(n_i,\cdot;n_j,\cdot)=-Q^{n_j-n_i}\uno{n_i<n_j}+(\SM_{-t,-{n}_i+\ell})^*\SN_{-t,n_j-\ell}^{\oepi(\theta_{\ell} X_0)}$.
Using now that, for $n_i,n_j>\ell$, $(\SM_{-t,-{n}_i+\ell})^*=(\SM_{-t,-{n}_i})^*Q^{\ell}$ and 
\begin{equation}\label{eq:epishift}
Q^{\ell}\SN_{-t,n_j-\ell}^{\oepi(\theta_{\ell} X_0)}=\SN_{-t,n_j}^{\oepi(X_0)},
\end{equation}
which follows from the definition of $\SN_{-t,n}^{\oepi(X_0)}$, \eqref{eq:QinvQext}, and the fact that $\theta_{\ell}X_0(j)=\infty$ for $j=1,\dotsc,\ell$, we see that \eqref{eq:Kt-2} still holds in this case.

The path integral formula \eqref{eq:path-int-kernel-TASEPgem} is proved in Appx.\,\ref{app:proofPathIntTASEP}, and follows from a variant of \cite[Thm. 3.3]{bcr} proved in Appx.\,\ref{app:altBCR}.
\end{proof}

\begin{ex}{\bf (Step initial data)}\label{ex:step}
\enspace Consider TASEP with step initial data, i.e. $X_0(i) = -i$ for $i \geq 1$.
If we start the random walk in \eqref{eq:sepi23} from $B_0 = z_1$ below the curve, i.e. $z_1 \leq -1$, then the random walk clearly never hits the epigraph. 
Hence, $\bar\P_{X_0(1)}{\SN}_{-t,n}^{\oepi(X_0)} \equiv 0$ and the last term in \eqref{eq:Kt-2-alt} vanishes. 
For the second term in \eqref{eq:Kt-2-alt} we have, from \eqref{def:sm} and \eqref{def:sn},
\begin{align}
 (\SM_{-t,-{n}_i})^*\P_{X_0(1)}\SN_{-t,{n}_j}(z_1, z_2) = \frac{1}{(2\pi\I)^2} \oint_{\Gamma_0}\d w \oint_{\Gamma_0}\d v\, \frac{(1-w)^{n_i} (1-v)^{n_j + z_2}}{2^{z_1-z_2} w^{n_i + z_1 +1} v^{n_j}} \frac{e^{t(w + v-1)}}{1 - v - w}.
\end{align}
Using this in \eqref{eq:Kt-2-alt} yields exactly the formula derived previously in the literature (see e.g. \cite[Eq. 82]{dimers}), modulo the conjugation by $2^{z_2-z_1}$.
\end{ex}

\begin{ex}{\bf ($\bm{2}$-periodic initial data)}
\enspace We are interested now in TASEP with 2-periodic initial data $X_0(i)=2i$, $i\in\zz$ (we consider more general periods in the next example).
To obtain a formula for the kernel in this case we will approximate by considering first the finite periodic initial data $X_0(i) = 2(N-i)$ for $i=1,\dotsc,2N$.
For simplicity we will compute only $K^{(n)}_t=K^\TASEP(n,\cdot;n,\cdot)$. 

\noindent We start by computing ${\SN}_{-t,n}^{\oepi(X_0)}$.
Observe that for $\lambda>-\log(2)$ we have that $\big(e^{\lambda B_m-m\varphi(\lambda)}\big)_{m\geq0}$, with $\varphi(\lambda)=-\log(2e^\lambda-1)$ the logarithm of the moment generating function of a negative Geom$[\frac12]$ random variable, is a martingale.
Thus if $z\leq 2(N-1)$, $\ee_{B_0=z}[e^{\lambda B_\tau-\tau\varphi(\lambda)}]=e^{\lambda z}$.
But it is easy to see from the definition of $X_0$ that if the walk starts below the curve then $B_\tau$ is necessarily $2(N-(\tau+1))+1$, so we have $\ee_{B_0=z}[e^{-2 \lambda \tau}(2e^{\lambda}-1)^\tau]=e^{(z-2N+1)\lambda}$.
Introducing a new variable $\eta\in(0,1)$ through $\lambda=\log(\eta^{-1}(1+\sqrt{1-\eta}))$ yields $\ee_{B_0=z}[\eta^\tau]=\left(\eta^{-1}(1+\sqrt{1-\eta})\right)^{z-2N+1}$.
As a consequence we obtain \mbox{$\pp_{B_0=z}(\tau=k)=\lim_{\eta\to0}\frac1{k!}\frac{\d^k}{\d\eta^k}\big(\frac{1+\sqrt{1-\eta}}{\eta}\big)^{z-2N+1}$}, and then using Cauchy's integral formula we can compute ${\SN}_{-t,n}^{\oepi(X_0)}(z_1,z_2)$ from \eqref{def:sn} and \eqref{eq:sepi23} as
\[\textstyle\lim_{\eta\to0}\frac1{(2\pi\I)^2}\oint_{\gamma_r} \d w\oint_{\eta+\gamma_{r}}\d\xi\,\sum_{k=0}^{n-1}\left(\frac{1+\sqrt{1-\xi}}{\xi}\right)^{z_1-2N+1}\!\!\frac1{(\xi-\eta)^{k+1}}\frac{(1-w)^{z_2-2N+n+k}}{2^{2(N-k)-1-z_2}w^{n-k}}e^{t(w-\frac12)},\]
where $\gamma_r$ is a circle of radius $r$ centered at the origin and we take $\eta<r<3/4$.
The $\eta\to0$ limit is now straightforward to compute, and since the resulting integrand is analytic in $\xi$ for $k<0$, we may extend the sum to $k=-\infty$ and then compute the sum (using that $|4\xi^{-1}w(1-w)|>1$ for our choice of $r$) to get
$\frac1{(2\pi\I)^2}\oint_{\gamma_r} \d w\oint_{\gamma_{r}}\d\xi\left(\frac{1+\sqrt{1-\xi}}{\xi}\right)^{z_1-2N+1}\!\!\frac{(1-w)^{z_2-2(N-n)}}{2^{2(N-n)-1-z_2}\xi^n}\frac1{4w(1-w)-\xi}\ts e^{t(w-\frac12)}$.
Now we introduce the change of variables $\xi=4v(1-v)$, which is locally one-to-one near $v=1$.
For small $r$, if $v$ lies in $1+\gamma_r$ then $4v(1-v)$ lies approximately in $\gamma_{4r}$ so we may adjust the contours to get that the last integral equals $\frac1{(2\pi\I)^2}\oint_{\gamma_r} \d w\oint_{1+\gamma_{r'}}\d v\frac{(1-w)^{z_2-2(N-n)}}{2^{y-z_2}v^n(1-v)^{z_1-2N+n+1}}\frac{1-2v}{w(1-w)-v(1-v)}\ts e^{t(w-\frac12)}$ with $r'\sim r/4$.
From this and \eqref{def:sm} we may compute the product $(\SM_{t,-{n}})^*\bar\P_{2(N-1)}\SN_{t,n}^{\oepi(X_0),\eta}(z_1,z_2)$, which equals
\[\textstyle\frac1{(2\pi\I)^3}\oint_{\gamma_r}\d u\oint_{\gamma_r} \d w\oint_{1+\gamma_{r'}}\d v\ts \frac{(1-w)^{z_2-2(N-n)}(1-u)^n}{2^{z_1-z_2}v^n(1-v)^{n-1}u^{z_1-2N+n+2}}\frac{1-2v}{(1-u-v)[w(1-w)-v(1-v)]}\ts e^{t(w+u-1)}.\]
Since $1-v$ lies inside $\gamma_r$, the $w$ integral has a pole at $w=1-v$, and computing the residue yields
\begin{equation}\label{eq:periodicPiece1}
\textstyle\frac1{(2\pi\I)^2}\oint_{\gamma_r}\d u\oint_{1+\gamma_{r'}}\d v\ts \frac{v^{z_2-2N+n}(1-u)^n}{2^{z_1-z_2}(1-v)^{n-1}u^{z_1-2N+n+2}}\frac{1}{1-u-v}\ts e^{t(u-v)}.
\end{equation}
On the other hand, a simpler computation (as in the previous example) shows that the other term making up $K_t^{(n)}(z_1,z_2)=K^\TASEP_t(n,z_1;n,z_2)$ in \eqref{eq:Kt-2-alt}, namely $(\SM_{t,-{n}})^*\P_{2(N-1)}\SN_{t,n}(z_1,z_2)$, equals
\begin{equation}\label{eq:periodicPiece2}
\textstyle\frac1{(2\pi\I)^2}\oint_{\gamma_r}\d u\oint_{\gamma_{r''}} \d v\ts \frac{(1-v)^{z_2-2(N-n)+1}(1-u)^n}{2^{z_1-z_2}u^{z_1-2N+n+2}v^n}\frac{1}{1-u-v}\ts e^{t(u+v-1)},
\end{equation}
where we take $r''<r$.
Hence $K_t^{(n)}(z_1,z_2)$ in this case is given as the sum of the last two integrals.

\noindent In order to obtain the kernel which yields the distribution of the $m$-th particle for the full $2$-periodic initial condition, $K^{2\uptext{-prd},(m)}_t$, we proceed as in \cite{borFerPrahSasam}, using the last formula and focusing on particles which start at a fixed distance from the origin, that is $n=N+m$ with $m$ fixed (which corresponds to the particle that started at $-2m$), and taking $N\to\infty$.
To this end, for fixed $z_1$ take $N\geq z_1+m+2$, so that $u=0$ is not a pole in both \eqref{eq:periodicPiece1} and \eqref{eq:periodicPiece2}.
We see now that \eqref{eq:periodicPiece2} vanishes, because the $u$ integrand is analytic given our choice of contours.
On the other hand, for \eqref{eq:periodicPiece1} we have that $1-v$ lies inside $\gamma_r$ so the $u$ integral has a pole at $u=1-v$, and computing the residue yields
\begin{equation}
\textstyle K^{2\uptext{-prd},(m)}_t(z_1,z_2)=-\frac1{2\pi\I}\oint_{1+\gamma_{r'}}\d v\ts \frac{v^{z_2+2m}}{2^{z_1-z_2}(1-v)^{z_1+2m+1}}\ts e^{t(1-2v)}.\label{eq:Kn2per}
\end{equation}
This is exactly the kernel derived (modulo the conjugation $2^{z_2-z_1}$ and after a simple change of variables) in \cite[Thm. 2.2]{borFerPrahSasam}.
\end{ex}

\begin{ex}{\bf ($\bm{\ell}$-periodic initial data)}\label{example210}
\enspace Now we turn to TASEP with $\ell$-periodic initial data, given by $X_0(i)=-\ell\tts i$, $i\in\zz$, with $\ell\geq 2$.
In contrast to the last example, the value of $B_\tau$ is not fixed as a function of $\tau$ if $\ell > 2$, and thus computing $\SN^{\epi(X_0)}_{t,-n}$ becomes more complicated.
But the hitting probabilities for $B^*$ can be computed in a way similar to the last example, so in this case it is simpler to compute the biorthogonal functions $\Phi^n_k$ using \eqref{eq:defPhink} and \eqref{eq:hnk-prob} and then obtain the kernel $K^{(n)}_t=K^\TASEP_t(n,\cdot;n,\cdot)$ directly from \eqref{eq:Kt}.
We do this next.

\noindent As in the last example, we consider first the truncated initial data $X_0(i) = \ell(N-i)$, $i=1,\dotsc,2N$.
For fixed $0\leq k<n\leq2N$ and $z \leq X_0(1)$ we want to compute $h^n_k(z)=\pp_{B^*_{-1}=z}(\tau^*=k)$, where $\tau^*$ is the hitting time of the strict epigraph of $\big(X_0(n-m)\big)_{m=0,\dotsc,n-1}$.
Proceeding as above, for $\lambda < \log(2)$ and $\varphi^*(\lambda)=-\log(2e^{-\lambda}-1)$ we have $\ee_{B^*_{-1}=z}[e^{\lambda B^*_{\tau^*}-\tau^*\varphi^*(\lambda)}]=e^{\lambda z + \varphi^*(\lambda)}$, giving
\[\textstyle e^{\lambda z + \varphi^*(\lambda)} = \sum_{k \geq 0} \sum_{m \geq 1} \pp_{B^*_{-1}=z} \big(\tau^* = k, B^*_k = X_0(n - k) + m \big) e^{\lambda (X_0(n-k) + m)-k\varphi^*(\lambda)}.\]
Using the memoryless property of the geometric distribution we may rewrite the above probability as $\pp_{B^*_{-1}=z}\big(\tau^* = k, B^*_k = X_0(n - k) + 1 \big) 2^{1 - m}$.
The sum over $m$ is then just $\sum_{m \geq 1}(e^\lambda/2)^m = e^{\varphi^*(\lambda)}$, and thus the right hand side equals $2\ts\ee_{B^*_{-1}=z} \big[e^{\lambda B^*_{\tau^*} - \tau^* \varphi^*(\lambda)} \uno{B^*_{\tau^*} = X_0(n - \tau^*) + 1} \big] e^{\varphi^*(\lambda)-\lambda}$, which leads to $2\ts\ee_{B^*_{-1}=z} \big[e^{(\lambda\ell - \varphi^*(\lambda))\tau^*} \uno{B^*_{\tau^*} = X_0(n - \tau^*) + 1} \big] = e^{\lambda [z + \ell(n-N)]}$.
Setting $v = 1 - e^{\lambda}/2 > 0$ we get $2\ts\ee_{B^*_{-1}=z} \big[ (2^\ell (1-v)^{\ell-1} v)^{\tau^*} \uno{B^*_{\tau^*} = X_0(n - \tau^*) + 1} \big] = (2(1-v))^{z + \ell(n-N)}$.
The function $v \longmapsto p(v) \coloneqq 2^\ell (1-v)^{\ell-1} v$ is locally one-to-one near $0$, so
\begin{equation}
 \textstyle2\ts\pp_{B^*_{-1}=z} \big(\tau^* = k, B^*_{\tau^*} = X_0(n - \tau^*) + 1 \big) = \lim_{v \downarrow 0} \frac{1}{k!} \frac{\d^k}{\d p(v)^k} \left(2(1-v)\right)^{z + \ell(n-N)}.
\end{equation}
Using again the memorylessness of the walk, the left hand side equals $\pp_{B^*_{-1}=z}\big(\tau^* = k\big)$, while, by Cauchy's formula, the right hand side equals
\[\textstyle\inv{2\pi\I}\oint_{\Gamma_{0}} \d v\frac{1}{p(v)^{k+1}} p'(v) (2(1-v))^{z + \ell(n-N)}
=\frac{1}{2 \pi \I} \oint_{\Gamma_{0}} \d v\, \frac{(1-v)^{z + \ell(n-N) - 1}}{2^{\ell(N-n+k)-z} ((1-v)^{d-1} v)^{k}}\frac{(1 - \ell v)}{v},\]
where $\Gamma_0$ goes around $0$ but not $1$.
In principle this is only valid for $z\leq X_0(1)$, but the right hand side is analytic in $z$ so we actually get a formula for $h^n_k(0,z)$ for all $z\in\zz$.
Using \eqref{eq:kernelrt} and \eqref{eq:defPhink} we get
\begin{equation}
\textstyle\Phi^n_k(z) = \frac{1}{2 \pi \I} \oint_{\Gamma_{0}} \d v\, \frac{(1-v)^{z + \ell (n-N) - 1}}{2^{\ell (N-n+k)-z}((1-v)^{\ell -1} v)^{k}}\frac{1-\ell v}{v} e^{tv}.
\end{equation}
These functions extend trivially to 0 for $k<0$ so we may now perform the summation in \eqref{eq:Kt} over $k \geq 1$ (using the explicit formula \eqref{eq:defPsi} for $\Psi^n_k$) to get the kernel for truncated $\ell$-periodic initial data
\begin{equation}
\textstyle K^{(n)}_{t,N}(z_1,z_2) = \frac{1}{(2\pi\I)^2} \oint_{\Gamma_{0}} \d w \oint_{\Gamma_{0}'} \d v\, \frac{(1-v)^{z_2 - \ell (N-1) + n - 2} (1-w)^{n}}{2^{z_1 - z_2} v^n w^{z_1 -\ell N + n+1}} \frac{(1 - \ell v) e^{t(v + w - 1)}}{(1-w) w^{\ell -1} - (1-v)^{\ell -1} v},
\end{equation}
where the contours are so that $|(1-v)^{\ell -1}v|<|(1-w)w^{\ell -1}|$.

\noindent Finally, and as in the $2$-periodic case, we set $n=N+m$ in the last kernel, which gives
\begin{equation}
\textstyle\frac{1}{(2\pi\I)^2} \oint_{\Gamma_{-1}} \d w \oint_{\Gamma_0} \d v\, \frac{(1+v)^{z_2 - (\ell -1)N + \ell + m - 2} w^{N+m}}{2^{z_1 - z_2} v^{N+m} (1+w)^{z_1 -(\ell -1)N +m+1}} \frac{(1+\ell v) e^{t(w-v)}}{(1+w)^{\ell -1}w-(1+v)^{\ell -1}v}
\end{equation}
(where we have changed variables $v\longmapsto-v$, $w\longmapsto1+w$), and then take $N\to\infty$ to get the kernel for the full $\ell$-periodic initial condition.
For fixed $z_1$ and large enough $N$ the $w$ integral has no pole at $-1$.
Let $w_1(v),\dotsc w_{\ell-1}(v)$ be the $\ell-1$ solutions of $(1+w)^{\ell-1}w=(1+v)^{\ell-1}v$ other than $w=v$.
One can check that all these $\ell-1$ roots are distinct and lie inside the $w$ contour, while $w=v$ lies outside of it.
The full $\ell$-periodic kernel then evaluates to a sum over the residues of these $\ell-1$ simple poles, and after simplification (using the equation satisfied by the $w_\ell(v)$'s) we get
\begin{align}
\textstyle K^{\ell\uptext{-prd},(m)}_{t}(z_1,z_2) &= \textstyle\frac{1}{2\pi\I}\oint_{\Gamma_0} \d v\,\sum_{j=1}^{\ell-1}\frac{1+\ell v}{1+\ell w_j(v)}\frac{(1+v)^{z_2 + \ell + m - 2} w_j(v)^{m}}{2^{z_1 - z_2} v^{m} (1+w_j(v))^{z_1 + \ell + m - 1}}e^{t(w_j(v)-v)}\\
\textstyle&=\textstyle\frac{1}{(2\pi\I)^2} \oint_{\Gamma_{-1}} \d w \oint_{\Gamma_0} \d v\, \frac{(1+v)^{z_2 + \ell + m - 2} w^{m}}{2^{z_1 - z_2} v^{m} (1+w)^{z_1+m+1}} \frac{1+\ell v}{(1+w)^{\ell-1}w-(1+v)^{\ell-1}v}e^{t(w-v)}.
\end{align}
These formulas are very similar to \cite[Eqs. 2.3, 4.11]{bfp} (which are for discrete time TASEP).
In the case $\ell=2$ we have $w_1(v)=-1-v$ and we recover \eqref{eq:Kn2per} after a simple change of variables.
\end{ex}

\subsection{Integrability}\label{sec:integrability}

There are many notions of classical integrable systems:  Liouville integrability, algebraic integrability, etc.
Quantum integrability usually is used to mean that a quantum mechanical model possesses an infinite number of conserved quantities. 
Another notion of integrable system is simply that one has a representation under which the flow is linearized.
Theorem \ref{thm:tasepformulas} presents TASEP with right-finite initial data\footnote{One also has a formula for two-sided initial data, but because of the analytic extension it is cumbersome, and the proof is quite lengthy; moreover, it is not clear to us yet that the formula can be used for asymptotics. We leave it to a future paper.} as a new type of \emph{stochastic integrable system}, the dynamics being trivialized at the level of kernels, which 
satisfy the Lax equation\footnote{$K^\TASEP_t$ acts on the Hilbert space $\ell^2(\{n_1,\dotsc,n_m\}\times\zz$), and can be identified with an operator-valued $n\times n$ matrix acting on $\bigoplus_{n\in\{n_1,\dotsc,n_m\}}\ell^2(\zz)$. Under this identification, $\nabla^-$ is identified with the diagonal matrix with the specified entries along the whole diagonal.\label{foot:directsum}},
\begin{equation}\label{eq:comint2}
\partial_t K^\TASEP_t = \tfrac12[K^\TASEP_t, \nabla^-].
\end{equation}
The $m$-point distributions at time $t$ are obtained from $K^\TASEP_t $ by projecting down via the Fredholm determinant, and the full space time field is recovered from these transition probabilities using the Markov property.

TASEP has long been known to be solvable, by the coordinate Bethe ansatz, resulting in Sch\"utz's formula \eqref{eqGreen}.
One also has the algebraic Bethe ansatz in which the eigenfunctions are computable \cite{prolhac-spectrum}.
However, the resulting formulas do not directly integrate the dynamics\ts---\ts\tts{}i.e. solve the problem starting from generic initial data\ts---\ts\tts{}in a \emph{useful} way.
One might refer to them as \emph{exact solvability} versus the \emph{stochastic integrability} given in \eqref{eq:comint2}.

\section{1:2:3 scaling limit}\label{sec:123}

For each $\ep>0$ the 1:2:3 rescaled TASEP height function is 
\begin{equation}\label{eq:hep}
	\fh^{\ep}(\ft,\fx) = \ep^{1/2}\!\left[h_{2\ep^{-3/2}\ft}(2\ep^{-1}\fx) + \ep^{-3/2}\ft\right].
\end{equation}

\begin{rem}
The KPZ fixed point has \emph{one} free parameter\footnote{\label{jaragonfoot} It was recently proposed that the KPZ fixed point is given by $\partial_t h =  \lambda(\partial_x h)^2 -\nu(-\partial_x^2)^{3/2} h + \nu^{1/2}(-\partial_x^2)^{3/4}\xi$, $\nu>0$, the evidence being that formally it is invariant under the 1:2:3 KPZ scaling \eqref{123} and it preserves Brownian motion.  
Besides the non-physical non-locality, and the inherent difficulty of making sense of this equation,  one can see that it is not correct because it has \emph{two} free parameters  instead of one.  
Presumably, it converges to the KPZ fixed point in the limit $\nu\searrow 0$.
On the other hand, the model has critical scaling, so it is also plausible that if one introduces a cutoff (say, smooth the noise) and then take a limit, the result has $\nu=0$, and possibly even a renormalized $\lambda$.
So it is possible that, in a rather uninformative sense, the conjecture could still be true. 
}, corresponding to $\lambda$ in \eqref{KPZ}.
Our choice of the height function in TASEP moving downwards corresponds to $\lambda>0$\footnote{To get some intuition, check that $-\fx^2/\ft$ is a solution of 
$\partial_t\fh = \tfrac14(\partial_x\fh)^2$; it 
corresponds to step initial data.  The bulk downward movement of the TASEP height function has been compensated by the huge shift upward in \eqref{eq:hep}.}.
The scaling of space and time by the factor $2$ in \eqref{eq:hep} corresponds to the choice $|\lambda| =1/4$.

\noindent Note also that, for fixed $\ft$, the TASEP height function in \eqref{eq:hep} is being rescaled diffusively in space.
In particular, this fixes our study of the scaling limit to perturbations of density $1/2$.
We could perturb off any density $\rho\in (0,1)$ without extra difficulty by observing the system in an appropriate moving frame, but in order to avoid heavier notation we do not pursue it here.
\end{rem}

Assume that we have initial data  $X_0^\ep$ chosen to depend on $\ep$ in such a way that
\begin{equation}\label{xplim}
\fh_0=\lim_{\ep\to 0} \fh^{\ep}(0,\cdot)
\end{equation}
in distribution, in the $\UC$ topology described below.
We will also choose the frame of reference
\begin{equation}\label{eq:x0conv} 
  \xx_0^{-1}(-1)=1,
\end{equation}  
i.e. the particle labeled $1$ is initially the rightmost in $\zz_{<0}$.
Because the $\xx^\ep_0(k)$ are in reverse order, and because of \eqref{eq:x0conv} and the inversion \eqref{defofh}, \eqref{xplim} is equivalent to
\begin{equation}\label{x0limit}
\ep^{1/2}\left(\xx^\ep_0(\ep^{-1}\fx)+2\ep^{-1}\fx-1\right) \xrightarrow[\ep\to0]{} -\fh_0(-\fx)
\end{equation}
in distribution, in $\UC$, where the left hand side is interpreted as a linear interpolation to make it a continuous function of $\fx\in \rr$.

For fixed $\ft>0$, we will now show that the limit 
\begin{equation}\label{eq:height-cvgce}
\fh(\ft,\fx;\fh_0)=\lim_{\ep\to0}\fh^{\ep}(\ft,\fx)
\end{equation}
also exists in distribution, in $\UC$.
In Sec. \ref{sec:tightandmarkov} we will prove the
Markov property, which gives us, in principle, the multi-time and space distributions of the entire field.
We take \eqref{eq:height-cvgce} essentially as our \emph{definition} of the KPZ fixed point $\fh(\ft,\fx; \fh_0)$.  We will often omit $\fh_0$ from the notation when it is clear from the context.

\subsection{State space and topology}\label{UC}

The state space in which we will always work, and where \eqref{xplim}, \eqref{x0limit} will be assumed to hold and \eqref{eq:height-cvgce} will be proved, in distribution, will be\footnote{The bound $\fh(\fx)\le \gga + \g |\fx|$ is not as general as possible.  With work, one can extend to the class $\fh(\fx)\le \gga+\g_0|\fx|^2$ up to time $\ft=\g_0^{-1}$.  At that time initial data such as $\fh_0(\fx)= \gga+\g_0|\fx|^2 $ will actually have an explosion.\label{foot:linearbd}} 
\begin{multline}
\UC=\text{upper semicontinuous fns. $\fh\!:\rr \to [-\infty,\infty)$ with $\fh(\fx)\le \gga + \g |\fx|$  for some $\gga,\g<\infty$}\\
\text{and $\fh(\fx) >-\infty$ for some $\fx$}
\end{multline}
with the topology of local $\UC$ convergence, which we now describe.

Recall $\fh$ is upper semicontinuous ($\UC$) if and only if its \emph{hypograph} $\hypo(\fh) = \{(\fx,\fy): \fy\le \fh(\fx)\}$ is closed in $[-\infty,\infty)\times \rr$. 
We endow $[-\infty,\infty)$ with the distance\footnote{This allows continuity at time $0$ for initial data which takes values $-\infty$, such as half-flat (see Sec.\,\ref{sec:airyprocess}).} $d_{[-\infty,\infty)}(\fy_1,\fy_2) = |e^{\fy_1} - e^{\fy_2}|$.
Given $\fh_1,\fh_2\in\UC$ and $M>0$ we say that the hypographs $\mathfrak{H}_1^M$ and $\mathfrak{H}_2^M$ of $\fh_1$ and $\fh_2$ restricted to $[-M,M]$ are \emph{$\delta$-close} if
$
B_{\delta}(\mathfrak{H}^M_1)\subseteq\mathfrak{H}_2^M$ and $B_{\delta}(\mathfrak{H}_2^M)\subseteq\mathfrak{H}_1^M$, where $B_\delta (\mathfrak{H}) := \cup_{(\ft,\fx)\in\mathfrak{H}}B_{\delta}((\ft,\fx))$, 
$B_\delta((\ft,\fx))$ being the ball of radius $\delta$ around $(\ft,\fx)$, i.e. we use the \emph{Hausdorff distance} on the restricted hypographs.
We say then that $\big(\fh_\ep\big)_\ep\subseteq\UC$ \emph{converges locally in $\UC$} to $\fh\in\UC$ if there is a $\g>0$ such that $\fh_\ep(\fx) \le \gga + \g |\fx|$ for all $\ep>0$ and for every $M\geq1$ there is a $\delta=\delta(\ep,M)>0$ going to 0 as $\ep\to0$ such that the hypographs $\mathfrak{H}_{\ep}^M$ and $\mathfrak{H}^{M}$ of $\fh_\ep$ and $\fh$ restricted to $[-M,M]$ are $\delta$-close.

Another characterization is that
$\fh^n\to \fh$ locally in $\UC$ if and only if for each $\fx$, $\limsup_{\fx^n\to\fx} \fh^n(\fx^n) \le \fh(\fx)$ and $\exists\,\fx^n\to\fx$ with $\liminf \fh^n(\fx^n)\ge  \fh(\fx)$. 

We will also use the space $\LC=\big\{\fg\!: -\fg\in\UC\!\big\}$ (made of lower semicontinuous functions), the topology now being defined in terms of \emph{epigraphs}, $\epi(\fg) = \{(\fx,\fy): \fy\ge \fg(\fx)\}$.

The Borel sets of $\UC$ will be denoted $\mathcal{B}(\UC)$.
It is fairly easy to see that  $\UC$ is a Polish space and the subspace $\mathcal{B}_0(\UC)\subseteq\mathcal{B}(\UC)$  of sets $A$ of the form $A=\{\fh\in \UC, \fh(\fx_i)\le \fa_i, i=1,\ldots,n\}$ form a generating family for the $\sigma$-algebra $\mathcal{B}(\UC)$, as does the subspace $\mathcal{B}_1(\UC)\subseteq\mathcal{B}(\UC)$  of sets of the form $A_{\mathfrak{g}} = \{ \fh\in \UC:  \fh(\fx) \le \fg(\fx),\,\fx\in\rr\}$, $\fg\in\LC$.

The  $\UC$ topology is very natural for interface growth, incorporating the inherent lateral growth mechanism and the $\fh\to-\fh$ asymmetry.

One could alternatively take ${\mathscr{C}}_{\g}= \{$continuous functions in $\UC$ satisfying $\fh(\fx)\le \gga + \g |\fx|\}$ and use $\mathscr{C}=\cup_{\g>0}\mathscr{C}_{\g}$ as state space; the topology on $\UC$, when restricted to $\mathscr C$, is the topology of uniform convergence on compact sets.
Or we could consider the local H\"older spaces defined by the family of semi-norms
\begin{equation}\label{eq:defHolderNorm}
\| \fh \|_{\beta, [-M,M]} = \sup_{\fx_1\neq \fx_2 \in [-M,M]} 
{ |\fh(\fx_2)-\fh(\fx_1)|}/{|\fx_2-\fx_1|^\beta},
\end{equation}
${\mathscr{C}}^\beta_{\g}=\{ \fh\in {\mathscr{C}}_{\g}$ with $\| \fh \|_{\beta, [-M,M]} <\infty$ for each $M=1,2,\ldots\}$.
These would suffice for any $\ft>0$, but many natural initial data are not in these spaces.  
For example the $\UC$ function $\mathfrak{d}_\fx(\fx) = 0$, $\mathfrak{d}_\fx(\fy) = -\infty$ for $\fy\neq \fx$, known as a {\em narrow wedge at $\fx$}, plays a role in the theory somewhat analogous to Dirac's delta function.

For our purposes, the following fact about $\UC$ is crucial.  Recall $\ftau_\ep\longrightarrow \ftau$ in distribution if and only if $\limsup \ftau_\ep \le \ftau$ and $\liminf \ftau_\ep \ge \ftau$ in distribution, i.e. $\limsup \pp( \ftau_\ep\ge r) \le \pp(\ftau\ge r)$ and $\liminf \pp( \ftau_\ep> r)
\ge \pp(\ftau> r)$. 
The probability spaces on which they are defined need have nothing to do with each other, but it can be conceptually easier
to construct them all on the same probability space, in which case the definitions are just the standard ones.

\begin{prop}\label{keyconvoftau} 
Suppose $\fg_\ep\longrightarrow \fg$ locally in $\LC$. Let $\fB (\fx) $, $\fx\ge 0$ be a Brownian motion starting at $z<\fg(0)$, and let $\fB_\ep(\fx)$ be stochastic processes with $\fB_\ep\longrightarrow \fB$ in $\UC$, in distribution.  Let $\ftau^\ep=\inf\{ \fx\ge 0: \fB_\ep (\fx) \ge \fg_\ep(\fx)\}$ and $\ftau=\inf\{ \fx\ge 0:
\fB (\fx) \ge \fg(\fx)\}$ be the first hitting times of $\epi(\fg_\ep)$ and $\epi(\fg)$.  Then $\ftau^\ep\longrightarrow \ftau$ in distribution. Furthermore, the convergence is uniform over $\fg$ in sets of bounded H\"older $\beta$-norm, $\beta\in (0,1/2)$.
\end{prop}

\begin{proof}
If there exists a subsequence $\ftau^{\ep_n}\longrightarrow \fx$ then $0\le \limsup \fB_{\ep_n}(\ftau^{\ep_n}) - \fg_{\ep_n}(\ftau^{\ep_n}) \le \fB(\fx) -\fg(\fx)$ because they are converging in $\UC$, so $\fx\ge \ftau$, and thus $\liminf \ftau^\ep \ge \ftau$.
For the other direction, let
$f_\delta(0)=0$, $f_\delta(\fx) = \delta$ for $\fx\ge \delta$ and linear in between.  Let $\bar\ftau_\delta=\inf\{ \fx\ge 0:
\fB (\fx) \ge \fg(\fx)+f_\delta(\fx)\}$.   By the Cameron-Martin formula $\bar\ftau_\delta\searrow \ftau$ in distribution: $\pp( \bar\ftau_\delta\in I) = \ee[\exp\{ B(\delta)-B(0) -\delta/2\} \uno{ \ftau\in I}]\longrightarrow \pp(\ftau\in I)$ since the integrands are uniformly integrable.
Since $\fB_\ep-\fg_\ep\longrightarrow \fB-\fg$ in $\UC$, there are $\fx_\ep\longrightarrow \bar\ftau_\delta$ with $ \liminf \fB_\ep(\fx_\ep) -\fg_\ep(\fx_\ep) >0$.  So   $\ftau^\ep <\bar\ftau_\delta+\delta$ for $\ep$ sufficiently small.  This proves $\limsup \ftau^\ep
\le \ftau$ in distribution. 

To prove the uniformity, for any $\delta>0$, and restricting to $[0,M]$, we have $\hypo(\fB_\ep-\fg_\ep) \subset B_\delta(\hypo( \fB-\fg))$ for sufficiently small $\ep>0$.  There is a $\gamma>0$ depending only on the H\"older $\beta$-norm of $ \fB-\fg $ 
and going to $0$ with $\delta$ such that $ B_\delta(\hypo( \fB-\fg)) \subset \hypo(\fB-\fg + \gamma)$. Hence $\ftau^\ep\ge \ftau_{z+\gamma}$ in distribution, where the subscript $z+\gamma$ indicates that the hitting time $\ftau_{z+\gamma}$ is for the Brownian  starting with $\fB(0)=z+\gamma$. 
In the other direction, for any $\delta>0$ there exists $\gamma>0$ depending only on the H\"older $\beta$-norm of  $ \fB-\fg $ such that $\fB(\fx) - \fg(\fx) >\delta$ for any $|\fx-\ftau_{z-\gamma}|<\delta$, and there exist $|\fx_\delta -\ftau_{z-\gamma}|<\delta$ for which $(\fB_\ep-\fg_\ep)(\fx_\delta) \ge  (\fB-\fg)(\ftau_{z-\gamma})-\delta$.  Hence
$\ftau^\ep\le \fx_\delta <\ftau_{z-\gamma} +\delta$. 
Now let $(\overline\fB,\underline\fB)$ be a pair of coalescing Brownian motions starting at $(z+\gamma,z-\gamma)$ defined by letting $\overline\fB(\fy)=2z-\underline\fB(\fy)$ until the first time $\fsigma_z$ they meet, and $\overline\fB(\fy)=\underline\fB(\fy)$ for $\fy>\fsigma_z$.
We have $\pp(\ftau_{z+\gamma}\le T)-\pp(\ftau_{z-\gamma}\le T) = \pp( \overline\fB~\text{hit~but ~} \underline\fB~\text{didn't}) \le \pp(\fsigma_z>\ftau_{z+\gamma} )$.  Recall $z<\fg(\fx)$.  We also have $\fg(\fx+\fy) \ge \fg(\fx) -C\tts\fy^\beta$ for $0\le\fy\le T$ by the uniform H\"older bound, so if we let $\bm{\nu}_{z+\gamma}$ be the hitting time of $\fg(\fx) - C\tts\fy^\beta$ by $\overline\fB$, we get  $\ftau_{z+\gamma}\ge \bm{\nu}_{z+\gamma}$.
Hence $\pp(\ftau_{z+\gamma}\le T)-\pp(\ftau_{z-\gamma}\le T)\le\pp(\fsigma_z>\bm{\nu}_{z+\gamma} )\longrightarrow0$, as $\gamma\to 0$  with a rate depending only on $C,\gamma$ and $\beta$.
\end{proof}

\subsection{Approximation setup}\label{sec:finitespeed}

For any $\fh_0\in \UC$, we can find initial data $X^\ep_0$  so that \eqref{x0limit} holds in the $\LC$
topology.  
This is easy to see, because any $\fh_0\in \UC$ is the limit of functions which are finite at finitely many points, and  $-\infty$ otherwise.
In turn, such functions can be approximated by initial data $X^\ep_0$ where the particles are densely packed in  blocks.
Note there is a mild abuse here as the left hand side of \eqref{x0limit} is a function on $\ep\zz$.  We can always extend it to $\rr$ in a simple way,
say taking it piecewise constant on $(\ep n, \ep( n+1))$, and choosing the endpoints so that it is lower semi-continuous.  Similarly,
\eqref{eq:hep} will be taken to be a piecewise constant $\UC$ function.

Our goal is to take such a sequence of initial data $\xx^\ep_0$ and compute $\pp_{\fh_0} \!\left(\fh(\ft,\fx_i)\leq \fa_i,\;i=1,\dotsc,m\right)$ which, from \eqref{defofh}, \eqref{eq:hep}, \eqref{eq:x0conv} and \eqref{eq:height-cvgce}, is the limit as $\ep\to 0$ of
\begin{equation}\label{eq:TASEPtofp}
	\pp_{X_0}\!\left(X_{2\ep^{-3/2}\ft}(\tfrac12\ep^{-3/2}\ft-\ep^{-1}\fx_i-\tfrac12\ep^{-1/2}\fa_i+1)>2\ep^{-1}\fx_i-2,\;i=1,\dotsc,m\right).
\end{equation}
We therefore want to consider Thm.\,\ref{thm:tasepformulas} with 
\begin{equation}\label{eq:KPZscaling}
t=2\ep^{-3/2}\ft,\qquad  n_i = \tfrac{1}{2}\ep^{-3/2}\ft-\ep^{-1}\fx_i-\tfrac12\ep^{-1/2}\fa_i+1,\qquad a_i=2\ep^{-1}\fx_i-2
\end{equation}
where we will always assume that $\ep$ is small enough so that $n_i>0$ for each $i$.

The formula \eqref{eq:Kt-2} for the TASEP kernel requires initial data which is right-finite.
While one can build a formula which holds without this restriction, it is not nice for passing to limits.
But there is no loss of generality in considering right-finite data because of the next lemma, which says that we can safely cut off our data far to the right.  It also
tells us how fast information is transmitted in the fixed point (see Thm.\,\ref{thm:fps}). 

\begin{defn}\label{def:cutoff}{\bf (Cutoff data)}
\enspace For each integer $L$, the cutoff data is $X_0^{\ep,L}(n) = X^\ep_0(n)$ if 
$n  > -\lfloor\ep^{-1}L'\rfloor$  and $X_0^{\ep,L}(n) = \infty$ if $n\le -\lfloor\ep^{-1}L'\rfloor$, where $L'\sim L/2$ is chosen so that $\ep X^\ep_0(-\lfloor\ep^{-1}L'\rfloor)=L$.
This corresponds to replacing $\fh^{\ep}_0(\fx)$ by $\fh^{\ep,L}_0(\fx)$ with a straight line with slope $-2\ep^{-1/2}$ to the right of $L$.
This is the \emph{$\UC$ cutoff at $L$}.
The \emph{$\LC$ cutoff of $\fg$ at $L$} is just minus the $\UC$ cutoff of $-\fg$.
\end{defn}

The following will be proved in Appx.\,\ref{app:cutoff}:
  
\begin{lem}{\bf (Finite propagation speed)}\label{cutofflemma}
\enspace Suppose that $X^\ep_0$ satisfies \eqref{x0limit} with $\fh_0\in\UC$.
There are $\ep_0>0$ and $C<\infty$, $\delta>0$ independent of $\ep\in(0,\ep_0)$ such that the difference of \eqref{eq:TASEPtofp} computed with initial data $X^\ep_0$ and with initial data $X_0^{\ep,L}$ is bounded by $ C e^{-(\frac23 - \delta) L^3}$.
\end{lem}

\subsection{Limiting operators}\label{sec:one-sided}
The limits are stated in terms of an (almost)  group of operators
\begin{equation}\label{eq:groupS}
\fT_{\ft,\fx}=\exp\{ \fx\partial^2 + \tfrac{\ft}3\tts\partial^3 \}, \qquad \fx,\ft\in\rr^2\setminus \{\fx<0, \ft= 0\}, 
\end{equation}
satisfying $\fT_{\fs,\fx}\fT_{\ft,\fy}=\fT_{\fs+\ft,\fx+\fy}$ as long as all subscripts avoid $\{\fx<0, \ft= 0\}$.  We can think of them as unbounded operators with
domain $\mathscr{C}_0^\infty(\rr)$.  
It is somewhat surprising that they even make sense for
$\fx<0$, $\ft\neq 0$, but it is just an elementary consequence of the following explicit kernel and basic properties of the Airy function\footnote{\label{footairy}${\langle}~~$ is the positively oriented contour going from $e^{-\I\pi/3}\infty$ to $e^{\I\pi/3}\infty$ through $0$.} $\Ai(z)= \frac1{2\pi\I} \int_{\langle}\d w\ts e^{\frac{1}3 w^3-z w}$.  
The $\fT_{\ft,\fx}$ act by convolution
$\fT_{\ft,\fx} f(z) = \int_{-\infty}^\infty\d y\,\fT_{\ft,\fx}(z,y) f(y) = \int_{-\infty}^\infty\d y\, \fT_{\ft,\fx}(z-y) f(y)$
where, for $\ft>0$, 
\begin{equation}\label{eq:fTdef}
\fT_{\ft,\fx}(z)=\frac1{2\pi\I} \int_{\langle}\ts\d w\,e^{\frac{\ft}3 w^3+\fx  w^2+z w} = \ft^{-1/3} e^{\frac{2 \fx^3}{3\ft^2}-\frac{z\fx}{\ft} }\tsm\Ai(-\ft^{-1/3} z+\ft^{-4/3}\fx^2),
\end{equation}
and $\fT_{-\ft,\fx}=(\fT_{\ft,\fx})^*$, or $\fT_{-\ft,\fx}(z,y)=\fT_{-\ft,\fx}(z-y)=\fT_{\ft,\fx}(y-z)$.
From this we get directly the identity $(\fT_{\ft,\fx})^*\fT_{\ft,-\fx}=\fI$, which we will use often without reference.

In addition to $\fT_{\ft,\fx}$ we need to introduce the limiting version of ${\SN}^{\epi(X_0)}_{t,n}$.
It will actually be more convenient for us to introduce the hypograph variant of this operator first, since it is the one that will show up more often in our formulas: For $\fh\in\UC$ we define 
\begin{equation}\label{eq:defShypo}
\fT^{\hypo(\fh)}_{\ft,\fx}(v,u)=\ee_{\fB(0)=v}\big[\fT_{\ft,\fx-\ftau}(\fB(\ftau),u)\uno{\ftau<\infty}\big]
\end{equation}
where $\fB(x)$ is a Brownian motion with diffusion coefficient $2$ and $\ftau$ is the hitting time of the hypograph of $\fh$\footnote{It is important that we use $\fB(\ftau)$ in \eqref{eq:defShypo} and not $\fh(\ftau)$ which, for discontinuous initial data, could be strictly larger.}\footnote{$\fT_{\ft,\fx-\fy}(\fB(\fy),u)$ is a martingale in $\fy\ge 0$. However, it is not uniformly integrable and one \emph{cannot} apply the optional stopping theorem to conclude that $\ee_{\fB(0)=v}\big[\fT_{\ft,\fx-\ftau}(\fB(\ftau),u)\uno{\ftau<\infty}\big]= \fT_{\ft,\fx}(v,u)$. For example, if $\fh\equiv0$, one gets instead $\ee_{\fB(0)=v}\big[\fT_{\ft,\fx-\ftau}(\fB(\ftau),u)\uno{\ftau<\infty}\big]= \fT_{\ft,\fx}(-v,u)$ for $v>0$ (see \cite[Prop. 3.6]{flat}).}.
Note that, trivially, $\fT^{\hypo(\fh)}_{\ft,\fx}(v,u)=\fT_{\ft,\fx}(v,u)$ for $ v\leq\fh(0)$. 
The fact that the expectation in \eqref{eq:defShypo} is finite will be proved in Appx.\,\ref{app:hs}.
$\fT^{\hypo(\fh)}_{\ft,\fx}$ really depends only on the values of $\fh$ on $[0,\infty)$, and we will sometimes evaluate it at functions defined only there. 

One way to think of $\fT^{\hypo(\fh)}_{\ft,\fx}(v,u)$ is as a sort of asymptotic transformed transition density for the Brownian motion $\fB$ to go from $v$ to $u$ hitting the hypograph of $\fh$.
To see what we mean, write
\begin{equation}\label{eq:asymptTransTransProb}
\fT^{\hypo(\fh)}_{\ft,\fx}=\lim_{\mathbf{T}\to\infty}\fT^{\hypo(\fh)}_{[0,\mathbf{T}]}\fT_{\ft,\fx-\mathbf{T}}
\quad\text{with}\quad\fT^{\hypo(\fh)}_{[0,\mathbf{T}]}(v,u)=\ee_{\fB(0)=v}\big[\fT_{0,\mathbf{T}-\ftau}(\fB(\ftau),u)\uno{\ftau\leq\mathbf{T}}\big]
\end{equation}
and note that $\fT^{\hypo(\fh)}_{[0,\mathbf{T}]}(v,u)$ is nothing but the transition density for $\fB$ to go from $v$ at time 0 to $u$ at time $\mathbf{T}$ hitting $\hypo(\fh)$ in $[0,\mathbf{T}]$. 

The epi version of the operator is defined similarly: For $\fg\in\LC$,
\begin{equation}\label{eq:defSepi}
\fT^{\epi(\fg)}_{\ft,\fx}(v,u)=\ee_{\fB(0)=v}\big[\fT_{\ft,\fx-\ftau}(\fB(\ftau),u)\uno{\ftau<\infty}\big]
\end{equation}
where $\ftau$ is now defined as the hitting time of the epigraph of $\fg$ (the meaning of $\ftau$ will always be clear from the context); now we have $\fT^{\epi(\fh)}_{\ft,\fx}(v,u)=\fT_{\ft,\fx}(v,u)$ for $v\geq\fg(0)$.
As a consequence of \eqref{eq:asymptTransTransProb} one can see that the epi and hypo operators are related through
\begin{equation}\label{eq:Sepi-flipped}
\fT^{\epi(\fg)}_{-\ft,\fx}(v,u)=\fT^{\hypo(-\fg)}_{\ft,\fx}(-v,-u).
\end{equation}

\begin{lem}\label{lem:KernelLimit1}
Under the scaling \eqref{eq:KPZscaling} (dropping the $i$ subscripts) and assuming that \eqref{x0limit} holds in $\LC$, if we set $z=2\ep^{-1}\fx+\ep^{-1/2}(u+\fa)-2$ and $y'=\ep^{-1/2}v$, then we have for $\ft>0$ as $\eps\searrow 0$,
\begin{align}\label{eq:QRcvgce}
\fT^\ep_{-\ft,\fx}(v,u)&\coloneqq\ep^{-1/2}\SM_{-t,-n}(y',z)\longrightarrow{} \fT_{-\ft,\fx}(v,u),\\\label{eq:QRcvgce2}
\bar\fT^\ep_{-\ft,-\fx}(v,u)&\coloneqq\ep^{-1/2}\SN_{-t,n}(y',z)\longrightarrow {} \fT_{-\ft,-\fx}(v,u),\\\label{eq:QRcvgce3}
\bar\fT^{\ep,\epi(-(\fh_0^\ep)^-)}_{-\ft,-\fx}(v,u)&\coloneqq\ep^{-1/2} {\SN}^{\oepi(X_0)}_{-t,n}(y',z)\longrightarrow {} \fT^{\epi(-\fh_0^-)}_{-\ft,-\fx}(v,u)
\end{align} 
pointwise, where $\fh^-(x)=\fh(-x)$ for $x\geq0$.  Here $\SM_{-t,-n}$, $\SN_{-t,n}$ are defined in \eqref{def:sm} and \eqref{def:sn}.
\end{lem}

Note that the kernels on the left hand side also depend on $\fa$, but we will not write the dependence explicitly.

The pointwise convergence does not actually suffice for our purposes; it will be suitably upgraded in Appx. \ref{app:hs-estimates}.
The asymptotics in Lem. \ref{lem:KernelLimit1} is elementary and not really a steepest descent. 
Where steepest descent is needed is in Appx. \ref{app:hs-estimates}, to study the asymptotics in $\fx, v, u_i$ of the approximating functions on the left hand side of \eqref{eq:QRcvgce}, \eqref{eq:QRcvgce2}, \eqref{eq:QRcvgce3}, in order to bound the kernels in trace norm (see Sec.\,\ref{app:hs-estimates}).

\begin{proof}
First we give a heuristic proof using operators, which helps one understand where the third derivative comes from.
Since $Q^{-1}=I+2\nabla^+$ with $\nabla^+f(x)=f(x+1)-f(x)$, and dropping lower order terms, the left hand side of \eqref{eq:QRcvgce} is $e^{-\ep^{-3/2}\ft\nabla^-}Q^{-\ep^{-3/2}\ft/2}= e^{\ep^{-3/2}\ft[-\nabla^- + \frac12\log(I+2\nabla^+)]}$.
The scaled lattice is $\ep^{1/2}\zz$, so $\nabla^\pm\sim \ep^{1/2}$ and therefore $-\nabla^- + \tfrac12\log(I+2\nabla^+) = -\nabla^-+\nabla^+-(\nabla^+)^2+\tfrac43(\nabla^+)^3 +\mathcal{O}(\ep^2)$.
Now $(-\nabla^-+\nabla^+-(\nabla^+)^2+\tfrac43(\nabla^+)^3)f (x) = \tfrac43( f(x+3\ep^{1/2})-3f(x+2\ep^{1/2})+3f(x+\ep^{1/2})-f(x))-\tfrac12( f(x+2\ep^{1/2})-3f(x+\ep^{1/2})+3f(x)-f(x-\ep^{1/2}))\sim \tfrac13\ep^{3/2}\partial^3 f(x)$.
We also have $Q^{\ep^{-1}x}\sim e^{x\partial^2}$ under our scaling by the central limit theorem.
This explains how \eqref{eq:groupS} arises.  

Now we switch to the rigorous proof which uses the contour integral representations.
Note that $\fT^{\ep}_{-\ft,\fx}(v,u)$ and $\bar\fT^{\ep}_{-\ft,\fx}(v,u)$ only depend on $v-u$.
Writing $\fT^{\ep}_{-\ft,\fx}(u)=\fT^{\ep}_{-\ft,\fx}(u,0)$ and $\bar\fT^{\ep}_{-\ft,\fx}(u)=\bar\fT^{\ep}_{-\ft,\fx}(u,0)$, changing variables $w\mapsto\frac12(1-\ep^{1/2}\tilde w)$ in \eqref{def:sm} and \eqref{def:sn}, and using the scaling \eqref{eq:KPZscaling}, we have 
\begin{align}
&\fT^{\ep}_{-\ft,\fx}(u) = \frac{1}{2\pi\I} \oint_{\tts C_\ep}\d\tilde w\, e^{\ep^{-3/2}\ft F(\ep^{1/2} \tilde w,\ep^{1/2}\fx_\ep/\ft,\ep u_\ep/\ft ) },\label{eq:ft10}
\\
&\bar\fT^\ep_{-\ft,\fx}(u)= \frac{1}{2\pi\I} \oint_{\tts C_\ep}\d\tilde w\, e^{\ep^{-3/2}\ft F(\ep^{1/2} \tilde w,\ep^{1/2}\bar\fx_\ep/\ft,\ep\bar u_\ep/\ft ) },
\label{eq:ft1} \\
 &F( w,x,u)= \arctanh w -w - x\log(1- w^2) - u  \arctanh w,\label{eq:ft11}
\end{align}
where $\fx_\ep = \fx-\ep^{1/2}(u-\fa)/2-\ep /2$, $\bar\fx_\ep = \fx +\ep^{1/2}(u-\fa)/2 +3\ep /2$, $u_\ep= u-\ep^{1/2}$, and $\bar u_\ep= u+\ep^{1/2}$.  $C_\ep $ is a circle of radius $\ep^{-1/2}$ centred at $\ep^{-1/2}$
and $\arctanh w = \tfrac12[\log(1+w)-\log(1-w)]$.
It is striking how similar the formulas are in this representation and scaling, even if $\bar\fT^{\ep}_{-\ft,\fx}(u)$ comes from an analytic extension of $\fT^\ep_{-\ft,\fx}(u)$. Note that 
\begin{equation}\label{eq:ft3}
\partial_w  F( w,x,u)=(w-w_+)(w-w_-)(1- w^2)^{-1},\qquad w_\pm = w_\pm(x,u) \coloneqq -x \pm \sqrt{ x^2+ u},
\end{equation}
so, in particular, $\p_w F(\ep^{1/2} \tilde w,\ep^{1/2}\fx_\ep/\ft,\ep u_\ep/\ft)=\ep^{3/2}(\tilde w-w^\ep_+)(\tilde w-w^\ep_-)(1- \ep\tilde w^2)^{-1}$ with $w^\ep_\pm=w_\pm(\fx_\ep/\ft,u_\ep/\ft)$.
Keeping in mind that $\fT_{-\ft,\fx}=(\fT_{\ft,\fx})^*$, from \eqref{eq:fTdef} we see then that as, $\ep\searrow 0$, the exponents in \eqref{eq:ft10}, \eqref{eq:ft1} converge to the correct exponents in \eqref{eq:QRcvgce2}, \eqref{eq:QRcvgce3}.
Deform $C_\ep$ to the contour $\langle_{{}_\ep} ~\cup ~C^{\pi/3}_\ep$  where $\langle_{{}_\ep} $ is the part of the Airy contour $\langle$ (see footnote \ref{footairy}) within the ball of radius $\ep^{-1/2}$ centred at $\ep^{-1/2}$, and $C^{\pi/3}_\ep$ is the part of $C_\ep$ to the right of $\langle$.  
As $\ep\searrow 0$, $\langle_{{}_\ep} \longrightarrow \langle$, and it is easy to see that the integral over the part of $\langle$ which is not in $\langle_\ep$ goes to $0$, so it only remains to show that the integral over $C^{\pi/3}_\ep$ converges to $0$.
To see this note that the real part of the exponent of the integral over $C_\ep$ in \eqref{eq:ft10}, parametrized as $\tilde w=\ep^{-1/2}(1+e^{\I\theta})$, is given by $\ep^{-3/2}\ft[-1-\cos(\theta) + (\tfrac14+\mathcal{O}(\ep^{1/2}))\log(5+4\cos(\theta))]$.
Using $\log(1+x) \le x$ for $x\ge 0$, this is bounded by $\ep^{-3/2}\tfrac{\ft}2[-1-\cos(\theta)-\log(2)]$ for sufficiently small $\ep$.
The $\tilde{w}\in C^{\pi/3}_\ep$ correspond to $|\theta|\leq\pi/3$, so the exponent there is less than $-\ep^{-3/2}\kappa\ft$ for some $\kappa>0$.
Hence this part of the integral vanishes.

Now define the scaled walk $\fB_\ep(\fx) = \ep^{1/2}\big(B_{\ep^{-1}\fx} + 2\ep^{-1}\fx-1\big)$ for $\fx\in \ep\zz_{\geq0}$, interpolated linearly in between, and let $\ftau^\ep$ be the hitting time by $\fB_\ep$ of $\epi(-(\fh_0^\ep)^-)$.
By Donsker's invariance principle \cite{billingsley}, $\fB_\ep(\fx)$ converges locally uniformly in distribution to a Brownian motion $\fB(\fx)$ with diffusion coefficient $2$, and therefore (using \eqref{x0limit} and Prop. \ref{keyconvoftau}) the hitting time $\ftau^\ep$ converges to $\ftau$ as well.
\end{proof}

We will compute next the limit of \eqref{eq:TASEPtofp} using \eqref{eq:extKernelProb} under the scaling \eqref{eq:KPZscaling}.
To this end we change variables in the kernel as in Lem. \ref{lem:KernelLimit1}, so that for $z_i=2\ep^{-1}\fx_i+\ep^{-1/2}(u_i+\fa_i)-2$ we need to compute the limit of $\ep^{-1/2}\big(\bar\chi_{2\ep^{-1}\fx-2}K_t\bar\chi_{2\ep^{-1}\fx-2}\big)(z_i,z_j)$.
Note that the change of variables turns $\bar\chi_{2\ep^{-1}\fx-2}(z)$ into $\bar\chi_{-\fa}(u)$.
We have $n_i<n_j$ for small $\ep$ if and only if $\fx_j<\fx_i$ and in this case we have, under our scaling,
\begin{equation}\label{convergence1}\ep^{-1/2}Q^{n_j-n_i}(z_i,z_j)\longrightarrow e^{(\fx_i-\fx_j)\p^2}(u_i,u_j),\end{equation}
as $\ep\searrow 0$.
The rescaled second term in \eqref{eq:Kt-2}, $\ep^{-1/2}(\SM_{-t,-{n}_i})^*\SN^{\oepi(X_0)}_{-t,n_j}(z_i,z_j)$,  can be written as $(\fT_{-\ft,\fx_i}^\ep)^*\fT^{\ep,\epi(-(\fh_0^\ep)^-)}_{-\ft,-\fx_j}(u_i,u_j)$, and we can read off from Lem. \ref{lem:KernelLimit1} that this can be expected to converge to $(\fT_{-\ft,\fx_i})^*\fT^{\epi(-\fh_0^-)}_{-\ft,-\fx_j}(u_i,u_j)$.
The limiting kernel 
\begin{equation}
 \fK_{\lim}\coloneqq e^{(\fx_i-\fx_j)\p^2}\uno{\fx_j<\fx_i}+(\fT_{-\ft,\fx_i})^*\fT^{\epi(-\fh_0^-)}_{-\ft,-\fx_j}\label{eq:Klim}
\end{equation}
would be surrounded by projections $\bar\chi_{-\fa}$.\noeqref{eq:Klim}
For aesthetic reasons, it is nicer to have projections $\chi_{\fa}$, so we change variables $u_i\longmapsto-u_i$ and replace the Fredholm determinant of the kernel by that of its adjoint to get 
$\det\!\left(\fI-\chi_{\fa}\fK^{\hypo(\fh_0)}_{\ft,\uptext{ext}}\chi_{\fa}\right)$ with $
\fK^{\hypo(\fh_0)}_{\ft,\uptext{ext}}(u_i,u_j)=\fK_{\lim}(\fx_j,-u_j;\fx_i,-u_i)$.
The choice of superscript $\hypo(\fh_0)$ in the resulting kernel comes from  \eqref{eq:Sepi-flipped}, which together with $\fT_{-\ft,\fx}(-u,v)=(\fT_{\ft,\fx})^*(-v,u)$ yield

\begin{prop}{\bf (One-sided fixed point formula)}\label{prop:Kfixedpthalf}
\enspace Let $\fh_0\in \UC$ with  $\fh_0(\fx) = -\infty$ for $\fx>0$.
Assume that we start TASEP with right-finite initial data $X_0$ such that the rescaled height function $\fh^\ep(\ft,\fx)$ given by \eqref{eq:hep} satisfies $\fh^\ep_0(\fx)\coloneqq\fh^\ep(0,\fx)\longrightarrow\fh_0(\fx)$ in distribution in $\UC$ as $\ep\to0$. 
Then for any distinct $\fx_1,\dotsc,\fx_m\in\rr$ and any $\fa_1,\dotsc,\fa_m\in\rr$ we have
\begin{equation}
\lim_{\ep\to0}\pp_{\fh^\ep_0}\!\left(\fh^\ep(\ft,\fx_1)\leq \fa_1,\dotsc,\fh^\ep(\ft,\fx_m)\leq \fa_m\right)
=\det\!\left(\fI-\chi_{\fa} \fK^{\hypo(\fh_0)}_{\ft,\uptext{ext}}\chi_{\fa}\right)_{L^2(\{\fx_1,\dotsc,\fx_m\}\times\rr)}\label{eq:onesideext}
\end{equation}
with
\begin{equation}\label{eq:Kexthalf}
\fK^{\hypo(\fh_0)}_{\ft,\uptext{ext}}(\fx_i,\cdot;\fx_j,\cdot)=-e^{(\fx_j-\fx_i)\partial^2}\uno{\fx_i<\fx_j}
+(\fT^{\hypo(\fh_0^-)}_{\ft,-\fx_i})^*\fT_{\ft,\fx_j}.
\end{equation} 
\end{prop}

Our computations here only give pointwise convergence to each of the factors in \eqref{eq:Kexthalf}; even pointwise convergence of the kernels does not follow as there is an integration in the middle of $(\fT^{\hypo(\fh_0^-)}_{\ft,-\fx_i})^*\fT_{\ft,\fx_j}$.
In  Appx.\,\ref{app:hs-estimates} we prove that the operators actually converge in trace class, which  yields convergence of the Fredholm determinants.  

\begin{rem}
A remarkable thing has happened in the limiting operation, showing how non-trivial the limit is. From the biorthogonality condition, the one point TASEP kernel $K_t(n,\cdot;n,\cdot)$ for any initial data is easily seen to be a projection. This property is \emph{lost} in the limit; $\fK^{\hypo(\fh_0)}_{\ft,\uptext{ext}}(\fx,\cdot;\fx,\cdot)$ is \emph{not} a projection in general. In the special case of narrow wedge, it is. But for typical examples, such as half-flat, or flat (with the two-sided formula to appear) it is readily checked that it is not a projection.
\end{rem}

\subsection{From one-sided to two-sided formulas}\label{sec:1to2sided}

The formula for the KPZ fixed point with general initial data $\fh_0$ is obtained in the $L\to\infty$ limit of the formula with truncated initial data $\fh_0^L(\fx)=\fh_0(\fx)\uno{\fx\leq L}-\infty\cdot\uno{\fx>L}$, which is derived from the previous proposition by translation invariance.
The fact that the $L\to \infty$ and $\ep\to 0$ limits commute follows from the fact that the bound in Lem. \ref{cutofflemma} is independent of $\ep>0$.

Given any function $\fg$ we write
\[\fg^{\pm}(\fy)=\fg(\pm\fy)\]
for $\fy\geq0$.
The shift invariance of TASEP, \eqref{eq:transInv}, tells us that $\fh(\ft, \fx; \fh_0^L)\stackrel{\text{dist}}{=} \fh(\ft,\fx-L;\theta_L\fh_0^L)$, where $\theta_L$ is the shift operator from \eqref{eq:shift}, extended to real $L$.
With these shifts, Prop.\,\ref{prop:Kfixedpthalf} tells us that for $\UC$ cutoff data
$\fh_0^{\ep,L}$ (see Def.\,\ref{def:cutoff}),
\begin{equation}\label{eq:oneone}\lim_{\ep\to0}\pp_{\fh^{\ep,L}_0}\!\left(\fh^\ep(\ft,\fx_1)\leq \fa_1,\dotsc,\fh^\ep(\ft,\fx_m)\leq \fa_m\right)
=\det\!\left(\fI-\chi_{\fa}\wt\fK^{\theta_L\fh_0^L}_{L,\uptext{ext}}\chi_{\fa}\right)_{L^2(\{\fx_1,\dotsc,\fx_m\}\times\rr)}\end{equation}
with
\begin{equation}\label{eq:onesidedforlimit}
\wt\fK^{\theta_L\fh_0^L}_{L,\uptext{ext}}(\fx_i,\cdot ;\fx_j,\cdot)=-e^{(\fx_j-\fx_i)\p^2}\uno{\fx_i<\fx_j}+ (\fT^{\hypo((\theta_L\fh_0^L)^-)}_{\ft,-\fx_i+L})^*\fT_{\ft,\fx_j-L}.
\end{equation}
Use \eqref{eq:groupS} to write the second term as $e^{\fx_i\p^2}\!\big((\fT^{\hypo((\theta_L\fh_0^L)^-)}_{\ft,L})^*\fT_{\ft,-L}\big)e^{-\fx_j\p^2}$.
Since $(\theta_{L}\fh_0^L)^+(\fy)=-\infty$ for all $\fy\geq0$, we have $\fT^{\hypo((\theta_L\fh_0^L)^+)}_{\ft,-L}\equiv0$, and then we may rewrite $(\fT^{\hypo((\theta_L\fh_0^L)^-)}_{\ft,L})^*\fT_{\ft,-L}$ as
\begin{equation}
\fI-(\fT_{\ft,L}-\fT^{\hypo((\theta_L\fh_0^L)^-)}_{\ft,L})^*(\fT_{\ft,-L}-\fT^{\hypo((\theta_L\fh_0^L)^+)}_{\ft,-L}).
\end{equation}
The crucial fact, first discovered in \cite{flat}, is that the last expression depends on $L$ only through $\fh_0^L$, and it actually equals $\fK^{\hypo(\fh_0^L)}_\ft$ with
\begin{equation}
\fK^{\hypo(\fh_0)}_\ft= \fI - (\fT_{\ft,0} - \fT^{\hypo(\fh^-_0)}_{\ft,0} )^*(\fT_{\ft,0}-\fT^{\hypo(\fh^+_0)}_{\ft,0}),\label{eq:heatkerout2}
\end{equation}
see \eqref{eq:trcldecomp2a} below and the proof sketch that follows it (note also that in this formula the function $\fh^+_0$ appearing in the last operator no longer needs to be truncated in any way).
Furthermore, it was shown in \cite{flat} (for a more restricted class of $\fh_0$) that we can take $L\to\infty$ on $\fK^{\hypo(\fh_0^L)}_\ft$. 
More precisely, and in the context of the initial data $\fh_0\in\UC$ of the present paper, since $\fh_0^L\longrightarrow \fh_0$ in $\UC$, by Thm.\,\ref{tm:bst1} this kernel converges in trace norm to $\fK^{\hypo(\fh_0)}_\ft$.
The limiting \emph{(extended) Brownian scattering operator} can then be written as\footnote{There is a slight abuse of notation here, and in our earlier rewriting of the second term of \eqref{eq:onesidedforlimit}, when we write $e^{\fx_j\p^2}$ for $\fx_j<0$.
But note that in \eqref{eq:trcldecomp2aa} $e^{\fx_j\p^2}$ appears to the right of $\fK^{\hypo(\fh_0^L)}_{\ft}$, and the action of the backwards heat kernel is well-defined when applied to the second variable of $\fK^{\hypo(\fh_0)}_{\ft}$; the same is true for the earlier formula.
Whenever we write $\fK^{\hypo(\fh)}_{\ft}e^{\fy\p^2}$ for $\fy<0$, we mean that the backwards heat kernel is applied to this second variable first.
In other words, $\fK^{\hypo(\fh)}_{\ft}e^{\fy\p^2}$ is shorthand for the cumbersome $\big(e^{\fy\p^2}(\fK^{\hypo(\fh)}_{\ft})^*\big)^*$.}
\begin{equation} \label{eq:trcldecomp2aa}
\fK^{\hypo(\fh_0)}_{\ft,\uptext{ext}}
= -e^{(\fx_j-\fx_i)\p^2}\uno{\fx_i<\fx_j}+e^{-\fx_i\p^2}\fK^{\hypo(\fh_0)}_\ft e^{\fx_j\p^2}.
\end{equation}
We sometimes also refer to the one-point kernel \eqref{eq:heatkerout2} as the Brownian scattering operator since it is clear how to obtain one from the other.
As we mentioned, by Lem. \ref{cutofflemma} we can interchange limits on the left hand side of \eqref{eq:oneone}.
So we have shown:

\begin{thm}\label{prop:Kfixedptconv}
\enspace Let $\fh_0\in \UC$ and let $\fh_0^\ep$ be rescaled TASEP height functions converging to $\fh_0$ in $\UC$. Let $\fx_1,\dotsc,\fx_m,\fa_1,\dotsc,\fa_m\in\rr$. Then
\begin{equation}
\lim_{\ep\to0}\pp_{\fh^\ep_0}\!\left(\fh^\ep(\ft,\fx_1)\leq \fa_1,\dotsc,\fh^\ep(\ft,\fx_m)\leq \fa_m\right)
=\det\!\left(\fI-\chi_{\fa} \fK^{\hypo(\fh_0)}_{\ft,\uptext{ext}}\chi_{\fa}\right)_{L^2(\{\fx_1,\dotsc,\fx_m\}\times\rr)}.\label{eq:onesideext2pre}
\end{equation}
\end{thm}

\subsection{Tightness and Markov property}\label{sec:tightandmarkov}

The local H\"older spaces $\mathscr C^\beta_{\g}$, $\beta\in(0,1/2)$, $\g<\infty$ defined just after \eqref{eq:defHolderNorm} are compact subsets of our state space $\UC$.

\begin{thm}{\bf (H\"older $\frac12-$ regularity in space)}\label{regep}
\enspace Fix $\ft>0$, $\fh_0\in \UC$ and initial data $X^\ep_0\,$ for TASEP such that 
$\fh^\ep(0,\cdot)\xrightarrow[\ep\to0]{}\fh_0$ in distribution, in $\UC$.
Let $\fh^{\ep}(\ft,\cdot)\in\UC$ be given by \eqref{eq:hep}.
Then for each $\beta\in (0,1/2)$ and $M<\infty$,
\begin{equation}
\lim_{A\to \infty} \limsup_{\ep\to 0} \pp( \| \fh^\ep(\ft)\|_{\beta, [-M,M]}\ge A) =0.\label{tight1ep}
\end{equation}
Consequently, if $\pp_\ep$ represents the law of the functions $\fh^{\ep}(\ft,\cdot)\in\UC$ given by \eqref{eq:hep}, then the family of probability measures
$\{\pp_\ep\}_{0<\ep<1}$ on $\UC$ is \emph{tight} (precompact in the topology of weak convergence of measures).
\end{thm} 
	
The H\"older regularity \eqref{tight1ep} will be proved in Appx.\,\ref{sec:reg} using the exact formulas.
The method is the Kolmogorov continuity theorem, which reduces regularity to two point functions, which we can estimate using trace norms following the proof for the Airy$_1$ process in \cite{quastelRemAiry1}.  To prove the tightness, we need to find compact sets $K_\delta$ in  $\UC$ such that $\limsup_{\delta\to0}\limsup_{\ep\to0}\pp( \fh_\ep(\ft)\not\in K_\delta)=0$.
Since the spaces $\mathscr{C}^{\beta}_{\g}$ are compact in $\UC$, tightness follows from \eqref{tight1ep} as long as we can show that if for all $\ep>0$, $\fh^\ep(0,\fx) \le \gga + \g |\fx|$ for some $\g<\infty$ almost surely, then, for some (possibly random) $\g_\ft<\infty$, $\fh^\ep(\ft,\fx) \le \g_\ft(1+|\fx|)$ for all $\ep>0$ with probability $1$.
From the preservation of max property for TASEP (the analog of Thm.\,\ref{thm:sym}\eqref{pom} below), it suffices to show that $\lim_{A\to \infty}\limsup_{\ep\to 0}\pp( \fh^\ep(\ft, \fx; \g(1+\fx)\uno{\fx>0}) \le A (1+|\fx|), \fx\in\rr)= 1$.
Let $\bar\fh^{\ep,B}_0$ be a rescaled simple asymmetric random walk path with $\bar\fh^{\ep,B}_0 (0) =B$ and  drift $B$.  Then $\lim_{B\to \infty}\limsup_{\ep\to 0}\pp( \bar\fh^{\ep,B}_0(\fx) \ge \g(1+\fx)\uno{\fx>0}, \fx\in\rr)= 1$.  Since the asymmetric random walk is invariant, and the drift
under rescaling is convergent, the height shift is as well.  Therefore 
$\lim_{A\to \infty}\pp( \fh^\ep(\ft, \fx; \bar\fh^{\ep,B}_0(\fx)) \le A (1+|\fx|),\, \fx\in\rr)= 1$ and the result follows from the ordering.

At this point it follows that any fixed $\ft$ distributional limit $\fh(\ft,\cdot)$ has 
finite dimensional distributions given by the right hand side of \eqref{eq:onesideext2pre}.  In particular, it is unique in distribution.  We upgrade to multiple times using the Markov property. However, while one expects the limit of Markov processes to be Markov, this is not always the case.
Note that the limiting transition probabilities given by \eqref{eq:onesideext2} are Feller (continuous functions of $\fh\in\UC$) by Thm.\,\ref{tm:bst1} and the fact that $\mathcal{B}_0(\UC)$, introduced in Sec.\,\ref{UC}, is a generating family for $\mathcal{B}(\UC)$. 

\begin{lem}\label{lem:MP} 
Let $\funnyp^\ep_\fh(t,A)$ be Feller Markov kernels on a Polish space $\mathscr{S}$ for each $\ep>0$, and $\funnyp_\fh(t,A)$ a measurable family of Feller probability kernels on $\mathscr{S}$, such that for each $t>0$ and $\delta>0$ there is a compact subset $K_\delta$ of $\mathscr{S}$ such that $\funnyp^\ep_\fh(t,K_\delta^\mathsf{c})<\delta$, $\funnyp_\fh(t,K_\delta^\mathsf{c})<\delta$ and  $\lim_{\ep\to 0} \funnyp^\ep_\fh(t,A)= \funnyp_\fh(t,A)$ uniformly over $\fh\in K_\delta$ for each $A$ in a generating family.
Then $\funnyp_\fh(t,A)$ satisfies the Chapman-Kolmogorov equations 
\begin{equation} \int \funnyp_\fh(t,\d\fg) \funnyp_\fg(s,A) = \funnyp_\fh(t+s,A).
\label{CKeqns}\end{equation}
\end{lem}

\begin{proof}
   Fix $s,t>0$, $\fh\in \mathscr{S}$, $\delta>0$ and $A\in\mathcal{B}(\mathscr{S})$, choose a compact $K_\delta\subseteq \mathscr{S}$, and choose $\ep_0$ so that for all $\ep<\ep_0$,  $\funnyp^\ep_\fh(t,K^\mathsf{c}_\delta) + \funnyp_\fh(t,K^\mathsf{c}_\delta)<\delta/3$, $|\funnyp^\ep_\fg(s,A) - \funnyp_\fg(s,A)|  <\delta/3$ for all $\fg\in K_\delta$, and $|\tsm\int_{\mathscr{S}} ( \funnyp^\ep_\fh(t,\d\fg) -\funnyp_\fh(t,\d\fg))   \funnyp_\fg(s,A)|<\delta/3$.
   Then $\int_{\mathscr{S}}\big[\funnyp^\ep_\fh(t,\d y) \funnyp^\ep_\fg(s,A) -  \funnyp_\fh(t,\d\fg) \funnyp_\fg(s,A)\big]$ is bounded in absolute value by $\funnyp^\ep_\fh(t,K^\mathsf{c}_\delta) + \funnyp_\fh(t,K^\mathsf{c}_\delta) $ plus  
  \begin{equation}
 	\left|\int_{K_\delta}  \funnyp^\ep_\fh(t,\d\fg) ( \funnyp^\ep_\fg(s,A) -  \funnyp_\fg(s,A) )\right|  + \left|\int_{\mathscr{S}} (\funnyp^\ep_\fh(t,\d\fg) -\funnyp_x(t,\d\fg))  \funnyp_\fg(s,A) \right|,
  \end{equation} 
  all three of which are $<\delta/3$.
\end{proof}

In Appx.\,\ref{app:hs-estimates} we will show

\begin{prop}\label{uniformtightness}  The convergence in Thm.\,\ref{prop:Kfixedptconv} is uniform over initial data $\fh^\ep(0,\cdot)$ in sets of locally bounded H\"older $\beta$ norm, $\beta\in (0,1/2)$.
\end{prop}

As a consequence, we can make the following definition.

\begin{defn}{\bf (KPZ fixed point)}\label{def:fixpt}
\enspace The \emph{KPZ fixed point} is the (unique)  Markov process taking values on $\UC$ with transition probabilities 
given by the extension from the cylindrical sub-algebra $\mathcal{B}_0(\UC)$ to the Borel sets $\mathcal{B}(\UC)$ (see Sec.\,\ref{UC}) of
\begin{equation}
\pp_{\fh_0}\!\left(\fh(\ft,\fx_1)\leq \fa_1,\dotsc,\fh(\ft,\fx_m)\leq \fa_m\right)
=\det\!\left(\fI-\chi_{\fa} \fK^{\hypo(\fh_0)}_{\ft,\uptext{ext}}\chi_{\fa}\right)_{L^2(\{\fx_1,\dotsc,\fx_m\}\times\rr)}.\label{eq:onesideext2}
\end{equation}
The transition probabilities are Feller.
 \end{defn}

In the next section we describe its properties.
We complete this section by recording the statement we have obtained about the convergence of TASEP to this process from Thm.\,\ref{prop:Kfixedptconv}, Lem.\,\ref{lem:MP} and Prop.\,\ref{uniformtightness}.

\begin{theorem}{\bf (Convergence of TASEP)}\label{tm:2}
\enspace Let $\fh^\ep(\ft,\fx)$ be the rescaled height function of TASEP given by \eqref{eq:hep}.
Let $\fh_0$ be a element of $\UC$.
Assume that we have initial data for TASEP chosen to depend on $\ep>0$ in such a way that $\fh^\ep(0,\fx)\longrightarrow\fh_0(\fx)$ in $\UC$ as $\ep\to0$.
Then for each $0<\ft_1<\cdots<\ft_m$, the rescaled (multi-time) TASEP height function $(\fh^\ep(\ft_1,\cdot),\dotsc,\fh^\ep(\ft_m,\cdot))$ converges in distribution in $\UC^m$ to the (multi-time) KPZ fixed point $(\fh(\ft_1,\cdot),\dotsc,\fh(\ft_m,\cdot))$ with initial condition $\fh(0,\cdot)=\fh_0$.
The initial data can be random, and converging in distribution in $\UC$, provided that it is independent of the randomness used to evolve each Markov process.
\end{theorem}

\section{The invariant Markov process}\label{sec:invariant}

\subsection{Brownian scattering theory}\label{sec:bst} 

For height functions $\fh$ in our state space $\UC$ of upper semi-continuous functions (see Sec.\,\ref{UC}), define
local ``Hit'' and ``No hit'' operators by  
 \begin{align}
{\bf P}_{\ell_1,\ell_2}^{\uptext{No hit}\,\fh}(u_1,u_2)\d u_2&=\pp_{\fB(\ell_1)=u_1}\!\left(\fB(\fy)>\fh(\fy)\text{ on }[\ell_1,\ell_2],\,\fB(\ell_2)\in\d u_2\right)
\end{align}
and ${\bf P}_{\ell_1,\ell_2}^{\uptext{Hit}\,\fh}=\fI- {\bf P}_{\ell_1,\ell_2}^{\uptext{No hit}\,\fh}$,
where $\fB$ is a Brownian motion with diffusion coefficient $2$.  

The \emph{Brownian scattering transform} is the map which takes $\fh$ to the $\ft>0$ dependent operator (acting on suitable subspaces of $L^2(\rr)$) introduced in \eqref{eq:heatkerout2}, which can be written
\begin{equation}\label{eq:Kd}
\fK^{\hypo(\fh)}_{\ft}=\lim_{\small\substack{\ell_1\to -\infty\\\ell_2\to \infty}}(\fT_{\ft,\ell_1})^*{\bf P}_{\ell_1,\ell_2}^{\uptext{Hit}\,\fh}\fT_{\ft,-\ell_2},
\end{equation}
where $\fT_{\ft,\fx}=\exp\{ \fx\partial^2 + \tfrac{\ft}3\tts\partial^3 \}$
are defined by the kernels in \eqref{eq:fTdef}.
Before the limit, the right hand side of \eqref{eq:Kd} is exactly  $\fK^{\hypo(\fh_{\ell_1,\ell_2})}_{\ft}$ where $\fh_{\ell_1,\ell_2}$ is $\fh$ in $[\ell_1,\ell_2]$ and $-\infty$ otherwise.  The existence of the limit was first proved for a more restricted class in \cite{flat}.
Since $\fh_{\ell_1,\ell_2} \longrightarrow \fh$ in $\UC$, the limit in \eqref{eq:Kd} follows from Thm.\,\ref{tm:bst1} below\footnote{Thm. \ref{tm:bst1} also asks for the kernel to be conjugated by $\vartheta$, but as can be seen from the arguments in Appx. \ref{app:hs}, this is not necessary for the single-time kernel.}, in trace class in $L^2([\fa,\infty))$ for any fixed $\fa\in\rr$

It was  proved in \cite{flat} that the limit can also be represented as\footnote{The derivation in \cite{flat} takes a different route, starting with known path integral kernel formulas for the Airy$_2$ process and passing to limits. These known formulas themselves arise from the exact TASEP formulas for step initial data. Here we have started from general initial data and derived $\fK^{\hypo(\fh)}_{\ft}$ in a multi-point formula at a later time. Specializing the present derivation to the one-point case, the two routes are linked through time inversion, as explained around \eqref{eq:path-int-kernel-TASEPgem}.}\footnote{\cite{flat} works with the epi version of $\fK^{\hypo(\fh)}_\ft$, which is defined by considering the hitting probabilities of the epigraph of a lower semicontinuous function (see \eqref{eq:defSepi}), and only in the case $\ft=1$, but the proof can be adapted straightforwardly. That paper also works under an additional regularity assumption on the barrier function, the more general setting which we work with here can be handled as in Appx.\,\ref{app:hs}.}
\begin{equation} \label{eq:trcldecomp2a}
\fK^{\hypo(\fh)}_{\ft}=
 ( \fT^{\hypo(\fh^{\fx,-})}_{\ft,\fx} )^*\fT_{\ft,-{\fx}}  
 + (\fT_{\ft,{\fx}} )^* \fT^{\hypo(\fh^{\fx,+})}_{\ft,-{\fx}}  - ( \fT^{\hypo(\fh^{\fx,-})}_{\ft,{\fx}} )^* \fT^{\hypo(\fh^{\fx,+})}_{\ft,-{\fx}}
\end{equation}
for any choice of \emph{splitting point} $\fx$, where
\begin{equation}
\fh^{\fx,\pm}(\fy)=\fh(\fx\pm\fy)\label{eq:defsplith}
\end{equation}
for $\fy\geq0$; the case $\fx=0$ is just a rewriting of \eqref{eq:heatkerout2}, while in \cite{flat} it is shown that the right hand side does not depend on $\fx$. We  sketch the idea:
\begin{equation}\label{eq:splitmethod}
{\bf P}_{\ell_1,\ell_2}^{\uptext{No hit}\,\fh}(u_1,u_2) = p_{\ell_1,u_1} (\ell_2, u_2) p_{\ell_1,u_1, \ell_2,u_2}( \textup{no hit}),
\end{equation}
where $p_{\ell_1,u_1} (\ell_2, u_2)$ is the transition density for Brownian motion to 
be at $u_2$ at time $\ell_2$ given that it started at $u_1$ at time $\ell_1$, and 
where the second factor is the probability for a Brownian bridge with the same endpoints not to hit $\hypo(\fh)$, which we can write for $\fx\in [\ell_1,\ell_2]$ as
\begin{equation}
\int_{\fh(\fx)}^\infty p_{\ell_1,u_1, \ell_2,u_2}(\fB(\fx)\in\d z)p_{0, z,\fx-\ell_1, u_1}( \uptext{no~hit~} \fh^{\fx,-}\uptext{~on~} [0,\fx-\ell_1]) 
p_{0, z,\ell_2-\fx, u_2}( \uptext{no~hit~} \fh^{\fx,+}\uptext{~on~} [0,\ell_2-\fx]).
\end{equation}
Now $ p_{\ell_1,u_1} (\ell_2, u_2) p_{\ell_1,u_1, \ell_2,u_2}(\fB(\fx)\in\d z) = 
p_{0,z}(\fx-\ell_1, u_1)p_{0,z}(\ell_2-\fx, u_2)\d z$ so \eqref{eq:splitmethod} becomes
\begin{equation}
\int_{\fh(\fx)}^\infty\d z~~ p_{0, z}( \uptext{no~hit~} \fh^{\fx,-}\uptext{~on~} [0,\fx-\ell_1],\fx-\ell_1, u_1) 
p_{0, z}( \uptext{no~hit~} \fh^{\fx,+}\uptext{~on~} [0,\ell_2-\fx],\ell_2-\fx, u_2),
\end{equation}
where the notation is meant to indicate that the factors in the integrand are now densities.
Again we can write $p_{0,z}( \uptext{no~hit~} \fh\uptext{~on~} [0,\ell],\ell,u)$ as $
p_{0, z}( \ell, u) -p_{0,z}( \uptext{hit~} \fh\uptext{~on~} [0,\ell],\ell,u) $ which is just
$e^{\ell\partial^2} (z,u) - \int_0^\ell p_z(\ftau \in\d\fs) e^{(\ell-\fs)\partial^2}(\fB(\ftau),u)$.
Therefore the right hand side of \eqref{eq:Kd} goes to the right hand side of 
\eqref{eq:trcldecomp2a} as $\ell_1\to \infty$ and $\ell_2\to \infty$.

In \eqref{eq:trcldecomp2aa} we defined the extended version of the Brownian scattering transform, which in view of \eqref{eq:trcldecomp2a} can be written
\begin{multline} \label{eq:trcldecomp2ab}
\fK^{\hypo(\fh)}_{\ft,\uptext{ext}}
= -e^{(\fx_j-\fx_i)\p^2}\uno{\fx_i<\fx_j}+
 ( \fT^{\hypo(\fh^{\fx,-})}_{\ft,\fx-\fx_i} )^*\fT_{\ft,-{\fx}+\fx_j}  \\
 + (\fT_{\ft,{\fx}-\fx_i} )^* \fT^{\hypo(\fh^{\fx,+})}_{\ft,-{\fx}+\fx_j}  - ( \fT^{\hypo(\fh^{\fx,-})}_{\ft,{\fx}-\fx_i} )^* \fT^{\hypo(\fh^{\fx,+})}_{\ft,-{\fx}+\fx_j}.
\end{multline}

For fixed $\ft$, and after conjugating by
\begin{equation}\label{eq:vartheta}
  \vartheta f(\fx_i,u)=\vartheta_i(u)f(\fx_i,u)\qquad\uptext{with}\qquad\vartheta_i(u)=(1+u^2)^{2i}
\end{equation}
and cutting off by $\P_\fa$ (see \eqref{eq:defChis}), the Brownian scattering transform is continuous on $\UC$:

\begin{theorem}\label{tm:bst1}
 For fixed $\fa_1,\ldots,\fa_m>-\infty$  and $\ft> 0$,
$
\fh\longmapsto\vartheta\chi_{\fa}\fK^{\hypo(\fh)}_{\ft,\uptext{ext}}\chi_{\fa}\vartheta^{-1}$
is a continuous map from $\UC$ into the trace class operators on $L^2(\{\fx_1,\cdots,\fx_m\} \times \rr)$.
\end{theorem}

Moreover, a $\UC$ function $\fh$ can be recovered from its Brownian scattering transform $\fK^{\hypo(\fh)}_{\ft,\uptext{ext}}$:

\begin{theorem}{\bf (Inversion Formula)}\label{tm:bst2}
\enspace For any $\fh \in \UC$,
\begin{equation}
\lim_{\ft\searrow 0}\det\!\left(\fI-\chi_{\fa}\fK^{\hypo(\fh)}_{\ft,\uptext{ext}}\chi_{\fa}\right)_{L^2(\{\fx_1,\dotsc,\fx_m\} \times\rr)} = \prod_{j=1}^m\uno{\fh(\fx_j)\le \fa_j}.
\end{equation} 
\end{theorem}

The fact that the kernel appearing in Thm.\,\ref{tm:bst1} is trace class will be proved in Appx.\,\ref{app:hs}; the continuity stated in the result follows from the arguments in Appx.\,\ref{sec:cvgce}.
Thm.\,\ref{tm:bst2} follows directly from the Chapman-Kolmogorov equations \eqref{CKeqns}. 

\subsection{KPZ fixed point formula}\label{sec:fixedpt}

The Brownian scattering transform linearizes the time evolution of the fixed point transition probabilities:
At the level of Brownian scattering operators, the time flow is linear, satisfying the Lax equation
\begin{equation}
\partial_\ft \fK^{\hypo(\fh)}_{\ft,\uptext{ext}}= [ -\tfrac13 \partial^3, \fK^{\hypo(\fh)}_{\ft,\uptext{ext}}].\label{eq:fixedpt-lax}
\end{equation}
As shown in Sec. \ref{sec:123}, the Fredholm determinant maps this linear flow to the Markov transition probabilities given by the \emph{KPZ fixed point formula},
\begin{equation}
\pp_{\fh_0}\!\left(\fh(\ft,\fx_1)\leq\fa_1,\dotsc,\fh(\ft,\fx_m)\leq\fa_m\right)
=\det\!\left(\fI-\chi_{\fa}\fK^{\hypo(\fh_0)}_{\ft,\uptext{ext}}\chi_{\fa}\right)_{L^2(\{\fx_1,\dotsc,\fx_m\}\times\rr)},\label{eq:twosided-ext}
\end{equation}
The resulting Markov process, the \emph{KPZ fixed point} (see Def.\,\ref{def:fixpt}), is thus a stochastic integrable system in the sense discussed in the TASEP case (Sec.\,\ref{sec:integrability}).

As for TASEP, we also have a version of the fixed point formula in terms of a Fredholm determinant on $L^2(\rr)$ (as opposed to the ``extended space'' $L^2(\{\fx_1,\dotsc,\fx_m\}\times\rr)$).

\begin{prop}\label{prop:pathint-fixedpt}{\bf (Path integral formula for the KPZ fixed point)}
\enspace For $\fh_0\in\UC$, $\ft>0$, and $\fx_1<\dotsm<\fx_m$, we have
\begin{multline}\label{eq:twosided-path}
\pp_{\fh_0}\!\left(\fh(\ft,\fx_1)\leq\fa_1,\dotsc,\fh(\ft,\fx_m)\leq\fa_m\right)\\
=\det\!\left(\fI-\fK^{\hypo(\fh_0)}_{\ft,\fx_1}+\bar\P_{\fa_1}e^{(\fx_2-\fx_1)\p^2}\bar\P_{\fa_2}\dotsm e^{(\fx_m-\fx_{m-1})\p^2}\bar\P_{\fa_m}e^{(\fx_1-\fx_m)\p^2}\fK^{\hypo(\fh_0)}_{\ft,\fx_1}\right)_{L^2(\rr)},
\end{multline}
where
$\fK^{\hypo(\fh_0)}_{\ft,\fx}(\cdot,\cdot)=\fK^{\hypo(\fh_0)}_{\ft,\uptext{ext}}(\fx,\cdot,\fx,\cdot)$.
\end{prop}
This results from an application of \cite[Thm.\,3.3]{bcr}, and is proved in Appx.\,\ref{app:altBCR}.
Taking a continuum limit gives a very symmetric version of the fixed point formula, from which the skew time reversal symmetry, Thm.\,\ref{thm:sym}\eqref{str}, follows by the cyclicity of the  determinant. 

\begin{prop}{\bf (Continuum statistics)}\label{thm:contstats}
\enspace For any $\fh_0\in\UC$, $\fg\in \LC$, and $\ft>0$,
\begin{equation}\label{eq:fpform}
 \pp_{\fh_0}( \fh(\ft, \fx) \le \fg(\fx),\,\fx\in\rr) = \det\!\left(\fI-\fK^{\hypo(\fh_0)}_{\ft/2}\fK^{\epi(\fg)}_{-\ft/2}\right)_{L^2(\rr)}
\end{equation}
with $\fK^{\epi(\fg)}_{-\ft}(u,v)=\fK^{\hypo(-\fg)}_\ft(-u,-v)$.
\end{prop}

$\fK^{\epi(\fg)}_{-\ft}$ is just an upside down version of the Brownian scattering transform introduced in \eqref{eq:trcldecomp2a}, and is built in an analogous way out of hitting probabilities of the epigraph of lower semicontinuous functions (replacing $\fT^{\hypo(\fh)}_{\ft,\fx}$ by $\fT^{\epi(\fg)}_{-\ft,\fx}$, which was defined in \eqref{eq:defSepi}.)
In Appx.\,\ref{app:hs} we show that the operator inside the above Fredholm determinant is trace class after an appropriate conjugation.

\begin{proof}
Consider first $\fg\in\LC$ which is $\infty$ outside of some interval $[-R,R]$.
The left hand side of \eqref{eq:fpform} is then $\pp_{\fh_0}( \fh(\ft, \fx) \le \fg(\fx),\,\fx\in[-R,R])$, which we may obtain by computing $\pp_{\fh_0}( \fh(\ft, \fx_i) \le \fg(\fx_i),\,i=1,\dotsm,m)$ on a mesh $\fx_1<\dotsc<\fx_m$ of $[-R,R]$ and letting the mesh size go to $0$ as $m\to\infty$.
To this end we use Prop. \ref{prop:pathint-fixedpt} with $\fa_i=\fg(\fx_i)$.
Since $\fT_{\ft/2,\fx_1}\fK^{\hypo(\fh_0)}_{\ft,\fx_1}(\fT_{\ft/2,-\fx_1})^*=\fK^{\hypo(\fh_0)}_{\ft/2}$, we can conjugate the kernel inside the determinant in \eqref{eq:twosided-path} by $\fT_{\ft/2,\fx_1}$ to get
\begin{equation}
\fK^{\hypo(\fh_0)}_{\ft/2}-\big[\fT_{\ft/2,\fx_1}\bar\P_{\fg(\fx_1)}e^{(\fx_2-\fx_1)\p^2}\bar\P_{\fg(\fx_2)}\dotsm e^{(\fx_m-\fx_{m-1})\p^2}\bar\P_{\fg(\fx_m)}(\fT_{\ft/2,-\fx_m})^*\big]\fK^{\hypo(\fh_0)}_{\ft/2}.\label{eq:prefulllimit}
\end{equation}
Taking $\fx_1=-R$ and $\fx_m=R$, let $\fg^{(m)}\in\LC$ be given as $\fg^{(m)}(\fx_i+R)=\fg(\fx_i)$, $i=1,\dotsc,m$, and $\fg^{(m)}(\fx)=-\infty$ for all other values of $\fx$, and note that $\bar\P_{\fg(\fx_1)}e^{(\fx_2-\fx_1)\p^2}\bar\P_{\fg(\fx_2)}\dotsm e^{(\fx_m-\fx_{m-1})\p^2}\bar\P_{\fg(\fx_m)}=e^{2R\p^2}-\fT^{\epi(\fg^{(m)})}_{0,2R}$ (because the Brownian motion $\fB$ inside the epi operator can only hit $\epi(\fg^{(m)})$ at $0,\fx_2-\fx_1,\dotsc,\fx_m-\fx_1$, and the left hand side is simply the transition probability for $\fB$ in $[0,2R]$ staying below the same epigraph).
Then the term in brackets in \eqref{eq:prefulllimit} equals 
\[\fT_{\ft/2,-R}\big(\fT_{0,2R}-\fT^{\epi(\fg^{(m)})}_{0,2R}\big)(\fT_{\ft/2,-R})^*=(\fT_{-\ft/2,0})^*\big(\fT_{-\ft/2,0}-\fT^{\epi(\fg^{(m)})}_{-\ft/2,0}\big)=\fI-\fK^{\epi(\fg^{m})}_{-\ft/2},\] 
where the last equality follows from using the epi version of the expansion \eqref{eq:trcldecomp2a} split at $\fx=0$.
As a consequence, the right hand side of  \eqref{eq:twosided-ext} can be written as $\det\!\left(\fI-\fK^{\hypo(\fh_0)}_{\ft/2}\fK^{\epi(\fg^{(m)})}_{-\ft/2}\right)_{L^2(\rr)}$ (after using the cyclic property of the determinant).
Since $-\fg^{(m)}\to -\fg$ in $\UC$, one would like to use Thm.\,\ref{tm:bst1} to pass to the limit and obtain \eqref{eq:fpform}.
The difficulty is that in this Fredholm determinant we are missing the necessary conjugations, but this is resolved by using \eqref{eq:magicConjugation} and the comment that follows it, which implies that the trace norm estimates of Appx. \ref{app:hs} are strong enough to yield continuity in exactly the form we need.

In order to extend the result to all $\fg\in\LC$ it is enough to truncate $\fg$ to a function which is $\infty$ outside $[-R,R]$, apply the result we just proved, and then take $R\to\infty$.
The left hand side clearly converges to $\pp_{\fh_0}( \fh(\ft, \fx) \le \fg(\fx),\,\fx\in\rr)$, while the right hand side converges to the desired Fredholm determinant using again Thm. \ref{tm:bst1} and the same argument as above.
\end{proof}

Analogously to Thm.\,\ref{tm:bst2} we also have the inversion formula 
\begin{equation}
\lim_{\ft\searrow0}\det\!\left(\fI-\fK^{\hypo(\fh)}_{\ft/2}\fK^{\epi(\fg)}_{-\ft/2}\right)=\uno{\fh(\fx)\leq\fg(\fx)~\forall\fx\in\rr}.\label{eq:contstatsinv}
\end{equation}

\subsection{Properties of the KPZ fixed point}\label{sec:properties} The KPZ fixed point satisfies a number of additional properties, which can be proved based both on the explicit formula and on approximation from TASEP.

\begin{theorem}{\bf (Symmetries)}\label{thm:sym}
\enspace Let $\fh(\ft,\fx; \fh_0)$ denote the KPZ fixed point with initial data $\fh_0\in \UC$.\\[-14pt]
\begin{enumerate}[label=\normalfont{(\roman*)},ref=\roman*]
\item \label{123a}  \emph{(1:2:3 scaling invariance)} \enspace
$\alpha\tts\fh(\alpha^{-3}\ft,\alpha^{-2}\fx;\alpha^{-1}\fh_0(\alpha^{2}\fx) )\stackrel{\uptext{dist}}{=} \fh(\ft, \fx;\fh_0), \quad \alpha>0$.
\item \label{ibm} \emph{(Invariance of Brownian motion)} \enspace  If $\fB(\fx)$ is a two-sided Brownian motion, then for each $\ft>0$, $\fh(\ft,\fx; \fB)-\fh(\ft,0;\fB)$ is a two-sided Brownian motion in $\fx$ with diffusion coefficient $2$.
\item  \label{str} \emph{(Skew time reversibility)} \enspace
$\pp\big(\fh(\ft,\fx; \fg)\le -\ff(\fx)\big) =\pp\big(\fh(\ft,\fx;\ff)\le -\fg(\fx)\big),\quad\ff,\fg\in\UC$.
\item \label{sis} \emph{(Stationarity in space)} \hskip0.1in$\fh(\ft,\fx+\fu; \fh_0(\fx-\fu))\stackrel{\uptext{dist}}{=} \fh(\ft,\fx; \fh_0)$.
\item  \label{ri} \emph{(Reflection invariance)} \hskip0.1in$\fh(\ft, -\fx;\fh_0(-\fx))\stackrel{\uptext{dist}}{=} \fh(\ft,\fx; \fh_0)$.
\item \label{ai} \emph{(Affine invariance)}  \hskip0.1in$\fh(\ft,\fx; \fh_0(\fx) + \fa + c\fx)\stackrel{\uptext{dist}}{=} \fh(\ft,\fx+\frac12c\ft; \fh_0(\fx)) + \fa + c\fx + \frac14c^2\ft$.
\item \label{pom} \emph{(Preservation of max)}
\enspace For any $\ff_1,\ff_2\in\UC$,
$\fh(\ft,\fx;\ff_1\vee \ff_2)\stackrel{\uptext{dist}}{=} \fh(\ft,\fx; \ff_1)\vee \fh(\ft,\fx; \ff_2)$.
\end{enumerate}
\end{theorem}

These properties follow from Thm.\,\ref{tm:2}; \eqref{123a} since $\fh$ is a limit and therefore a fixed point of the 1:2:3 rescaling and \eqref{ibm}--\eqref{pom} from the analog properties for TASEP.
\eqref{str}--\eqref{ai} can alternatively be seen to follow directly from the fixed point formula \eqref{eq:twosided-ext} (see also Thm.\,\ref{thm:contstats} in the case of \eqref{str}), and the affine invariance can also be proved from the variational formula,  Thm.\,\ref{thm:airyvar}.  
Note that in \eqref{ibm} there is a non-trivial global height shift and the Brownian motion measure itself is not invariant.
Combining \eqref{ibm} and \eqref{ai} one sees that drifted Brownian motion $\fB(\fx) + \rho\tts\fx$ is also invariant.  

Another property which follows directly by approximation from TASEP (see Lem. \ref{cutofflemma}) is

\begin{theorem}{\bf (Finite propagation speed)}\label{thm:fps}
\enspace Let $\fh_0\in\UC$ with $\fh_0(\fx)\le \gga + \g |\fx|$ and let $\fx_1,\ldots,\fx_m\in \rr$.  
For any $\delta>0$ there exists $C<\infty$ depending only on $\gga,\g, L$, and $\max_i|\fx_i|$, such that for any
$\tilde\fh_0\in \UC$ with  $\tilde\fh_0(\fx)\le \gga + \g |\fx|$ and $\fh_0(\fx)=\tilde\fh_0(\fx)$ for $|\fx|\le L$,
\begin{equation}\label{fsp2}
 \big|\pp_{\fh_0}\tsm\big( \fh(\ft, \fx_i) \le \fa_i,\, i=1,\ldots,m\big)- \pp_{\tilde\fh_0}\!\big( \fh(\ft, \fx_i) \le \fa_i,\, i=1,\ldots,m\big)\big|\le C\tts e^{-(\frac23 - \delta) L^3}.
\end{equation}
\end{theorem}

By bounding above and below by known cases, we obtain rather easily\footnote{These estimates are sharp as $\fa_i\to \infty$ (ignoring lower order terms) and do not depend on the initial data (within $\UC$), but they are not sharp as $\fa_i\to -\infty$. In fact the left tail depends on the initial data; for instance, $\pp( \fh(\ft, \fx) \le \fa)$ is of order $e^{ -\frac1{12}\ft^{-1} | \fa|^{3}}$ for narrow wedge but of order $e^{ -\frac1{6}\ft^{-1} | \fa|^{3}}$ for flat.}

\begin{prop}{\bf (Tail estimates)}\label{prop:tails}
\enspace Let $\fh_0\in\UC$, $\fh_0\not\equiv -\infty$.  Then for fixed $\ft>0$ we have
\begin{align}
 1-e^{ -\frac1{12}\ft^{-1} |\tsm\max_i \fa_i|^{3}(1+o_1(1)) }  \leq \pp_{\fh_0}( \fh(\ft, \fx_i) \ge \fa_i,\, i=1,\ldots,m) \leq e^{-\frac{4}{3} \ft^{-1/2} |\tsm\max_i \fa_i|^{3/2}(1+o_2(1))},
\end{align}
where $o_1(1)\longrightarrow0$ as $\max_i\fa_i\to-\infty$, $~o_2(1)\longrightarrow0$ as $\max_i\fa_i\to\infty$, and both depend only on $\fh_0$, $\ft$, and the $\fx_i$'s.
\end{prop}

\begin{proof} 
Fix $\fx_1,\dotsc,\fx_m\in\rr$ and let $p_m(\fa_1,\ldots, \fa_m)=\pp_{\fh_0}( \fh(\ft, \fx_i) \le \fa_i, i=1,\ldots,m)$.
To see that $p_m(\fa_1,\ldots, \fa_m) \longrightarrow 0$ at the desired speed as any of the $\fa_i$'s goes to $-\infty$ we use the trivial fact that $p_m(\fa_1,\ldots,\fa_m) \le p_1(\fa_i)$ for any $i$. 
By the skew time reversal symmetry and the affine invariance of the fixed point (Thm.\,\ref{thm:sym}(\ref{str},\ref{ai})) together with \eqref{eq:airy2}, we know the one dimensional marginals $p_1(\fa_i)= \pp(\aip_2(\fx) - (\fx-\fx_i)^2 \le -\fh_0(\fx) + \fa_i~~\forall\,\fx\in\rr)$, where $\aip_2(\fx)$ is the Airy$_2$ process (see Sec.\,\ref{sec:airyprocess}) and we have taken $\ft=1$ (general $\ft>0$ follows by scaling invariance).  
Choosing $\bar\fx$ so that $\fh_0(\bar\fx)>-\infty$, we can bound $p_1(\fa_i)$ by  $\pp(\aip_2(\bar\fx) - (\bar\fx-\fx_i)^2 \le -\fh_0(\bar\fx) + \fa_i)$, which is a shifted $F_\text{GUE}$.
Hence we have $p_m(\fa_1,\ldots,\fa_m)\lesssim \exp\{ -\tfrac1{12} |\fa_i|^{3} \}$ as any $\fa_i\to -\infty$, proving the lower bound of Prop.\,\ref{prop:tails}.

To show that $p_m(\fa_1,\ldots, \fa_m) \longrightarrow 1$ at the desired speed as all $\fa_i\to \infty$ one can use \eqref{eq:twosided-ext} together with the estimate $\left|\tts\det(\fI-\fK) - 1\right|\leq\|\fK\|_1e^{\|\fK\|_1+1}$ (with $\|\cdot\|_1$ denoting trace norm, see \eqref{eq:detBd}).
Computing carefully, this gives the desired limit and the upper bound of Prop.\,\ref{prop:tails}.
On the other hand, there is a simple trick using the preservation of max property, Thm.\,\ref{thm:sym}\eqref{pom} (whose proof is independent), which yields the same estimate.
Fix time $\ft=1$ again for simplicity.
Since $\fh_0(\fx)\le \g (1+|x|)$, we have by preservation of max that $\fh(1,\fx) \stackrel{\uptext{dist}}{\le}  \max\{ \fh(1,\fx; \g(1+x)),\fh(1,\fx; \g(1-x))\}$.
By affine invariance, Thm.\,\ref{thm:sym}(\ref{ai}), $\fh(1,\fx; \g(1+x))\stackrel{\uptext{dist}}{=}
 \fh(1,\fx+\tfrac12 \g;0)+ \g(1+x) +\tfrac14\g^2$ and $\fh(1,\fx; \g(1-x))\stackrel{\uptext{dist}}{=} 
 \fh(1,\fx-\tfrac12 \g;0)+ \g(1-x) +\tfrac14\g^2$. So, using \eqref{eq:airy1}, we get 
 $ \pp( \fh(1, \fx_i;\fh_0) \ge \fa_i)\le
\pp( \fh(1,\fx_i+\tfrac12 \g;0)+ \g(1+x_i) +\tfrac14\g^2 \ge  \fa_i) + \pp( \fh(1,\fx_i-\tfrac12 \g;0)+ \g(1-x_i) +\tfrac14\g^2 \ge  \fa_i) = 2-F_\text{GOE}(4^{1/3}(\fa_i-\g(1+x_i)-\tfrac14\g^2)- F_\text{GOE}(4^{1/3}(\fa_i-\g(1-x_i)-\tfrac14\g^2)\gtrsim e^{-\frac43(\ft^{-1/3}\min_i\fa_i)^{3/2}}$, which is what we want.\end{proof}

\begin{rem}{\bf (Replicas and factorization ansatz)}
\enspace An earlier attempt \cite{cqrFixedPt} based on non-rigorous replica methods gave a formula which does not appear to be the same (though there is room for two apparently different Fredholm determinants to coincide).
The replica derivation uses both divergent series and an asymptotic factorization assumption \cite{prolhacSpohn} for the Bethe eigenfunctions of the delta Bose gas.  
The divergent series are regularized through the \emph{Airy trick}, which uses the identity $\int\tsm\d x\Ai(x) e^{nx} = e^{n^3/3}$ to obtain $\sum_{n=0}^\infty (-1)^n e^{n^3/3} \,\ts``\!=\!"\,  \int\tsm\d x\Ai(x) \sum_{n=0}^\infty (-1)^ne^{nx} = \int\tsm\d x\Ai(x)\frac{1}{1+ e^x}$.
Although there is no justification, it is widely accepted in the field that the Airy trick gives consistently correct answers in KPZ.
The factorization assumption, on the other hand, has only been justified by the fact that it has led to the correct result in a few previously known cases.
\end{rem}

\begin{rem}{\bf (Extension in time)}
\enspace TASEP has the unusual property that the initial value problem where we start with $h_0(z)$, and solve for the process $h_t(z)$, $t>0$, can also be done backwards in time.  This is just because the backwards in time dynamics is nothing but the forward in time dynamics for $-h$.  So we can immediately extend the process to $h_t(z)$, $-\infty<t<\infty$.  Of course, for some initial data, such as step $h_0(z) =|z|$,
there will be no movement on $(-\infty,0]$.  The same property is inherited by the KPZ fixed point, except for the \emph{not} technical point that  even if $-\fh_0$ were upper-semicontinuous, it might no longer lie in $\UC$ if it violates the linear growth condition.  For example $\fh_0(\fx) = -\kappa \fx^2$, $\kappa>0$, is good initial data for the KPZ fixed point, but $-\fh_0$ has a finite lifetime $[0,1/\kappa)$ after which it ``explodes" to $+\infty$.  So the initial value problem for the fixed point on $\UC$ has an extension to $(\ft_0,\infty)$, where $\ft_0\le 0$.  The narrow wedge initial data is an
example where $\ft_0=0$.  Continuous $\fh_0$ with $-\fh_0$ satisfying the linear growth condition have $\ft_0=-\infty$.
\end{rem}

\begin{rem}{\bf (Domain Markov property)} 
\enspace {(Suggested by M. Hairer)}
\enspace The KPZ fixed point inherits a stronger space-time Markov property from TASEP, which we describe
informally and without complete proofs.

\noindent First we state the domain Markov property of the space-time TASEP height function
$h_t(z)$.  It is clear from the definition of TASEP that given the height function at $z\in 
\zz$ over some time interval $[t_1,t_2]$, what happens to the height function strictly to the right of $z$ over that time interval is independent of what happens strictly to the left.  
Bootstrapping from this, we see that if $A$ is any connected open subset of 
$ (-\infty,\infty)\times \zz $ which is a finite  union of rectangles $(\ft_1,\ft_2)\times(\fx_1,\fx_2)$ (some of which could be infinite), then $h_t(z)$ has the domain Markov property: If we call the boundary of a subset of the integers those at distance exactly $1$ from the set, then  $\{h_t(z),\,(t,z)\in A\}$ and $\{h_t(z),\,(t,z)\in(A\cup \partial\tsm A)^\mathsf{c}\}$ are independent given $\{h_t(z),\,(t,z)\in\partial\tsm A\}$.

\noindent Now let $\fh(\ft,\fx)$ be the KPZ fixed point on $(\ft_0,\infty)\times \rr$, and let 
$A$ be a connected open subset of this domain with a regular boundary $\partial\tsm A$.  Let $\mathcal{G}_{\partial\tsm A}
= \bigcap_{O\supset \partial\tsm A,\, O\uptext{ open}} \sigma\{ \fh(\ft,\fx), (\ft,\fx)\in O\}$ be the \emph{germ field} of the boundary.  Taking limits from TASEP we see that  $\{\fh(\ft,\fx),\,(\ft,\fx)\in A\}$ and $\{\fh(\ft,\fx),\,(\ft,\fx)\in (A\cup \partial\tsm A)^\mathsf{c}\}$ are independent given $\mathcal{G}_{\partial\tsm A}$.  
 One expects that $\mathcal{G}_{\partial\tsm A}$  actually equals $\sigma(\{ \fh(\ft,\fx), (\ft,\fx)\in \partial\tsm A\})$, but it is not immediately clear how to prove this.

\noindent An important consequence is that the fixed point is not just a Markov process in $\ft$, it is also a Markov process sideways, in $\fx$.
This may partially explain results like the recent two time formulas \cite{johanssonTwoTimeBrownian,johanssonTwoTimeGeometric,baikLiu,johanssonRahman}.
At any rate, it means that while our description of the KPZ fixed point as a Markov process in $\ft$ is the first characterization of the field, it is far from a complete description (see also Rem. \ref{rem:dov}).
\end{rem}

\begin{rem}{\bf (Locality)}  
\enspace There are various notions of locality, the domain Markov 
property above being one; an even stronger statement of locality would follow if we knew the sharp version $\mathcal{G}_{\partial A}=\sigma(\{ \fh(\ft,\fx), (\ft,\fx)\in \partial\tsm A\})$.

\noindent More concretely, one could ask whether
\begin{equation}\label{locality}
\big|\pp_{\fh_0}(\fh(\ft,\fx_i) \le \fa_i, i=1,\ldots,M ) -\pp_{\fh_0^\delta}(\fh(\ft,\fx_i) \le \fa_i, i=1,\ldots,M )\big|=o(\ft)
\end{equation}
 as $\ft\to 0$
whenever $\fh_0^\delta\in\UC$ is such that $\fh^\delta_0(\fy) = \fh_0 (\fy)$ for $|\fy-\fx_i|<\delta$ for one of the $i$.  From the variational formula \eqref{eq:var}, it is fairly straightforward to bound the left hand side of  \eqref{locality} by 
$\exp\{ - C\delta^3/\ft^2 \}$ providing a strong statement of locality.  Presumably this could differentiate between the true fixed point and the non-local stochastic PDE suggested in footnote \ref{jaragonfoot}. The functions have to be in $\UC$; if they are allowed to grow quadratically there are counterexamples.
\end{rem}

\begin{rem}{(\bf Uniqueness and strong KPZ universality conjectures)}
\enspace The KPZ fixed point is expected to be the unique non-trivial (i.e. non-zero) space-time field satisfying locality in the sense of \eqref{locality} and Thm.\,\ref{thm:sym}(\ref{123a},\ref{str},\ref{sis})  (the inviscid limit given by \eqref{eq:invbur} satisfies all but \eqref{str}).

\noindent The \emph{strong KPZ universality conjecture} states that the KPZ fixed point is the limit under the 1:2:3 scaling of all models in the KPZ universality class.
This last statement can alternately be interpreted as the \emph{definition} of the universality class.
Note that it appears to exclude models such as vicious walkers and random matrices, which
have KPZ type fluctuations but seem to  lack a meaningful analogue of a large class of initial conditions.
\end{rem}

From Thm.\,\ref{regep} we obtain

\begin{thm}{\bf (H\"older $\frac12-$ regularity in space)}\label{reg}
\enspace Fix $\ft>0$, $\fh_0\in \UC$, and let $\fh(\ft)$ denote the fixed point at
time $\ft$. 
Then for each $\beta\in (0,1/2)$ and $M<\infty$,
\begin{equation}
\lim_{A\to \infty}  \pp( \| \fh(\ft)\|_{\beta, [-M,M]}\ge A) =0.\label{tight1}
\end{equation}
\end{thm} 

The bounds on the trace norms used to prove Thm. \ref{reg} also yield the local Brownian property for the fixed point (the proof is exactly the same as \cite{quastelRemAiry1}, with $\fK^{\hypo(\fh_0)}_{\ft}$ replacing $B_0$ there).

\begin{thm}{\bf (Local Brownian behavior)}\label{locbr}
\enspace For any $\ft>0$ and any initial condition $\fh_0\in\UC$, $\fh(\ft,\fx)$ is locally Brownian in $\fx$ in the sense\footnote{Since the first version of this article was posted, there has been progress on the stronger statement of absolute continuity with respect to Brownian motion on finite intervals \cite{hammondQuilt,calvertHammondHedge,sarkarVirag}. This uses a different class of approximating models which are shown to converge to the fixed point in \cite{nqr-RBM}.} that for each $\fy\in\rr$, the finite dimensional distributions of $\mathfrak{b}_\ep(\fx)= \ep^{-1/2}(\fh(\ft,\fy + \ep \fx)-\fh(\ft,\fy))$ converge, as $\ep\searrow 0$, to those of a double-sided Brownian motion $\fB$ with diffusion coefficient $2$ and $\fB(0)=0$.
\end{thm} 

By the 1:2:3 scaling invariance, Thm.\,\ref{thm:sym}\eqref{123a}, we have $\fh(\ft,\fx;\fh_0)\stackrel{\uptext{dist}}{=}\ft^{1/3}\fh(1,\ft^{-2/3}\fx;\ft^{-1/3}\fh_0(\ft^{2/3}\fx))$.
Hence the local Brownian behaviour of the fixed point is essentially equivalent to ergodicity.
Recall (see the comment after Thm.\,\ref{thm:sym}) that for any $\rho\in\rr$, drifted Brownian motion $\fB(\fx)+\rho\fx$ is invariant for the fixed point.
The following gives a fairly general condition on initial data in $\UC$ to see $\fB(\fx)+\rho\fx$ locally after a long time:

\begin{thm}{\bf (Ergodicity)}
\enspace 
For any (possibly random) initial condition $\fh_0\in\UC$ such that, for some $\rho\in\rr$, $\ep^{1/2} ( \fh_0(\ep^{-1}\fx) -\rho\ep^{-1}\fx )$ is convergent, in distribution, in $\UC$, the finite dimensional distributions of the process \begin{equation}
\label{eq;resh}\fh(\ft,\fx;\fh_0)-\fh(\ft,0;\fh_0)-\rho\fx\end{equation} converge, as $\ft\to\infty$, to those of a double-sided Brownian motion $\fB$ with diffusion coefficient $2$.
\end{thm}

A similar result was first proved by \citet{pimentelErgod} using coupling, in an article which appeared after the first version of this paper was posted.
The present theorem was added in the second version.  
  
\begin{proof}  By the 1:2:3 scaling and affine invariance properties, Thm.  \ref{thm:sym}(\ref{123a},\ref{ai}), \eqref{eq;resh} is equal in distribution to 
\begin{equation}
\ft^{1/3} \left( \fh(1,\ft^{-2/3}\fx; \ft^{-1/3}(\fh_0(\ft^{2/3}\fx) -\rho\ft^{2/3}\fx))
-\fh(1,0;\ft^{-1/3}(\fh_0(\ft^{2/3}\fx) -\rho\ft^{2/3}\fx))\right).
\end{equation}
Since the initial condition converges in $\UC$, one can repeat the proof of local Brownian behaviour from \cite{quastelRemAiry1}, using now the fact that if $\fh^\ep\to\fh$ in $\UC$
then $\fK^{\hypo(\fh^\ep)}_\ft\longrightarrow \fK^{\hypo(\fh)}_\ft$ in trace norm.
\end{proof}

\subsection{Recovery of the Airy processes}\label{sec:airyprocess}  

Although the determinantal formula \eqref{eq:twosided-ext} used in the definition of the KPZ fixed point looks imposing, we easily recover several of the classical Airy processes\footnote{Besides the ones we treat here, there are three more basic Airy processes $\aip_\text{stat}$, $\aip_{1\to \text{BM}}$ and $\aip_{2\to \text{BM}}$, obtained respectively by starting from a two-sided Brownian motion, a one-sided Brownian motion to the right of the origin and $0$ to the left of the origin, and a one-sided Brownian motion to the right of the origin and $-\infty$ to the left of the origin \cite{imamSasam1,bfsTwoSpeed,baikFerrariPeche,corwinFerrariPeche}.
However, using \eqref{eq:twosided-ext} in these cases involves averaging over the initial randomness and hence verifying directly that the resulting formulas coincide with those in the literature is more challenging.} by starting with special initial data for which the hitting times are explicit, and observing the spatial process at time $\ft=1$.

Start by considering the $\UC$ function $\mathfrak{d}_\fu(\fu) = 0$, $\mathfrak{d}_\fu(\fx) = -\infty$ for $\fx\neq \fu$, known as a {\em narrow wedge at $\fu$}.
It leads to the \emph{Airy$_2$ process} (sometimes simply the \emph{Airy process}):
\begin{align}
\mbox{}\hspace{1.2in}\fh(1,\fx;\mathfrak{d}_\fu)+ (\fx-\fu)^2&~=~ \aip_2(\fx)\qquad\text{(sometimes simply $\aip(\fx)$)}.\label{eq:airy2}\\
\shortintertext{\emph{Flat} initial data $\fh_0\equiv0$, on the other hand, leads to the \emph{Airy$_1$ process}:}
\fh(1,\fx;0)&~=~2^{1/3}\aip_1(2^{-2/3}\fx).\label{eq:airy1}
\shortintertext{Finally the $\UC$ function $\fh_{\text{h-f}}(\fx) = -\infty$ for $\fx<0$, $\fh_{\text{h-f}}(\fx)=0$ for $\fx\geq0$, called {\em  wedge} or {\em half-flat} initial data, leads to the \emph{Airy$_{2\to1}$ process}:}
\fh(1,\fx;\fh_{\text{h-f}})+\fx^2\uno{\fx<0}&= \Bt(\fx).
\end{align}

Formulas for the $m$-point distributions of these special solutions were obtained in the 2000's in \cite{prahoferSpohn,johansson,sasImamPolyHalf,sasamoto,borFerPrahSasam,bfp,bfs} in terms of Fredholm determinants of extended kernels, and later in terms of path-integral kernels in \cite{cqr,quastelRemAiry1,bcr}.  
The Airy$_{2\to1}$ process interpolates between the other two in the limits $\fx\to-\infty$ and $\fx\to\infty$.

We now show how the formula for the Airy$_{2\to1}$ process arises from the KPZ fixed point formula \eqref{eq:twosided-ext}.
The Airy$_1$ and Airy$_2$ processes can be obtained analogously (or in the limits $\fx\to\pm\infty$).
We have to take  $\fh_0(\fx) = -\infty$ for $\fx<0$, $\fh_0(\fx) = 0$ for $\fx\geq0$ in \eqref{eq:twosided-ext}.
It is straightforward to check that $\fT^{\hypo(\fh_0^-)}_{\ft,0}=\bar\P_0\fT_{\ft,0}$, so that $(\fT_{\ft,0}-\fT^{\hypo(\fh_0^-)}_{\ft,0})^*=(\fT_{\ft,0})^*\P_0$.
On the other hand, an application of the reflection principle based on \eqref{eq:asymptTransTransProb} (see \cite[Prop. 3.6]{flat} for the details in the case $\ft=1$) yields that, for $v\geq0$ (using \eqref{eq:Sepi-flipped} and writing $\ftau_0$ for the hitting time of $0$ by $\fB$),
\[\fT^{\hypo(\fh_0^+)}_{\ft,0}(v,u)=\fT^{\epi(-\fh_0^+)}_{-\ft,0}(-v,-u)=\int_0^\infty\!\pp_{-v}(\ftau_0\in\d\fy)\fT_{-\ft,-\fy}(0,-u)=\fT_{\ft,0}(-v,u),\]
which gives
\[\fK^{\hypo(\fh_0)}_{\ft}=\fI-(\fT_{\ft,0})^*\P_0[\fT_{\ft,0}-\varrho\fT_{\ft,0}]=(\fT_{\ft,0})^*(\fI+\varrho)\bar \P_0\fT_{\ft,0},\]
where $\varrho$ is the reflection operator $\varrho f(x)=f(-x)$.
Hence
\[\fK^{\hypo(\fh_0)}_{\ft,\uptext{ext}}(\fx_i,\cdot;\fx_j,\cdot)=-e^{(\fx_j-\fx_i)\p^2}\uno{\fx_i<\fx_j}+(\fT_{\ft,-\fx_i})^*(\fI+\varrho)\bar \P_0\fT_{\ft,\fx_j}.\]
Setting $\ft=1$, we get $\fK^{\hypo(\fh_0)}_{1,\uptext{ext}}(\fx_i,u_i-\fx_i^2\uno{\fx_i\leq0};\fx_j,u_j-\fx_j^2\uno{\fx_j\leq0})=K_{2\to1}(\fx_i,u_i;\fx_j,u_j)$ with $K_{2\to1}$ the extended kernel for the Airy$_{2\to1}$ process, as given in \cite[Eq. 1.8]{qr-airy1to2} (see also \cite{bfs}).
Therefore $\pp\!\left(\fh(1,\fx_i;\fh_0)\leq\fa_i,\,i=1,\dotsc,m\right)=\pp\!\left(\Bt(\fx_i)-\fx_i^2\uno{\fx_i\leq0}\leq \fa_i,\,i=1,\dotsc,m\right)$.

\begin{rem}
Thm.\,\ref{prop:Kfixedptconv} gives a much stronger statement about universality of the Airy processes with respect to initial conditions than was previously known (for one point marginals this appears in \cite{corwinLiuWang}, and to some extent \cite{flat}):
If we start with two rescaled TASEP height functions $\fh_0^{\ep,1}$ and $\fh_0^{\ep,2}$ which converge in distribution in $\UC$ to the same limit $\fh_0$ as $\ep\to0$, then for any $\ft>0$, $\fh^{\ep,1}(\ft,\cdot)$ and $\fh^{\ep,2}(\ft,\cdot)$ have the same (distributional) limit.

\noindent For example, for some fixed $\kappa>0$ one could consider a TASEP initial condition obtained from the periodic case $X_0(i)=2i$, $i\in\zz$, by taking, for each $j\geq1$, the particle at the position $\bar j$ which is closest to $2|j|^{1/\kappa}$, and moving it to $-\bar j-1$, which leads to an initial TASEP height function $h_0(i)\approx-|i|^\kappa$.
If $\kappa<1/2$ then $\fh^\ep(0,\cdot)\longrightarrow0$ as $\ep\to0$ and thus $\fh^\ep(1,\cdot)$ converges to the Airy$_1$ process, while if $\kappa>1/2$ then $\fh^\ep(0,\cdot)\longrightarrow\mathfrak{d}_0$ and $\fh^\ep(1,\cdot)$ converges to the Airy$_2$ process.
A statement such as this appears to have been outside of the scope of previous arguments.
\end{rem}

\subsection{Variational formulas}  

The KPZ fixed point satisfies a version (see \eqref{eq:var} below) of the Hopf-Lax variational formula \eqref{eq:invbur}
with a new noise:

\begin{ex}{\bf (Airy sheet)}
\enspace $\fh(1,\fy; \mathfrak{d}_\fx)+ (\fx-\fy)^2 = \aip(\fx,\fy)$ is called the {\em Airy sheet} (here $\mathfrak{d}_\fx$ is the narrow wedge defined in the last section).
In some contexts it is better to include the parabola, so one writes $\hat\aip(\fx,\fy) = \aip(\fx,\fy)-  (\fx-\fy)^2 $.
Several remarks are in order:
\begin{itemize}[leftmargin=0.8cm]
\item The KPZ fixed point formula does \emph{not} give explicit joint probabilities $\pp(\aip(\fx_i,\fy_i)\le \fa_i, i=1,\ldots,m)$ for the Airy sheet\footnote{The most general formula we can get from the results in Sec.\,\ref{sec:fixedpt} comes from Thm.\,\ref{thm:contstats} and reads $\pp(\hat\aip(\fx,\fy)  \le \ff(\fx)+\fg(\fy),\,\fx,\fy\in\rr)   = \det\!\left(\fI-\fK^{\hypo(-\fg)}_{1/2}\fK^{\epi(\ff)}_{-1/2}\right)$.
Even in the case when $\ff$, $\fg$ take two non-infinite values, it gives a formula for $\pp(\hat\aip(\fx_i,\fy_j) \le \ff(\fx_i)+\fg(\fy_j),\,i,j=1,2)$, but $ \ff(\fx_i)+\fg(\fy_j) $ only span a 3-dimensional linear subspace of $\rr^4$.
So it does \emph{not} determine the joint distribution of $\hat\aip(\fx_i,\fy_j) $, $i,j=1,2$.}, and we presently have no method to obtain them. 
\item Existence of the Airy sheet is obtained in our context from subsequential limits, see Rem. \ref{rem:sheets}.
While our methods leave open the question of uniqueness, this has been proved since the present article was submitted in \cite{dov}, see Rem. \ref{rem:dov}.
Since they start from a different model, one needs to combine their result with \cite{nqr-RBM}.
\item  By stationarity in space, Thm.\,\ref{thm:sym}\eqref{sis}, $\fh(1,\fy; \mathfrak{d}_\fx) \stackrel{\uptext{dist}}{=} 
\fh(1,\fy-\fx; \mathfrak{d}_0)$ and by reflection invariance, Thm.\,\ref{thm:sym}\eqref{ri}, since $\mathfrak{d}_0(-\fx)=\mathfrak{d}_0(\fx)$, $\fh(1,\fy-\fx; \mathfrak{d}_0) \stackrel{\uptext{dist}}{=} 
\fh(1,\fx-\fy; \mathfrak{d}_0)$.  
This gives the \emph{permutation symmetry}
\begin{equation}\label{persym}
\aip(\fx,\fy)\stackrel{\uptext{dist}}{=} \aip(\fy,\fx).
\end{equation}
\item Fixing either variable $\fx$ or $\fy$, $\aip(\fx,\fy)$ is an Airy$_2$ process in the other.
\item The Airy sheet is \emph{stationary}\footnote{Using the methods of this paper one can very easily prove \eqref{aist} in the case $\fx_0=\fy_0$; a proof of the general statement can be found in \cite{dov} (see Rem. \ref{rem:dov}).}: For any fixed $\fx_0,\fy_0$,
\begin{equation}\label{aist}
\aip(\fx+\fx_0,\fy+\fy_0)\stackrel{\uptext{dist}}{=}\aip(\fx,\fy).
\end{equation}
\item New non-obvious distributional symmetries of the Airy sheet have been discovered and conjectured recently (see \cite[Secs. 1.5, 1.6]{borodinGorinWheeler}).
\end{itemize}
\end{ex} 

By repeated application of Thm.\,\ref{thm:sym}\eqref{pom} to initial data which take finite values  $\fh_0(\fx_i)$  at $\fx_i$, $i=1,\ldots, n$, and $-\infty$ everywhere else ), and then taking limits, we obtain

\begin{thm}{\bf (Airy sheet variational formula)}\label{thm:airyvar}  
\enspace For each $\ft>0$,
\begin{equation}\label{eq:var}
\fh(\ft,\fx;\fh_0)  \stackrel{\uptext{dist}}{=}  \sup_{\fy\in\rr}\big\{ \ft^{1/3}\aip(\ft^{-2/3} \fx,\ft^{-2/3} \fy)- \tfrac1{\ft}(\fx-\fy)^2 + \fh_0(\fy)\big\}
\end{equation}
as processes in $\fx$.
In particular, $\aip$ satisfies the \emph{semi-group property}: If $\hat\aip^1$ and $\hat\aip^2$ are independent copies (with parabolas included) and $\ft_1+\ft_2=\ft$ are all positive, then
\begin{equation}\label{sgairy}
 \sup_\fz\left\{ \ft_1^{1/3}\hat\aip^1(\ft_1^{-2/3} \fx,\ft_1^{-2/3} \fz) +  \ft_2^{1/3}\hat\aip^2(\ft_2^{-2/3} \fz,\ft_2^{-2/3} \fy) \right\} \stackrel{\uptext{dist}}{=} \ft^{1/3}\hat\aip^1(\ft^{-2/3} \fx,\ft^{-2/3} \fy).
\end{equation}
\end{thm}

\begin{rem} \label{rem:fullvspoint}
The equalities in distribution \eqref{eq:var} and \eqref{sgairy} hold only for fixed $\ft_1,\ft_2,\ft$, and not as processes in $\ft$.
If $\ft^{1/3}\aip(\ft^{-2/3} \fx,\ft^{-2/3} \fy)$ on the right hand side of \eqref{eq:var} is replaced by $\ft^{1/3}\aip_2(\ft^{-2/3} (\fx- \fy))$, the equality in distribution holds for each fixed $\fx$ and $\ft$, but no longer as processes in $\fx$.
\end{rem} 

\begin{ex}
From \eqref{eq:airy1} and \eqref{eq:var} we deduce that the Airy$_1$ process satisfies
\[2^{1/3}\aip_1(2^{-2/3}\fx)\stackrel{\uptext{dist}}{=}\sup_{\fy\in\rr}\big\{\aip(\fx,\fy)-(\fx-\fy)^2\big\},\]
generalizing the famous identity of \citet{johansson} that the GOE Tracy-Widom distribution can be
written as the sup of the Airy$_2$ process minus a parabola.  The odd factors of $2^{1/3}$ on the left hand side
are the result of a mismatch in natural normalization between the original interpretation from random matrices, and the present one from growth models.
\end{ex}

\begin{rem}\label{rem:sheets}{\bf (Existence of Airy sheets)}
\enspace In TASEP there is a canonical coupling between the process starting from different initial conditions. Take independent Poisson processes of rate $1$, one for each site $x$.  When the Poisson process at $x$ jumps, the TASEP height function jumps down by $2$ if and only if $h(x)$ is a local maximum.  
The coupling just means to use the same background Poisson processes for several different evolving height functions.
It is clear that under such a coupling, the TASEP version of the preservation of max property (Thm.\,\ref{thm:sym}\eqref{pom}) holds.
This seems to have been first exploited by \cite{seppLO}, and leads to the result for the KPZ fixed point in the 1:2:3 limit.

\noindent Indeed, let $\aip^\ep(\fx,\fy)$ denote the 1:2:3 rescaled and recentered (as in \eqref{eq:hep}) TASEP version of the Airy sheet:  
$\aip^\TASEP(x,y) = h(1, y; -|\cdot- x|)$, i.e. the TASEP height function at $y$ at time $1$ starting with packed particles to the left of $x$.
Let $p^\ep_{(\fx_1,\fy_1),\ldots,(\fx_n,\fy_n)}$ denote the joint distribution of $\aip^\ep(\fx_1,\fy_1),\ldots,\aip^\ep(\fx_n,\fy_n)$.  These are a consistent
family of finite dimensional distributions, and the corresponding distributions $\pp^\ep$ of the approximating Airy sheets $\aip^\ep(\fx,\fy)$ are tight in $\mathscr{C}(\rr^2)$, since they satisfy the H\"older bounds \eqref{tight1ep} uniformly in $\ep$, in each variable separately, from the permutation symmetry (which holds at the TASEP level), and therefore in both variables, since the H\"older norm of a function of two variables is easily controlled by the sum of the H\"older norms in each variable.   Any limiting process is called an Airy sheet, and clearly satisfies \eqref{persym},  \eqref{aist}, \eqref{eq:var}, and \eqref{sgairy}\footnote{It is interesting that
although we are unable to prove uniqueness, the variational formulas  \eqref{eq:var}, and \eqref{sgairy} hold  for \emph{any} such limit, especially since the left hand side of \eqref{eq:var} \emph{is} unique.}.

\noindent Either through \eqref{eq:var}, or using a similar construction to the previous paragraph, one produces a basic coupling of the KPZ fixed point starting with different initial data.  So the KPZ fixed point can be thought of as a stochastic flow.  
\end{rem}

\begin{ex}
As a direct consequence of the variational formulas we deduce that if $\fh_\ep(0,\cdot)\longrightarrow\fh_0$ in distribution in $\UC$, then the asymptotic fluctuations of the rescaled TASEP height function $\fh_\ep(1,\fx)$ at a single point $\fx$ (and time $\ft=1$) have the same distribution as $\sup_{\fy\in\rr}\big\{\aip_2(\fx-\fy)-(\fx-\fy)^2+\fh_0(\fy)\big\}$.
This is the result proved in \cite{corwinLiuWang} in the context of discrete time TASEP with sequential update (although their assumptions on the convergence of the initial data are different).
It also coincides, in essence\footnote{The precise connection with the result in \cite{flat} rests on an
assumption which is widely believed to hold, but which currently escapes rigorous treatment (namely that the \emph{partially} asymmetric exclusion process with step initial data converges to the Airy$_2$ process), see Thm. 1.5 in that paper and the discussion preceding it for more details.}, with the result proved in \cite{flat} about the one-point fluctuations for the KPZ equation with general initial data, providing further evidence for the strong KPZ universality conjecture.
\end{ex}

\begin{rem}\label{rem:dov} 
\cite{dov} have recently constructed the \emph{directed landscape} as a (non-explicit) functional of the Airy line ensemble.  This process was earlier called the \emph{space-time Airy sheet} in \cite{cqrFixedPt}.
This nails down the right hand side of the variational formula \eqref{eq:var}, or, alternatively, one can think of the variational formula as defining the KPZ fixed point in terms of the Airy sheet.
Such a definition can be useful to obtain qualitative properties
of the fixed point, though, as we see in this section, many of them do not rely on uniqueness of the Airy sheet, but only its local, Brownian behaviour.
Qualitative properties of the sheet/landscape are obtained in \cite{dov}, but the key point of the present article, the integrability of the fixed point, does not (at the present time) appear to extend to the sheet.
\end{rem}

\subsection{Regularity in time}

We have seen that the fixed point is locally H\"older $\frac12-$ in space, and thus from the 1:2:3 scaling variance (Thm.\,\ref{thm:sym}\eqref{123a}) one expects that it is also locally H\"older $\frac13-$ in time.
This can be proved as an application of the variational formula \eqref{eq:var}; in fact for this purpose one only needs
the pointwise, and not process level, version of the variational formula, so on the right hand side we can replace the Airy sheet by an Airy process (see Remark \ref{rem:fullvspoint}).
To see this, fix $0<\fs<\ft$, $\fx_0\in\rr$, and $\alpha<1/3$, and choose $\beta<1/2$ such that $\beta/(2-\beta)=\alpha$.
We want to compare $\fh(\ft,\fx_0)$ and $\fh_0(\fs,\fx_0)$, but from the Markov property and the fact that at time $\fs$ the process is in $\mathscr C^\beta$, we can assume without loss of generality that $\fs=0$ and $\fh_0\in \mathscr C^\beta$.
There is an $R<\infty$ a.s. such that $|\aip(\fx)|\le R(1+ |\fx|^{\beta})$ and $| \fh_0(\fx)- \fh_0( \fx_0)| \le R( |\fx-\fx_0|^{\beta} + |\fx-\fx_0|)$. 
From the variational formula \eqref{eq:var}, $|\fh(\ft, \fx_0) - \fh(0,\fx_0)|$ is then bounded by
\begin{equation}
\sup_{\fx\in\rr}\Big(R  ( |\fx-\fx_0|^{\beta} + |\fx-\fx_0| + \ft^{1/3} +  \ft^{(1 - 2\beta)/3 }|\fx|^\beta)  - \tfrac1\ft (\fx_0-\fx)^2\Big) \le \tilde{R}\ts \ft^{ {\beta}/{(2-\beta)}}.
\end{equation}
In view of our choice of $\beta$, this yields the desired result: 

\begin{prop}{\bf (H\"older $\frac13-$ regularity in time)}\label{holder}
\enspace For any $0<\alpha<1/3$ and $\fx_0\in\rr$, $\fh(\ft,\fx_0)$ is locally H\"older $\alpha$ in $\ft>0$. 
\end{prop} 

\begin{rem}
One doesn't really expect Prop.\,\ref{holder} to be true at $\ft=0$, unless one starts with H\"older $\frac12-$ initial data, because of the lateral growth mechanism.
For example, we can take $\fh_0(\fx) =\fx^{\beta}\uno{\fx>0}$ with $\beta\in (0,1/2)$ and check using the variational formula that $\fh(\ft,0)-\fh(0,0)\sim \ft^{{\beta}/{(2-\beta)}}$ for small $\ft>0$, which can be much worse than H\"older $1/3-$.
On the other hand, the narrow wedge solution \emph{does} satisfy $\fh(\ft,0;\mathfrak{d}_0)-\fh(0,0;\mathfrak{d}_0)\sim \ft^{1/3}$. At other points $\fh(0,\fx;\mathfrak{d}_0)=-\infty$ while $\fh(\ft,\fx;\mathfrak{d}_0)>-\infty$ so there is not much sense to time continuity at a point.
It should be measured instead in $\UC$, which we leave for future work.
\end{rem}

\subsection{Equilibrium space-time covariance}

White noise plus an arbitrary height shift $\rho\in \rr$ is invariant for the distribution valued spatial derivative process $\mathfrak{u}=\partial_x\fh$ (see the remarks after Thm.\,\ref{thm:sym}) which could be called the \emph{stochastic Burgers fixed point}, since it is expected to be the 1:2:3 scaling limit of the \emph{stochastic Burgers equation} (introduced by \cite{burgers})
\begin{equation}\label{SBE}
\partial_t u = \nu \partial_x u^2 + \lambda\partial_x^2u + \sigma\partial_x\xi
\end{equation}
satisfied by $u=\partial_xh$ from \eqref{KPZ}.
Dynamic renormalization was performed by \cite{forsterNelsonStephen} leading to the dynamic scaling exponent $3/2$.
The equilibrium space-time covariance function was computed in \cite{ferrariSpohnStat}  by taking a limit from TASEP: 
With $\lambda=\nu=1/4$ and $\sigma=1$, and setting $\rho=1/2$,  
\begin{equation}\label{spt}
\ee[ \mathfrak{u}(\ft,\fx)\mathfrak{u}(0,0)] = \tfrac12\ft^{-2/3}g_{\uptext{sc}}''(\ft^{-2/3}\fx),
\end{equation}
where $g_{\uptext{sc}}(w) = \int s^2\d F_w(s)$ with $F_w(s) = \partial_s ( F_\uptext{GUE} (s+w^2) g(s+w^2, w))$, and where
\begin{equation}
g(s,w) = e^{-\frac13 w^3 } \Big[ \int_{\rr^2_-}\d x\,\d y\,e^{w(x+y)}\tsm\Ai(x+y+s) + \langle \hat{\Phi}_{w,s}, (I- K_{\Ai, s}  )^{-1} \hat{\Psi}_{w,s}\rangle_{L^2(\rr_+)}  \Big]
\end{equation}
with $\hat{\Phi}_{w,s}(x)=\int_{\rr_+}\d z\, e^{wz} K_{\Ai,s} (z,x) e^{ws}$, $\hat{\Psi}_{w,s}(y) = \int_{\rr_-}\d z\, e^{wz}\tsm\Ai(y+z+s)$, and $K_{\Ai,s}(x,y)=\int_{\rr_+}\d\lambda \Ai(\lambda+x+s) \Ai(\lambda+y+s) $.  

Since $\mathfrak{u}(\ft,\fx)$ is essentially a white noise in $\fx$ for each fixed $\ft$, one may wonder how the left hand side of \eqref{spt} could even make sense.  
In fact, everything is easily made rigorous: For smooth functions $\varphi$ and $\psi$ with compact support we define $\ee[\langle \varphi, \partial_\fx \fh^{\ep}(\ft,\cdot)\rangle \langle \psi, \partial_\fx \fh^{\ep}(0,\cdot)\rangle ]$ through $\langle \varphi, \partial_\fx \fh^{\ep}(\ft,\cdot)\rangle= -\int\d \fx\,\varphi'(\fx)\fh^{\ep}(\ft,\fx)$.
From our results they converge to $\ee[\langle \varphi, \partial_\fx \fh(\ft,\cdot)\rangle \langle \psi, \partial_\fx \fh(0,\cdot)\rangle ]$.
From \cite[Eqn. (1.10)]{ferrariSpohnStat} they converge to\footnote{The $\frac14$ prefactor comes from a minor correction in the final arXiv version of \cite{ferrariSpohnStat}; we thank Patrik Ferrari and Leandro Pimentel for bringing to our attention the correct scaling in this formula.}
\[\tfrac14\int_{\rr^2}\d\fx\,\d\fy\,\varphi(\tfrac12(\fy+\fx))\psi(\tfrac12(\fy-\fx))\ft^{-2/3}g_{\uptext{sc}}''(\ft^{-2/3}\fx).\]
This gives the equality \eqref{spt} in the sense of distributions.  But since the right hand side is a regular function, the left is as well, and the two sides are equal.

The novelty over \cite{ferrariSpohnStat} is the existence of the stationary Markov process having this space-time covariance.

\appendix

\settocdepth{section}

\section{Trace norm of the fixed point kernel}\label{app:hs}

If $K$ is an integral operator acting on the Hilbert space $\mathcal{H}=L^2(X,\d\mu)$ through its kernel $(Kf)(x) = \int_{X}\d\mu(y)\ts K(x,y)f(y)$, its {\it Fredholm determinant} is defined by
\begin{align}\label{FD}
  \det(I+ K) &= \sum_{n=0}^{\infty}\tr(\Lambda^n(K)) = \sum_{n=0}^{\infty} \frac{1}{n!} \int_{X^n}\d\mu(x_1)\cdots\d\mu(x_n) \det\left[K(x_i,x_j)\right]_{i,j=1}^{n} ,
\end{align}
where $\Lambda^n(K)$ denotes the action of the $n$-fold tensor product $K\otimes \cdots\otimes K$ on the antisymmetric subspace of $\mathcal{H}\otimes \cdots\otimes \mathcal{H}$.
One has $\tr(\Lambda^n(K)) \le \frac{1}{n!}\|K\|_1$,
where $$\|K\|_1=\tr\sqrt{K^*K}$$ is the trace norm, so the Fredholm determinant is finite for trace class operators;
in fact, it is also continuous with respect to the trace norm,
\begin{equation}\label{eq:detBd}
  \big|\tsm\det(I-A)-\det(I-B)\big|\leq \|A-B\|_1e^{1+\|A\|_1+\|B\|_1}.
\end{equation}

While the trace and the Fredholm determinant are invariant under conjugations $K\longmapsto \Gamma^{-1}K\Gamma$, the trace norm is \emph{not}.
So bounds on, and convergence in, trace norm, after appropriate conjugations, will allow us to justify the missing technical steps in Secs.\,\ref{sec:123} and \ref{sec:invariant}.  
(For more background on the Fredholm determinant, including the definition and properties of the Hilbert-Schmidt and trace norms, we refer to \cite{simon} or \cite[Sec.~2]{quastelRem-review}).

\subsection{Proof of Thm. \ref{tm:bst1}}

In this section we prove that the kernel $\fK^{\hypo(\fh_0)}_{\ft,\uptext{ext}}$ in the fixed point formula \eqref{eq:twosided-ext} is trace class (after conjugation by $\vartheta$, defined in \eqref{eq:vartheta}) and depends continuously on the initial data $\fh_0\in\UC$.
The arguments in this section will provide us also with a blueprint for the much harder proof of the fact, to be used crucially in the proof of Prop.\,\ref{prop:Kfixedpthalf}, that the approximating kernels from TASEP are trace class uniformly in the scaling parameter $\ep$ (this is proved in Appx. \ref{app:hs-estimates}).

Our kernels appear in a number of different forms throughout the article.
Since the approximating kernels from TASEP come naturally in the epi form we will prove the result for the epi version; the hypo version will just follow by reflection.
We also have the continuum statistics formula of Thm.\,\ref{thm:contstats}, which is apparently harder than the extended kernel formulas because those are always surrounded by explicit cutoffs $\P_{\fa}$ (in the hypo case; or $\bar\P_\fa$ in the epi case) but this one doesn't seem to have them; in a sense, in Thm.\,\ref{thm:contstats} the second $\fK^{\epi(\fg)}_{-\ft}$ has to act as the cutoff.
To see how this could work, use the definition of $\fK^{\epi(\fg)}_{\ft}$ given in Thm.\,\ref{thm:contstats} and for $f\in L^2(\rr)$ let
$\Gamma f(u)=e^{{G}(u)} f(u)$ where $G$ is antisymmetric, i.e. it produces a cutoff satisfying 
$\Gamma^{-1}\varrho = \varrho\Gamma$ with $\varrho$ the reflection operator $\varrho f(x)=f(-x)$.
Then, after conjugating the kernel by $\Gamma$, we have (using the definition of the epi version of the Brownian scattering transform after \eqref{eq:fpform}, which gives $\fK^{\epi(\fg)}_{-\ft}=(\fK^{\hypo(-\varrho\fg)}_{\ft})^*$)
\begin{equation}\label{eq:magicConjugation}
\begin{split}
\Gamma\fK^{\hypo(\fh)}_{\ft}\fK^{\epi(\fg)}_{-\ft}\Gamma^{-1}&=\big(\Gamma\fK^{\hypo(\fh)}_\ft\Gamma\big)\big(\Gamma^{-1}(\varrho\fK^{\hypo(-\varrho\fg)}_{\ft}\varrho)^*\Gamma^{-1}\big)\\
&=\big(\Gamma\fK^{\hypo(\fh)}_\ft\Gamma\big)\big(\varrho\Gamma\fK^{\hypo(-\varrho\fg)}_{\ft}\Gamma\varrho\big)^*.
\end{split}
\end{equation}
Since the trace class operators form an ideal and $\varrho$ is bounded, it suffices to prove that $\Gamma\fK^{\hypo(\fh)}_{\ft}\Gamma$ is trace class for $\fh\in\UC$.  It will be clear from the argument that the cutoffs  $\P_\fa$ can be replaced by such $\Gamma$ with ${G}(u) = \kappa \sgn(u) |u|^{3/2}$ with a sufficiently small $\kappa>0$, and we will not comment further on this.
  
The form of the kernel for $\fg\in \LC$ can be written explicity using the right hand side of \eqref{eq:trcldecomp2ab},
\begin{multline} \label{eq:trcldecomp2}
\fK^{\epi(\fg)}_{\ft,\uptext{ext}}
= -e^{(\fx_j-\fx_i)\p^2}\uno{\fx_i<\fx_j}+
 ( \fT^{\epi(\fg^{\fx_0,-})}_{-\ft,\fx_0-\fx_i} )^*\fT_{-\ft,-{\fx_0}+\fx_j} \\
  + (\fT_{-\ft,{\fx_0}-\fx_i} )^* \fT^{\epi(\fg^{\fx_0,+})}_{-\ft,-{\fx_0}+\fx_j}  - ( \fT^{\epi(\fg^{\fx_0,-})}_{-\ft,{\fx_0}-\fx_i} )^* \fT^{\epi(\fg^{\fx_0,+})}_{-\ft,-{\fx_0}+\fx_j},
\end{multline}
with ${\fx_0}$ the \emph{splitting point}.
We want to prove that $\vartheta\bar\P_\fa\fK^{\epi(\fg)}_{\ft,\uptext{ext}}\bar\P_\fa\vartheta^{-1}$ is trace class.
We will show that each of the three last terms  is trace class after surrounding by $\bar\P_\fa$.
The argument for the first term using the conjugation by $\vartheta$ is in \cite[Lem. A.2]{bfp}\footnote{Note that there is a typo in the statement of this result, where the ratio corresponding to the $\vartheta_i$'s should be inverted.} and works the same way here.
In the other terms, one can check through the argument that the conjugation by $\vartheta$ does not present any real difficulty, so to make the proof readable we leave them out.  
Note also that, by shifting the height and rescaling $\fh$, we may assume that $\fa=0$ and $\ft=1$.
We will always assume this in the proof in order to make it easier to follow.

The proof uses the classical bound on the Airy functions, $|\tsm\ttsm\Ai(x)|\leq C\tts e^{-2/3(x\vee0)^{3/2}}$, which in our context yields
\begin{equation}\label{eq:Stx-bd}
|\fT_{-1,\fx}(u)|\leq C\tts \exp \hat{F}_0(\fx ,u) 
\end{equation}
with
\begin{equation}\label{eq:Fhat0}
\hat{F}_0(\fx,u) = \fx\fy - \tfrac13\fx^3 -\tfrac23 (\fy \vee 0)^{3/2},\qquad \fy=\fx^2+u.
\end{equation}
By checking various cases, it is elementary to see from \eqref{eq:Stx-bd} that we also have 
\begin{equation}\label{eq:Stx-bd2}
\int_{-\infty}^0\d\eta\, |\fT_{-1,\fx}(u-\eta)|^2 \leq C\exp 2{F}_0(\fx ,u),
\end{equation}
where $F_0=\hat F_0$
unless $\fx>0$ and $u<0$ in which case $F_0=0$ or $\fx=0$ and $u<0$ in which case the bound is $F_0= \log (1+|u|)$.  Note that the constant $C$, and all the bounds here, do depend on $\fx$.

It is enough to control the trace norm of the third term in \eqref{eq:trcldecomp2} since the second term takes the form of a transpose of that one, and the fourth term is the product of two such terms.
Call $\fx_0-\fx_i=\fx_1$ and $-\fx_0+\fx_j=-\fx_2$.
We write our kernel explicitly as
\begin{align}\label{eq:A1}
(\fT_{-1,\fx_1} )^* \fT^{\epi(\fg^{\fx,+})}_{-1,-\fx_2}(z_1,z_2) 
=\int_{\fs\geq0\atop\fb,z\in\rr}\!\d z\,\pp_{\fB(0)=z}(\ftau\in\d\fs,\fB(\ftau)\in\d\fb)\,\fT_{-1,\fx_1}(z,z_1)\fT_{-1,-\fx_2-\fs}(\fb,z_2),
\end{align}
where $\ftau$ is the hitting time of $\epi(\fg^{\fx,+})$.
We can think of the right hand side as an integral of operator kernels in $z_1,z_2$ over some extra parameters $z,\fb$ and $\fs$, and we can estimate its trace norm $\|\cdot\|_1$ by the integral of those trace norms,
\begin{equation}\label{eq:A2}
  \|( \fT_{-1,\fx_1} )^* \fT^{\epi(\fg^{\fx,+})}_{-1,-\fx_2}\|_1
\le \int_{\fs\geq0\atop\fb,z\in\rr}\tsm\d z\,\pp_{\fB(0)=z}(\ftau\in\d\fs,\fB(\ftau)\in\d\fb)\ts\|\fT_{-1,\fx_1}(z,z_1)\fT_{-1,-\fx_2-\fs}(\fb,z_2)\|_1.
\end{equation}
The advantage of the expression on the right hand side is that it isolates very clearly the dependence of the trace norm on the function $\fg$ through the Brownian hitting time and position (see also Rem. \ref{rem:AB}).
Because of the cutoffs $\bar\chi_\fa$, $\fa=0$, the trace norm is computed on $L^2((-\infty,0])$.
The operator inside the norm above is rank one, so its trace norm is now just the product of $L^2$ norms, and using also \eqref{eq:Stx-bd2} we get
\begin{multline}\label{eq:A3}
\|\fT_{-1,\fx_1}(z,z_1)\fT_{-1,-\fx_2-\fs}(\fb,z_2)\|_1 = \sqrt{\int_{-\infty}^0\d z_1\,|\fT_{-1,\fx_1}(z-z_1)|^2
\int_{-\infty}^0\d z_2\,|\fT_{-1,-\fx_2-\fs}(\fb,z_2)|^2}\\
\le C \exp \left\{ {F}_0(\fx_1 ,z)+ {F}_0(-\fx_2-\fs ,\fb)\right\}.
\end{multline}

Now we use our key assumption $\fg(\fx)\ge -\gga - \g |\fx|$.  Observing from $\fx_0$ we have $\fg(\fx_0+\fx) \ge  -\gga - \g |\fx_0+\fx| \ge   -{\tilde{\bm{\alpha}}}- {\tilde{\bm{\gamma}}}\tts|\fx|$ for some new positive constants, which do depend on the $\fx_i$.  So we obtain 
\begin{equation}\label{eq:A4}
\fb\ge  -{\tilde{\bm{\alpha}}}- {\tilde{\bm{\gamma}}}\tts\fs.
\end{equation}
From this it is not hard to see that there are constants $\kappa_1>0$ and $C<\infty$ depending on $\fx_2$ such that 
\begin{align}\label{eq:A5}
{F}_0(-\fx_2-\fs ,\fb)\le -\kappa_1 \fs^3 + C.
\end{align}Furthermore there is a $C<\infty$ depending on $\fx_1$ such that 
\begin{equation}\label{eq:A6}
{F}_0(\fx_1 ,z)\le C + C|z|^{3/2} \uno{z\le -{\tilde{\bm{\alpha}}}} -\tfrac13 |z|^{3/2}
 \uno{z\ge -{\tilde{\bm{\alpha}}}}.
\end{equation}
Let $\fsigma$ be the hitting time of  the epigraph of $-{\tilde{\bm{\alpha}}}-{\tilde{\bm{\alpha}}}|\fx|$ by the Brownian motion  $\fB$. Clearly $\ftau\ge \fsigma$.
We have $\pp_z( \fsigma\le \fs) = \pp_z\big(\tsm\sup_{0\le \fx\le \fs} \fB(\fx) + {\tilde{\bm{\alpha}}}+ {\tilde{\bm{\gamma}}}\fx>0\big)$.
For $z<- {\tilde{\bm{\alpha}}}$ we can bound this by $\pp_z\big(\tsm\sup_{0\le \fx\le \fs} \fB(\fx) + {\tilde{\bm{\alpha}}}+ {\tilde{\bm{\gamma}}}\fs>0\big)$ which can be computed by the reflection principle, to give the bound
\begin{align}\label{}
\pp_z( \ftau\le \fs)\le C\exp\{- \kappa_2 \fs^{-1} (z+{\tilde{\bm{\alpha}}}+{\tilde{\bm{\gamma}}}\fs)^2 \}. 
\end{align}
Putting it all together we have that $\|( \fT_{-1,\fx_1} )^* \fT^{\epi(\fg^{\fx,+})}_{-1,-\fx_2}\|_1$ is bounded by a constant multiple of 
\begin{align}
 \int^{- {\tilde{\bm{\alpha}}}}_{-\infty}\d z \int_0^\infty\d\fs \, e^{ - \kappa_2 \fs^{-1} (z+{\tilde{\bm{\alpha}}}+{\tilde{\bm{\gamma}}}\fs)^2 + C|z|^{3/2} -\kappa_1\fs^3 }
+\int_{- {\tilde{\bm{\alpha}}}}^{\infty}\d z \int_0^\infty\d\fs \, e^{ -\frac13|z|^{3/2} -\kappa_1\fs^3 },
\end{align}
which converge.
This finishes proving $\vartheta\bar\P_\fa\fK^{\epi(\fg)}_{\ft,\uptext{ext}}\bar\P_\fa\vartheta^{-1}$ is trace class.

The method also allows us to show continuity with respect to $\fg\in\LC$ of the above kernel, which yields Thm. \ref{tm:bst1}.
In fact, consider a sequence of functions $(\fg_n)_{n\geq0}$ converging to $\fg$ in $\LC$.
The measures $\pp_{\fB(0)=z}(\ftau^n\in\d\fs,\fB(\ftau^n )\in\d\fb)$, where $\ftau^n$ is the hitting time of $\epi(\fg^n)$, converge to $\pp_{\fB(0)=z}(\ftau\in\d\fs,\fB(\ftau)\in\d\fb)$ as $n\to\infty$, analogously to the last paragraph of the proof of Lem. \ref{lem:KernelLimit1}.
This can be used, together with the above estimates, to show that $\int_{\fs\geq0\atop\fb,z\in\rr}\!\d z\,\pp_{\fB(0)=z}(\ftau^n\in\d\fs,\fB(\ftau^n )\in\d\fb)\,\fT_{-1,\fx_1}(z,z_1)\fT_{-1,-\fx_2-\fs}(\fb,z_2)$ converges in trace norm to $ \int_{\fs\geq0\atop\fb,z\in\rr}\!\d z\,\pp_{\fB(0)=z}(\ftau\in\d\fs,\fB(\ftau)\in\d\fb)\,\fT_{-1,\fx_1}(z,z_1)\fT_{-1,-\fx_2-\fs}(\fb,z_2)$, see Sec. \ref{sec:cvgce} (and in particular Prop. \ref{banach}), where this argument is implemented in the more complicated case of convergence of the TASEP kernels to their fixed point limit.

\begin{rem}\label{rem:AB}
Control of the Fredholm determinant is usually done either by Hadamard's inequality, or through a trace norm estimate.
Many articles in the field skip this step and only prove pointwise convergence of the integral kernel.
Earlier complete arguments were in special cases where all objects in an equation like \eqref{eq:A1} were completely explicit, and Hadamard's inequality is easier to apply.
The trace norm is natural (and yields precise estimates) for Fredholm determinants due to \eqref{eq:detBd}, but has the disadvantage that it is in general difficult to compute.
Usually one tries to write the operator as product $K=AB$, and bound the trace norm by the product of the Hilbert-Schmidt norms of its factors $\|AB\|_1\le \|A\|_2\|B\|_2$, the latter being easy to compute.
In our case the left hand side of \eqref{eq:A1} is obviously a product, but one has to take $A= (\fT_{-1,\fx_1} )^* T^{-1}$, $B=T \fT^{\epi(\fg^{\fx,+})}_{-1,-\fx_2}$ where $T$ is a carefully chosen multiplication operator, depending heavily on the initial data and $\fx$ and $z$ through $\pp_{\fB(0)=z}(\ftau\in\d\fs,\fB(\ftau)\in\d\fb)$.
While writing a proof along these lines, we noticed that it was suggesting that we could just take the trace norm inside the integration as in \eqref{eq:A2}.  It is far from obvious, but true, that the estimate does not give away too much.  Once inside, the trace norm is that of a rank one operator, and easy to compute.  The resulting method, besides working in general, is much easier than earlier proofs.
\end{rem}

\subsection{Proof of Prop. \ref{prop:pathint-fixedpt}}\label{app:fixedpt-pathint}

 In order to check the path integral formula \eqref{eq:twosided-path} we will apply \cite[Thm. 3.3]{bcr} to the extended kernel formula $\det(\fI-\P_{\fa}\wt\fK^{\hypo(\fh_0)}_{\ft,\uptext{ext}}\P_{\fa})$, where $\wt\fK^{\hypo(\fh_0)}_{\ft,\uptext{ext}}$ is the conjugated kernel $\vartheta\ts\fK^{\hypo(\fh_0)}_{\ft,\uptext{ext}}\!\vartheta^{-1}$.
 Here we are using the fact that $\vartheta$ and $\vartheta^{-1}$ commute with $\P_{\fa}$ to see that this determinant is the same as the one in \eqref{eq:twosided-ext}; the conjugation by $\vartheta$ will enable us to check that the analytical assumptions in the \cite{bcr} result are satisfied.
 In the notation of that theorem, we have $Q_{\fx_i}=\P_{\fa_i}$ as well as $\cw_{\fx_i,\fx_j}=\vartheta_i\ts e^{(\fx_j-\fx_i)\p^2}\!\vartheta_j^{-1}$ for $\fx_i<\fx_j$, $K_{\fx_i}=\vartheta_i\ts e^{-\fx_i\p^2}\fK^{\hypo(\fh_0)}_\ft e^{\fx_i\p^2}\!\vartheta_i^{-1}$ and $\cw_{\fx_j,\fx_i}K_{\fx_i}=\vartheta_j\ts e^{-\fx_j\p^2}\fK^{\hypo(\fh_0)}_\ft e^{\fx_i\p^2}\!\vartheta_i^{-1}$ for $\fx_i<\fx_j$; additionally we set $V_{\fx_i}=\fI$, $V_{\fx_i}'=\fI$, $U_{\fx_i}=\Gamma$ and $U_{\fx_i}'=\Gamma^{-1}$ (see \eqref{eq:magicConjugation}).
 Note, however, that in \cite{bcr} the operators $Q_{t_i}$ appear multiplying only on the left of $\fK^{\hypo(\fh_0)}_{\ft,\uptext{ext}}$.
 While we could use the cyclic property of the determinant to remove the second projection $\P_{\fa}$ in our extended kernel formula, it is more convenient to leave it there and note instead that \cite[Thm. 3.3]{bcr} applies in this case just as well, with only a minor modification: Assuming that the operators $Q_{t_i}$ appearing in \cite{bcr} have a square root, then Assumption 1(i) of the theorem 
 is now the boundedness in $L^2(\rr)$ of $Q^{1/2}_{t_i}\cw_{t_i,t_j}Q^{1/2}_{t_j}$ for $i<j$ and of $Q^{1/2}_{t_i}K_{t_i}Q^{1/2}_{t_i}$ and $Q^{1/2}_{t_j}\cw_{t_j,t_i}K_{t_j}Q^{1/2}_{t_j}$ for all $i,j$, and similarly Assumption 3(ii) is the fact that the same operators are trace class when surrounded by $V_{t_i}$ and $V_{t_i}'$ (these modifications are analogous to part of what we do in Appx.\,\ref{app:altBCR}, where we multiply by $N^{1/2}$ on both sides in the left hand side of \eqref{eq:extToBVP}, and their validity can be checked simply by inspecting the proof of Thm.\,\ref{thm:alt-extendedToBVP}).
  
 We turn now to checking that the three assumptions of \cite[Thm. 3.3]{bcr} hold in our setting.
 Assumption 3(ii) (with the modification discussed above) corresponds exactly to the verifying that each of the entries $\P_{\fa}\wt\fK^{\hypo(\fh_0)}_{\ft,\uptext{ext}}\P_{\fa}(\fx_i,\cdot;\fx_j,\cdot)$ of our extended kernel are trace class in $L^2(\rr)$, which is what we just proved above.
 This also yields (the modified) Assumption 1(i), which corresponds to asking only that these operators are bounded.
 Assumption 1(ii), on the other hand, actually does not hold in our setting.
 We note, however, that the assumption is never really used in the proof of \cite[Thm. 3.3]{bcr}.
 In fact, the assumption appears there only because that paper worked in a setting where all Fredholm determinants under consideration involved bounded operators in $L^2(\rr)$, but all that actually matters in the proof is that the operator in that assumption is trace class after an appropriate conjugation, and this is exactly the content of Assumption 3(iii), which can be seen to hold using the above arguments (see the comment after \eqref{eq:magicConjugation}).
 Assumption 3(i) holds trivially.
 Finally, Assumption 2 follows directly from the definition of $\fK^{\hypo(\fh_0)}_\ft$ and the group property of the operators $\fT_{\ft,\fx}$.
 
 Having checked all the assumptions we may now apply the \cite{bcr} result, which yields the path integral formula \eqref{eq:twosided-path} conjugated by $\vartheta_1^{-1}$.

\section{Trace norm convergence of the rescaled TASEP kernels}
\label{app:hs-estimates}

\subsection{Estimates}

In this section we obtain uniform in $\ep$ bounds on the trace norm of the discrete approximations of the fixed point kernel.
We always assume that $\fg^\ep\longrightarrow\fg$ in $\LC$ and, in particular, that they satisfy the linear bound
$\fg^\ep(\fy),\fg(\fy) \ge - \gga - \g |\fy|$
for $\fy\in \rr$, uniformly in $\ep$.

The proof somewhat follows the lines of the continuum version, but there are several new difficulties.
The first is that the continuum proof used many asymptotics of the functions $\fT_{-\ft,\fx}$, each of which has to be done separately now using steepest descent on the contour integrals defining the functions $\fT^{\ep}_{-\ft,\fx}$ and $\bar\fT^{\ep}_{-\ft,\fx}$.
A more serious problem is that we don't have a split formula at the TASEP level, i.e. a formula of the type \eqref{eq:heatkerout2}.
In other words, we don't really have a usable formula for two-sided data for TASEP.
Such a formula appeared in the first version of this article on the arXiv, but it does not seem to be usable, and, in particular, we have not succeeded so far to employ it to control the trace norm of the kernel.
Because of this, and as we discussed in Sec. \ref{sec:123}, we need estimates for the TASEP kernel for the $\LC$ cutoffs of $\fg^\ep$ at $L<\infty$ (see Defn. \ref{def:cutoff}), uniformly in $L$, so that the cutoff can later be removed (through the finite propagation speed result proved in Appx. \ref{app:cutoff}).
The same type of estimates are needed to prove the uniform bounds on the H\"older norms of the fixed point (see Appx. \ref{sec:reg}). More precisely, we need to prove:

\begin{prop}\label{prop:epsilonkernel}
Consider $\fg^\ep$ as above and let $\fg^\ep_L$ denote its $\LC$ cutoff at $L>0$.
Then the trace norm of $(\fT^{\ep}_{-\ft,\fx_1-L})^*\bar\fT^{\ep,\epi((\fg_L^{\ep})^{L,-})}_{-\ft,-\fx_2+L}$
on $L^2((-\infty,\fa])$ is bounded uniformly in $\ep$ and $L$.
\end{prop}

By definition of the notation $\fg^{L,-}$ (see \eqref{eq:defsplith}) we have $(\fg_L^{\ep})^{L,-}(\fy)=(\fg^{\ep})^{L,-}(\fy)$ for all $\fy\geq0$, so in the proposition we could have written $\fg^\ep$ without the cutoff.
We have chosen this formulation to stress the role of $L$ in this result and our interest in it.
We remark that the uniformity in $L$ will not be used in the

From the proof of the continuum case, one can see already that proving this is going to be difficult.
In the proof we first take care of the case $L=0$ and then, after \eqref{shiftything}, extend to all $L>0$, using that crucial identity as the main tool.

From the continuum proof we see that the key point is to bound the trace norm by the integral of trace norms of rank one operators, which become $L^2$ norms.
So we introduce the notation
\begin{align}
\exp 2 F_\ep(\fx,u)=\int_{-\infty}^0\d\eta\,|\fT^{\ep}_{-1,\fx}(u-\eta)|^2 ,\qquad\exp 2\bar F_\ep(\fx,u)=\int_{-\infty}^0 \d\eta\,|\bar\fT^{\ep}_{-1,\fx}(u-\eta)|^2;
\end{align}
here and below we write $\fT^{\ep}_{-1,\fx}(u-z)=\fT^{\ep}_{-1,\fx}(u,z)$ and $\bar\fT^{\ep}_{-1,\fx}(u-z)=\bar\fT^{\ep}_{-1,\fx}(u,z)$.
Going back to \eqref{def:sm} and \eqref{def:sn} to compute these integrals as summations we obtain
\begin{equation}\label{eq:B15}
\int_{-\infty}^0 \d\eta\,|\fT^{\ep}_{ -1,\fx}(u-\eta)|^2= \itwopii{2} \oint_{\tts C^o_\ep}\d\tilde w_1 \oint_{\tts C^o_\ep}\d\tilde w_2\,
\frac{e^{\ep^{-3/2} \sum_{i=1}^2 F(\ep^{1/2} \tilde w_i, \ep^{1/2}\fx_\ep,\ep u_\ep)}}
{\tilde w_1+\tilde w_2-\ep^{1/2}\tilde w_1\tilde w_2}
\end{equation}
with $\fx_\ep =\fx-\ep^{1/2}u/2-\ep /2$  and $u_\ep= u-\ep^{1/2}$ and $F$ defined in \eqref{eq:ft11}, as well as 
\begin{multline}\label{eq:B16}
\int_{-\infty}^0 \d\eta\,|\bar\fT^{\ep}_{-1,\fx} (u-\eta)|^2\\
=\itwopii{2}\oint_{C^o_\ep}\d\tilde w_1\oint_{C^o_\ep}\d\tilde w_2\,\frac{e^{\ep^{-3/2}\sum_{i=1}^2 F(\ep^{1/2}\tilde w_i,\ep^{1/2}\bar\fx_\ep,\ep \bar u_\ep)}(1+\ep^{1/2}\tilde w_1)(1+\ep^{1/2}\tilde w_2)}{\tilde w_1+\tilde w_2+\ep^{1/2}\tilde w_1\tilde w_2},
\end{multline}
with  $\bar\fx_\ep = \fx +\ep^{1/2}u/2 +3\ep /2$ and $\bar u_\ep= u+\ep^{1/2}$; the notation $C_\ep^o$ means that the singularity at $0$ is outside the contour.

The following lemma  replaces \eqref{eq:Stx-bd2} in the discrete case.
The function $\hat{F}_0$ is replaced by  $F(\fw_+,\fx,u)$, which we call
\begin{equation}\label{eq:Fdecomp}
\hat F_\ep(\fx,u)
=\Re\!\big[(1+\nu_1(\ep ^{1/2} \fw_+))( \fx \fy -\tfrac13 \fx^3 -\tfrac23 \fy^{3/2}) + \nu_2(\ep^{1/2} \fw_+) ( \fx\fy  - \tfrac12 \fx^2\fy^{1/2}-\tfrac12\fy^{3/2})\big]
\end{equation}
where $\fw_+= -\fx+\sqrt{\fy}$, as before $\fy = \fx^2 +u$, and
\begin{align}
\nu_1(\fw) &=\fw^{-3}( 3(1+\fw^2)\arctanh \fw +3\fw\log (1-\fw^2)-3\fw)-1,\\
\nu_2(\fw) & =\fw^{-3}(-2(3+\fw^2)\arctanh \fw -4\fw\log (1-\fw^2) +6\fw).
\end{align}
These two functions are analytic in $\fw\in\mathbb{C}-(-\infty,-1]\cup[1,\infty)$, uniformly bounded in absolute value everywhere, vanish at $0$ like $\fw^2 $, and are non-negative on $(-1,1)$, since they have convergent series expansions $\nu_1(\fw) = \sum_{n\ge 2~\uptext{even}} \tfrac{6}{(n+1)(n+2)(n+3)} \fw^{n}$, $\nu_2(\fw)  =\sum_{n\ge 2~\uptext{even}} \tfrac{4n}{(n+1)(n+2)(n+3)} \fw^{n}$ there.
Here and below $\sqrt{\fy}$ always refers to the positive square root.

The following lemma covers different regions in the asymptotics of the functions $F_\ep(\fx,u)$ and $\bar F_\ep(\fx,u)$.  Unfortunately, there does not appear to be one argument which covers all
regions, as we have complicated functions of several variables converging in $\ep$.  On the other hand, we do not need all regions and the estimates we actually need  are far from the optimal ones\footnote{The estimates and arguments which we require in this lemma and in the rest of this section are much more involved than those appearing in earlier proofs of convergence to the classical Airy processes. One reason for this is that, whereas exact contour integral formulas were available in those special cases, our formulas involve expectations over random walk hitting times, and in order to handle them we need to control the behavior of the contour integrals in some additional, complicated regions. But even if this were not a problem, we are in a situation where we need to obtain much finer estimates on our integrals in order to prove the uniform bounds on the H\"older norms of the fixed point, which play a crucial role in our arguments.}. 

\begin{lem}\label{lem:B1}
In the following all constants are independent of everything including $\ep$ unless noted.
\begin{enumerate}[label=\normalfont{(\roman*)}]
\item  \label{item0} Suppose that $\fx_\ep^2+ u_\ep \ge 0$ and $-\fx_\ep+\sqrt{\fx_\ep^2+ u_\ep  }\ge \ep^{-1/2}$. Then $F_\ep(\fx,u) = -\infty$.
\item \label{item3} Suppose that $\fx_\ep^2+ u_\ep \ge 0$.  There is a $\delta>0$ such that for $-\delta\le -\fx_\ep+\sqrt{\fx_\ep^2+ u_\ep  }< \ep^{-1/2}$, and $\ep^{1/2}\fx_\ep \in (1-\sqrt{5},1+\sqrt{5})$ we have $F_\ep(\fx,u)\le \hat F_\ep(\fx_\ep, u_\ep)$. Under the same conditions on $\bar\fx_\ep$ and $\bar u_\ep$,  $\bar F_\ep(\fx,u)\le \hat F_\ep(\bar \fx_\ep, \bar u_\ep)$.
\item  \label{item1} Suppose that $\fx,\fx_\ep\ge 0$. Then $F_\ep(\fx,u) \le 
C\log(2+ |u|+|\fx|)$.
The same holds for $\bar F_\ep(\fx,u)$ under the conditions $\bar\fx_\ep \ge 0$ and $u\ge  -\ep^{-1/2}\fx$.
\item \label{item4} Suppose that $\fx_\ep^2+u_\ep < 0$, $-C\ep^{-1/4}\leq\fx_\ep<0$, and $|u_\ep|\le C\ep^{-1/2}$.  Then there is a $C'$ depending on $|u_\ep|^{1/2}/|\fx_\ep|$ such that for any $\delta>0$, $F_\ep(\fx,u)\le -(\tfrac23 -\delta)|\fx|^3 + C\delta^{-1} |u|^{3/2} +C'$.
\item \label{item5} Suppose that $\bar\fx_\ep^2+\bar u_\ep < 0$, $-C<\bar\fx_\ep<0$, and $|\bar u_\ep|\le C$. Then $\bar F_\ep(\fx,u)\le C$.
\end{enumerate}
\end{lem}

\begin{proof}
Recall that in \eqref{eq:B15} and \eqref{eq:B16} the contour $C^o_\ep$ is $C_\ep$, a circle of radius $\ep^{-1/2}$ centered at $\ep^{-1/2}$, with a little blip taken at its left so that $0$ lies outside the contour.
Recall also that the function $F$ appearing in the exponents there is given by $F(w,x,u)= \arctanh w -w - x\log(1- w^2) - u  \arctanh w$.
We note that the real part of this function is symmetric in the imaginary part of $w$, so in the proof it will be enough to estimate the integrand along the upper half of the contour.

\vskip2pt
\noindent\underline{Proof of \ref{item0}}.
Squaring both sides of $\sqrt{\fx_\ep^2+ u_\ep  }\ge \ep^{-1/2}+\fx_\ep$ and using $\fx_\ep^2+ u_\ep \ge 0$ gives $\alpha_\ep\coloneqq \tfrac12\ep^{-3/2}+\ep^{-1}\fx_\ep-\tfrac12\ep^{-1/2} u_\ep\le 0$.
But looking at the contour integral defining $\fT^\ep_{-1,\fx}(u)$ through \eqref{eq:QRcvgce} and \eqref{def:sm}, we see that the pole at $w=0$ disappears exactly when $\alpha_\ep\leq0$, which shows that the integrand is analytic and thus $\fT^\ep_{-1,\fx}(u)=0$ in this case.

\vskip2pt
\noindent\underline{Proof of \ref{item3}}.
We are trying to estimate \eqref{eq:B15} and \eqref{eq:B16}, and the term in the exponent is the same except in one case evaluated at $\fx_\ep,u_\ep$ and in the other at $\bar\fx_\ep,\bar u_\ep$.
So the proofs will be the same and we just call the value in the exponent $\tilde\fx_\ep,\tilde u_\ep$ to stand for one or the other (we will use this convention also in the proof of the other cases).
We deform the contour $C_\ep$  to a contour $\tilde\fw_+^\ep + r e^{\pm i\pi/4}$, with $r$ going from $0$ until it hits the right arc of the old contour $C_\ep$ (so that $r\in[0,c\tts\ep^{-1/2}]$ with $c\sim \sqrt{2}$), together with that right arc of the old contour $C_\ep$ with angles $\le \pi/4$.  Note that during the deformation we do not pass through any zeros of  the denominator in either  \eqref{eq:B15} or \eqref{eq:B16}.
The contour we have described is a ``steep descent curve" in the sense that it is close enough to the steepest descent curve for our purposes.
In this case, it does actually pass through the critical point $\fw_+^\ep$.
Our estimate is simply the value of the integrand at this point, times a constant estimating the integration along the rest of the curve.
To prove it, we therefore have to show that the rest of the integration is bounded independently of $\ep$.
In particular, we need to show that the real part of the exponent is decreasing along the curve uniformly in $\ep$ as we move away from the critical point.
The computation is not difficult because we have
$\partial_wF= (w-w_+)(w-w_-)(1-w^2)^{-1}$.  We get
$
\partial_{r}\Re[\ep^{-3/2}F(\ep^{1/2}(\tilde\fw_+^\ep+re^{\pm\I\pi/4}),\ep^{1/2}\tilde\fx_\ep,\ep\tilde u_\ep)]=-Kr^2
$
where $K=\sqrt{2}\left( 1-w^2 + \tilde r^2+[4w(w+x)+2\sqrt{2}\tilde r(w+x)]\right)/\left(
 (w^2+\sqrt{2}  \tilde rw-1)^2+2 \tilde r^2w^2+2\sqrt2 \tilde r^3 w+\tilde r^4\right)$ with  $w=\ep^{1/2}\tilde\fw_+^\ep$, $x=\ep^{1/2}\tilde\fx_\ep$, $u=\ep\tilde u_\ep$ and $\tilde r=\ep^{1/2} r$.
We will show that  $K\ge \sqrt{2}/20$.
In the numerator , if $w\ge 0$, the term in square brackets is non-negative by assumption, so we can drop it to get a lower bound.
In  the denominator we can use $( 1-w^2 +\sqrt{2}  \tilde rw)^2 \le 2( 1-w^2)^2 + 4 \tilde r^2 w^2$ and $ \tilde r^3 w \le \tfrac12 \tilde  r^2 + \tfrac12 \tilde r^4w^2$
as well as  $ \tilde r^4 \le 2 \tilde r^2$ to bound it by
$
10( ( 1-w^2)^2 +  \tilde r^2 w^2 +\tilde r^2) 
=
20 ( \tfrac12( 1-w^2)( 1-w^2  -  \tilde r^2)  +\tilde r^2).
$
So we just have to show that 
$
 1-w^2+\tilde r^2 \ge   \tfrac12( 1-w^2)( 1-w^2  -  \tilde r^2)  +\tilde r^2
$,
which, since $w\in[0,1)$, is easily seen to be true.
It is not hard to see that these inequalities remain true for $\tilde\fw_+\ge -\delta$ for some $\delta>0$.

Next we check that the exponent is decreasing along the arc of $C_\ep$ ending at $2\ep^{-1/2}$.  
Using now $\partial_wF= (w^2+2x w -u)(1-w^2)^{-1}$ we have to show that 
$\Re[ ((e^{\I\theta} +1)^2 +2\ep^{1/2}\tilde\fx_\ep (e^{\I\theta} +1) -\ep u ) (1-(e^{\I\theta} +1)^2)^{-1}\I e^{\I\theta}] >0$ for $\theta\in (0,\alpha\tts\pi)$ for some $\alpha\le 1/2$.  The real part is easily computed to be 
$(5 + 4\cos\theta)^{-1} ( 4(\cos\theta +1) + 2\ep^{1/2}\tilde\fx_\ep   +\ep\tilde u)\sin\theta $ 
so we only need to show that $4(\cos\theta +1) + 2\ep^{1/2}\tilde\fx_\ep   +\ep\tilde u_\ep \ge 0$ in this region.
Now $\cos\theta\ge 0$, so this is at least $4 +2\ep^{1/2}\tilde\fx_\ep + \ep\tilde u_\ep$.  We have $\tilde u_\ep\ge -\fx_\ep^2$ so the exponent is decreasing as long as $\tilde\fx_\ep \in (\ep^{-1/2}(1-\sqrt{5}),\ep^{-1/2}(1+\sqrt{5}))$.

\vskip2pt
\noindent\underline{Proof of \ref{item1}}.
The situation now is a little different because we may not be able to move to the critical point without passing through a pole.
On the other hand, we don't really need to because we are not trying to get an optimal estimate.    Instead we deform the contour so that it passes through the real line at $q\coloneqq(1+|\tilde u_\ep|+|\tilde \fx_\ep|)^{-1}\in(0,\ep^{-1/2})$, then move in the vertical direction until we hit the straight line from the proof of \ref{item3} coming at angle  $\pi/4$ out of the critical point, and then continue  until hitting the curve $C_\ep$ as before.
Along the vertical part we have $\partial_{r}\Re[\ep^{-3/2}F(\ep^{1/2}(q+ir,\ep^{1/2}\tilde\fx_\ep,\ep \tilde u_\ep)]= (-2(q+\tilde\fx_\ep)+\mathcal O(\ep)) r
$.  Since $\tilde\fx_\ep\ge 0$ the real part is decreasing along this verticle piece. Along the straight piece at angle  $\pi/4$, the proof of \ref{item3} still works to prove the uniform decrease.  The proof of uniform decrease along  the arc of $C_\ep$ is different for $F_\ep$ $\bar F_\ep$,  and depends on the precise dependence of $u_\ep$, $\fx_\ep$, $\bar u_\ep$ and $\bar\fx_\ep$ on the bare variables $u$ and $\fx$.  
In the first case,
$\Re [\ep^{-3/2}F(\ep^{1/2}  w, \ep^{1/2}\fx_\ep,\ep u_\ep)]=\Re [\ep^{-3/2}(\arctanh \ep^{1/2} w -\ep^{1/2}w)] - \ep^{-1} \fx \Re[\log(1+\ep^{1/2} w)] $ does not even depend on $u$, and decreases uniformly along the arc of $C_\ep$ as long as $\fx\ge 0$.
In the second case, $\Re[ \ep^{-3/2}F(\ep^{1/2}  w, \ep^{1/2}\bar\fx_\ep,\ep\bar u_\ep)]=\Re[ \ep^{-3/2}(\arctanh \ep^{1/2} w -\ep^{1/2}w)]
- \ep^{-1} (\fx +\ep^{1/2} u)\Re[\log(1+\ep^{1/2} w)] $  so we require  
$u>-\ep^{-1/2}\fx$.
Now that we have checked that the exponent is decreasing uniformly along the 
curve, we end up with an estimate in terms of the value of the integrand at $\tilde w_i=(1+|\tilde u_\ep|+|\tilde \fx_\ep|)^{-1}$.  The exponent is bounded and we pick up a term $\log (2+ |\tilde\fx_\ep|+|\tilde u_\ep|)$ from the denominator.

\vskip2pt
\noindent\underline{Proof of \ref{item4} and \ref{item5}}.
The critical points are complex now, and we deform $C_\ep$ to a contour passing through both $\fw_+^\ep$ and $\fw_-^\ep$.
The contour consists of a straight line from $0$ to $\fw_+^\ep$, then a straight line moving out from $\fw_+^\ep$ at angle  $\pi/4$ until it hits $C_\ep$, and then it continues along $C_\ep$ in the usual clockwise direction until it hits the real axis, after which it follows the reflected curve across the real axis, back to $0$.
However, to avoid the singularity at $0$ from the denominator in \eqref{eq:B15} and \eqref{eq:B16}, we
cut off the tip of the curve just to the right of $0$.

The first thing we need to check is that the real part of the exponent is increasing uniformly in $\ep$ along the linear piece between $0$ and $\tilde\fw_+^\ep$.
Using $\partial_wF= (w-w_+)(w-w_-)(1-w^2)^{-1}$ we compute $\partial_r \ep^{-3/2}F(\ep^{1/2}r\tilde\fw_+^\ep,\ep^{1/2}\tilde\fx_\ep,\ep \tilde u_\ep) = (r-1) (\tilde\fw_+^\ep)^2( r\tilde\fw_+^\ep-\tilde\fw_-^\ep)v$ where $v=(1-\ep r^2 (\tilde\fw_+^\ep)^2)^{-1}$.  Here $r\in(0,1)$ so $r-1<0$.
From the assumptions $\tilde\fw_\pm^\ep = R\tts e^{\pm\I\theta}$ with $R>0$ and $\theta \in (0,\pi/2)$ so $(\tilde\fw_+^\ep)^2( r\tilde\fw_+^\ep-\tilde\fw_-^\ep)=R^3( re^{3\I\theta}-e^{-\I\theta})=R'\ttsm e^{\I\theta'}$ with $R'>0$ and $\theta'\in (\pi/2,3\pi/2)$.
Now since $|\tilde\fw_+^\ep|\le C\ep^{-1/4}$, for $\ep$ small enough $\theta' +\arg v \in  (\pi/2,3\pi/2)$ and hence $\Re[(\tilde\fw_+^\ep)^2( r\tilde\fw_+^\ep-\tilde\fw_-^\ep)v]$ is strictly positive, uniformly in $\ep\in[0,\ep_0]$. Note this argument is not uniform in $\theta'$ and hence we end up with a constant which blows up with $|\tilde u_\ep|^{1/2}/|\tilde\fx_\ep|$.

Next we have to check that the real part of the exponent is decreasing uniformly in $\ep$ along the line at angle  $\pi/4$ coming out of $\tilde\fw_+^\ep$, so that this piece of the integral
is bounded uniformly in $\ep$.  Using the formula for $\partial_wF$ above, $\partial_{r}\ep^{-3/2}F(\ep^{1/2}(\tilde\fw_+^\ep+re^{\I\pi/4}),\ep^{1/2}\tilde\fx_\ep,\ep\tilde u_\ep) =  \ep^{-1}\partial_wF  e^{\I\pi/4}\big|_{w=\ep^{1/2}(\tilde\fw_+^\ep+re^{\I\pi/4})}= -r (2|\tilde\fy_\ep|^{1/2}- re^{\I 3\pi/4})  v$ where $\tilde\fy_\ep=\tilde\fx_\ep^2+\tilde u_\ep$ and
$v=( 1- \ep( \fw_+^\ep+re^{\I\pi/4})^2)^{-1}$.  Hence the real part is less than or
 $-Cr|\fy_\ep|^{1/2}$ as long as $\arg v\in (-\pi/4,\pi/2) $, i.e. it is enough that
$\arg ( 1- \ep( \fw_+^\ep+re^{\I\pi/4})^2 )\in (-\pi/2,\pi/4)$, which is true for small enough $\ep$ since   $r$ is less than $\sqrt{2} \ep^{-1/2}$, and $|\fw_+^\ep|\le C\ep^{-1/4}$.
Finally we need to check the exponent is still decreasing as we move along the arc of the curve $C_\ep$, but the proof given in case \ref{item3} above works in the same way here.

Hence we have an estimate $F_\ep(\fx,u) \le \hat F_\ep(\fx_\ep,u_\ep)$.  

To prove \ref{item4} we need to estimate $\hat F_\ep(\fx_\ep,u_\ep)$.
We use \eqref{eq:Fdecomp}.
First of all since $|\fw_+^\ep|\le C\ep^{-1/4}$ one has 
$|\nu_i(\ep^{1/2}\fw_+^\ep)|\le C\ep^{1/2}$.  Furthermore, $ |\fx \fy - \tfrac13 \fx^3 -\tfrac23 \fy^{3/2}|, |\fx\fy  - \tfrac12 \fx^2\sqrt{\fy}-\tfrac12\fy^{3/2}|\le C(|\fx|^3 + |u|^{3/2})$.  Thus 
$\Re\!\big[\nu_1(\ep ^{1/2} \fw_+)( \fx \fy - \tfrac13 \fx^3 -\tfrac23 \fy^{3/2}) + \nu_2(\ep^{1/2} \fw_+) ( \fx\fy  - \tfrac12 \fx^2\sqrt{\fy}-\tfrac12\fy^{3/2})\big]\le C\ep^{1/2} (|\fx|^3 + |u|^{3/2})$.  Here $\fx=\fx_\ep$ and $\fy= \fx_\ep^2 + u_\ep$. 
So it just remains to bound the real part of the term $\fx \fy - \tfrac13 \fx^3 -\tfrac23 \fy^{3/2}$.  Since $\fy<0$, the real part of the
 third term vanishes.  Write the first and second as $\fx_\ep u_\ep + \tfrac23 \fx_\ep^3= \fx_\ep u_\ep - \tfrac23 |\fx_\ep|^3\le  - (\tfrac23-\delta) |\fx|^3 + C\delta^{-1} |u|^{3/2}$, which yields the desired estimate.
 
To prove \ref{item5} it remains to show that $\hat F_\ep(\bar\fx_\ep,\bar u_\ep)\le C$.  By the same argument as above we have $\hat F_\ep(\bar\fx_\ep,\bar u_\ep)\le (1+C\ep^{1/2}) (|\fx_\ep|^3 + |u_\ep|^{3/2})$, which proves it.
\end{proof}

\begin{proof}[Proof of Prop.\,\ref{prop:epsilonkernel}]
The first step is to obtain a bound  when $\fx_1,\fx_2\ge2$ and the cutoff is at $L=0$.
In this case $(\fg^\ep_L)^{L,-}$ becomes simply $(\fg^\ep)^-$, and the operator appearing in the result is given by 
\begin{multline}
(\fT^\ep_{-1,\fx_1} )^* \bar\fT^{\ep,\epi((\fg^\ep)^-)}_{-1,-\fx_2}(z_1,z_2) \\
=\int_{\fb,z\in\rr,\ts\fs\in[0,\ep n)}\!\d z\,\pp_{\fB_\ep(0)=z}(\ftau_\ep\in\d\fs,\fB_\ep(\ftau_\ep)\in\d\fb)\,\fT^\ep_{-1,\fx_1}(z,z_1)\bar\fT^\ep_{-1,-\fx_2-\fs}(\fb,z_2)
\end{multline}
(here $n=\tfrac12\ep^{-3/2}+\ep^{-1}\fx_2+1$ from \eqref{eq:KPZscaling}).
We  think of the right hand side as an integral of operator kernels in $z_1,z_2$ over some extra parameters $z,\fb$ and $\fs$.
We estimate exactly as in \eqref{eq:A1}-\eqref{eq:A3} to see that $\big\|( \fT^\ep_{-1,\fx_1} )^* \bar\fT^{\ep,\epi((\fg^\ep)^-)}_{-1,-\fx_2}\big\|_1$ is bounded by 
\begin{equation}
 \label{secondtermww} 
 \int_{\fb,z\in\rr,\ts \fs\in[0,\ep n)}\d z~~\pp_{\fB_\ep(0)=z}(\ftau_\ep\in\d\fs,\fB_\ep(\ftau_\ep)\in\d\fb)\ts \exp\!\left\{ F_\ep(\fx_{1},z)+\bar F_\ep(-\fx_{2}-\fs,\fb)\right\}.
\end{equation}
It is convenient to recall at this point that, in the context of the above bound, the parameters appearing in \eqref{eq:B15}, \eqref{eq:B16} and Lem. \ref{lem:B1} are given by
\begin{equation}
\fx_\ep =\fx_1-\tfrac12\ep^{1/2}z-\tfrac12\ep, \quad u_\ep= z-\ep^{1/2},\quad \bar\fx_\ep = -\fx_2-\fs +\tfrac12\ep^{1/2}\fb +\tfrac32\ep\qand\bar u_\ep=\fb+\ep^{1/2}.\label{eq:thexueps}
\end{equation}
We remark that $\fx_1$ and $\fx_2$ here are fixed; constants in the estimates below may (and will) depend on them.

Consider first the case $z\geq-1$.
Recalling that $\fx_2\geq2$, one can check that $\fx_\ep^2+u_\ep\geq0$ and that if we let $\bar z=\frac12\ep^{-1}+\ep^{-1}\fx_1$ then $-\fx_\ep+\sqrt{\fx_\ep^2+u_\ep}-\ep^{-1/2}$ is negative for $z\in[-1,\bar z)$ and non-negative for $z\geq\bar z$.
In the first case we may use Lem.\,\ref{lem:B1}\ref{item3} and the fact that for $\fy=\fx_\ep^2+u_\ep \ge 0$ we have that $\nu_1$ and $\nu_2$ are positive and bounded and $\fx\fy  - \tfrac12 \fx^2\sqrt{\fy}-\tfrac12\fy^{3/2}\le 0$ to find a $C<\infty$ such that $\hat F_\ep(\fx_{\ep},z)\le C -\frac13|z|^{3/2}$, while for $z\geq\bar z$ then we may simply use Lem.\,\ref{lem:B1}\ref{item0} to get a much better bound.
On the other hand, when $z<-1$, we use Lem.\,\ref{lem:B1}\ref{item1} to find $C<\infty$ such that $F_\ep(\fx_{1},z) \le C (1+ \log |z|)$.
Therefore we may choose a constant $C>0$ depending on $\gga$ such that
\begin{equation}
 F_\ep(\fx_{1},z)\le C + C\log |z|\uno{z\le -\gga} -\tfrac13|z|^{3/2}\uno{z\ge -\gga}.
\end{equation} 
Note that we got a better bound than \eqref{eq:A6} because we are assuming $\fx_1\ge2$.

Next we deal with the other term inside the exponential in \eqref{secondtermww}.
Assume first that $z\leq\fg(0)$ (so in particular what follows holds also for $z\leq-\gga$).
We claim then that, as in \eqref{eq:A5}, there are $\kappa_1>0$ and $C<\infty$ such that
\begin{equation}\label{thefunnybd}
\bar F_\ep(-\fx_{2}-\fs,\fb)\le -\kappa_1 \fs^3 + C.
\end{equation}
We still have the linear lower bound \eqref{eq:A4}, $\fb\ge  -\gga- \g\tts\fs$.
Note first of all that the random walk simply cannot jump upwards farther than $\ep^{-1/2}\fs$ in time $\fs$ and therefore $\fb\le z+\ep^{-1/2}\fs$. 
Then we have $\bar\fx_{\ep}\leq 0$ for small enough $\ep$ (see \eqref{eq:thexueps}; here we have used  $z\leq\fg(0)$).
Assume furthermore that $\bar\fx_\ep^2 + \fb\ge 0$.
It is easy to check then that, for $\fs<\ep n$ (as we have in \eqref{secondtermww}), the hypotheses of Lem.\,\ref{lem:B1}\ref{item3} are satisfied for small $\ep$, so we have using also \eqref{eq:Fdecomp} that
\begin{equation}\label{thethingiuse}
\bar F_\ep(-\fx_{2}-\fs,\fb)\le \hat F_\ep(\bar\fx_\ep,\bar u_\ep)\leq(1+\nu_1) [\bar\fx_\ep\fb+\tfrac23\bar\fx_\ep^3- \tfrac23(\bar\fx_\ep^2+\fb)^{3/2}]
\end{equation}  
with $\nu_1$ non-negative and bounded (note that the second term in \eqref{eq:Fdecomp} is clearly negative in our case; additionaly, note that in \eqref{eq:Fdecomp} the middle term in the first parenthesis has a $+1/3$ in the middle term, the $-2/3$ here is because we are writing $\fx\fy - \tfrac13\fx^3 = \fx\fu +\tfrac23 \fx^3$).
Using the linear lower bound $\fb\ge  -{\tilde{\bm{\alpha}}}- {\tilde{\bm{\gamma}}}\tts\fs$, and the definition \eqref{eq:thexueps} of $\bar\fx_\ep$, which in our case is negative, \eqref{thefunnybd} follows from \eqref{thethingiuse}. 
The alternative, $\bar\fx_\ep^2+\fb<0$ can only happen for $\fb<0$ and $0\le \fs\le \fs_0$, for a constant $\fs_0<\infty$ depending on $\gga$, $\g$ and $\fx_2$.
Then, since we also have $\fb \ge -\gga- \g\tts\fs_0$, we can use  Lem.\,\ref{lem:B1}\ref{item5} to get $\hat F_\ep(-\fx_{2,\ep}-\fs,\fb)\le C$ and hence, since $\fs\le \fs_0$,  \eqref{thefunnybd}.

Now suppose $z>\fg(0)$.
We can use Lem.\,\ref{lem:B1}\ref{item3} because the condition
$ -\bar\fx_\ep+\sqrt{\bar\fx_\ep^2+ \bar u_\ep  }< \ep^{-1/2}$ reduces in this case to  $\tfrac12\ep^{-3/2} - \ep^{-1}s + \ep^{-1}\fx_2>-1$, which holds because $s<\ep n=\frac12\ep^{-1/2}$.
We get $\bar F_\ep(-\fx_2,z)\leq \hat F_\ep(\bar \fx_\ep, \bar u_\ep)\le C$.
 
Let $\fsigma^\ep$ be the hitting time of  the epigraph of $-\gga-\g |\fx|$ by the random walk $\fB_\ep$. Clearly $\ftau^\ep\ge \fsigma^\ep$.
We have $\pp_z( \fsigma^\ep\le \fs) = \pp_z\big(\sup_{0\le \fx\le \fs} \fB_\ep(\fx) + \gga+\g\fx>0\big)$. $ \fB_\ep(\fx) + \gga+\g\fx$ is a 
submartingale, and by Doob's submartingale inequality, for any $\lambda>0$,
\begin{multline}
\pp_z\Big(\sup_{0\le \fx\le \fs}  \fB_\ep(\fx) + \gga+\g\fx>0\Big)\le \ee_z\!\left[ e^{\lambda( \fB_\ep(\fs)  + \gga+\g\fs)} \right]
= e^{\lambda(z+\gga+\g\fs)+\ep^{-1}\fs \log M(\ep^{1/2} \lambda) }\nonumber
\end{multline}
where $M(\lambda) = e^{\lambda}/(2-e^{-\lambda})$, $\lambda>-\log 2$, is the moment generating function of a centered negative Geom$[\tfrac12]$ random variable.  By choosing $\lambda>0$ carefully, we can find a $\kappa_2 >0$ such that  the right hand side is bounded above
by $\exp\{- \kappa_2 \tts\fs^{-1} (z+\gga+\g\fs)^2 \}$ when $z\leq-\gga$, and therefore for such $z$,
\begin{align}\label{eq:B9a}
\pp_z( \ftau_\ep\le  \fs)\le \exp\{- \kappa_2\ts \fs^{-1} (z+\gga+\g\fs)^2 \}. 
\end{align}Putting it all together we have that $\|( \fT^\ep_{-1,\fx_1} )^* \bar\fT^{\ep,\epi((\fg^\ep)^-)}_{-1,-\fx_2}\|_1$ is bounded by a constant multiple of 
\begin{multline}
 \int^{-\gga}_{-\infty}\d z \int_0^\infty\d s\, e^{ - \kappa_2 \fs^{-1} (z+\gga+\g\fs)^2 + C\log |z| -\kappa_1\fs^3 }\\
+\int_{- \gga}^{\infty}\d z \int_{\fs\geq0,\,\fb\in\rr}\, \pp_{\fB_\ep(0)=z}(\ftau_\ep\in\d\fs,\,\fB_\ep(\ftau_\ep)\in\d\fb) e^{-\frac13|z|^{3/2}+\bar F_\ep(-\fx_2-\fs,\fb)}.
\end{multline}
The first integral clearly converges.
The second one can be split into $z\in[-\gga,\fg(0)]$ and $(\fg(0),\infty)$.
On the first piece the bound \eqref{thefunnybd} still holds, so the integral is clearly finite.
On the second piece we observe that $z>\fg(0)$ forces $\ftau_\ep=0$ and $\fb=z$, so the integral is just $\int_{- \gga}^{\infty} \d z\,e^{-\frac13|z|^{3/2}+\bar F_\ep(-\fx_2,z)}\le C$.

The next step is to extend to  $L>0$.

Suppose that $\ff_\ell$ is the $\LC$ cutoff of some $\ff\in\LC$ to the right of position $\ell$.
Then if $L>\ell$, because the random walk is free for time $L-\ell$ (see \eqref{eq:epishift}), 
\begin{equation}\label{shiftything}
(\fT^{\ep}_{-1,\fx_1-L})^*\bar\fT^{\ep,\epi((\ff_\ell)^{L,-})}_{-1,-\fx_2+L}= (\fT^{\ep}_{-1,\fx_1-\ell})^*\bar\fT^{\ep,\epi(\ff^{\ell,-})}_{-1,-\fx_2+\ell}.
\end{equation}
On the other hand, we can write $(\fT^{\ep}_{-1,\fx_1-L})^*\bar\fT^{\ep,\epi(\ff^{L,-})}_{-1,-\fx_2+L}$ as a telescoping series
\begin{align} \label{eq:alternatingseries}
(\fT^{\ep}_{-1,\fx_1})^*\bar\fT^{\ep,\epi(\ff^{0,-})}_{-1,-\fx_2}
+ \sum_{\ell=1}^L\Big[(\fT^{\ep}_{-1,\fx_1-\ell})^*\bar\fT^{\ep,\epi(\ff^{\ell,-})}_{-1,-\fx_2+\ell }- (\fT^{\ep}_{-1,\fx_1-(\ell-1)})^*\bar\fT^{\ep,\epi(\ff^{\ell-1,-})}_{-1,-\fx_2+\ell-1 }\Big].
\end{align}
We have the following estimate:

\noindent
\begin{lem} \label{balkk}
For each $\delta>0$, there is a $C<\infty$ such that for all $\ell\leq L$
\begin{equation}\label{eq:balkk}
\|(\fT^{\ep}_{-1,\fx_1-\ell})^*\bar\fT^{\ep,\epi((\fg^\ep_L)^{\ell,-})}_{-1,-\fx_2+\ell }- (\fT^{\ep}_{-1,\fx_1-(\ell-1)})^*\bar\fT^{\ep,\epi((\fg^\ep_L)^{\ell-1,-})}_{-1,-\fx_2+\ell-1 }\|_1\le C\tts e^{-(\frac23-\delta) \ell^3}
\end{equation}
\end{lem}

Before proving the lemma we will employ it to finish the proof of Prop. \ref{prop:epsilonkernel}.
Using \eqref{eq:balkk} in \eqref{eq:alternatingseries} we have, for $\fx_1\ge 2,\fx_2\ge 2$,
\begin{equation}
\|\bar\P_0(\fT^{\ep}_{-1,\fx_1-L})^*\bar\fT^{\ep,\epi((\fg^\ep_{L})^{L,-})}_{-1,-\fx_2+L}\bar\P_0\|_1 \le \|\bar\P_0(\fT^{\ep}_{-1,\fx_1})^*\bar\fT^{\ep,\epi((\fg^\ep_L)^{0,-})}_{-1,-\fx_2}\bar\P_0\|_1 + C\sum_{\ell=1}^L e^{-\frac12\ell^3}\le C',
\end{equation}
uniformly in $L$ as needed, using the $L=0$ case of the proposition (note that $(\fg^\ep_L)^{0,-}(\fy)=(\fg^\ep_0)^{0,-}(\fy)$ for all $\fy\geq0$).
The condition $\fx_1,\fx_2\geq2$ can now clearly be dropped using this estimate by taking $L$ larger if needed and shifting $\fg$ slightly.  
\end{proof}

\begin{proof}[Proof of Lem. \ref{balkk}] 
Note first that for all $\ell\leq L$ we have $(\fg^\ep_L)^{\ell,-}(\fy)=(\fg^\ep_\ell)^{\ell,-}(\fy)$ for all $\fy\geq0$.
Using this and \eqref{shiftything} the operator inside the trace norm in \eqref{eq:balkk} can be written as
\begin{equation}\label{eq:trnweneed}
(\fT^{\ep}_{-1,\fx_1-\ell})^*\big[\bar\fT^{\ep,\epi((\fg^\ep_\ell)^{\ell,-})}_{-1,-\fx_2+\ell }-\bar\fT^{\ep,\epi((\fg^\ep_{\ell-1})^{\ell,-})}_{-1,-\fx_2+\ell }\big]
\end{equation} 
We can write the kernel of the operator in brackets at $(z,z_2)$ as
\begin{equation}
\ee_{B^\ep_0=z}\!\left[ \bar\fT^{\ep}_{-1,-\fx_2+\ell-\ftau_\ep} (\fB_\ep(\ftau_\ep) -z_2)\uno{\ftau_\ep<\ep n+\ell} -\bar\fT^{\ep}_{-1,-\fx_2+\ell-\bar \ftau_\ep} (\fB_\ep(\bar \ftau_\ep) -z_2)\uno{\bar \ftau_\ep<\ep n+\ell}\right]
\end{equation}
where $n$ is still $\frac12\ep^{-3/2}+\ep^{-1}\fx_2+1$ (as above) while $\ftau_\ep$ is the hitting time of $\epi((\fg^\ep_\ell)^{\ell,-})$ and $\bar \ftau_\ep$ is the hitting time of $\epi((\fg^\ep_{\ell-1})^{\ell,-})$ by the scaled walk $\fB_\ep(\fx) = \ep^{1/2}\big(B_{\ep^{-1}\fx} + 2\ep^{-1}\fx-1\big)$.
Note that $\bar\ftau_\ep\geq\ftau_\ep$ with equality whenever $\ftau_\ep\geq1$.
So we can put $\uno{\ftau_\ep<1}$ inside the expectation and use the sum instead of the difference in our bounds.
We can then estimate the trace norm of \eqref{eq:trnweneed} (on $L^2((-\infty,0])$) by
\begin{align}\label{traceclassrankone}
&\int p^\ep_z(\d\fs,\d\bar \fs,\d\fb,\d\bar \fb)\d z\,\uno{\fs<1}\| \fT^{\ep}_{ -1,\fx_1-\ell}(z,z_1) \uno{z_1\le0}\bar\fT^{\ep}_{-1,-\fx_2+\ell-\fs} (\fb -z_2)\uno{z_2\le 0}\|_1 \\\label{traceclassrankone'}
&\hspace{0.15in}+\int p^\ep_z(\d\fs,\d\bar \fs,\d\fb,\d\bar \fb)\d z\,\uno{\fs<1,\bar \fs<\ep n+\ell}
\| \fT^{\ep}_{ -1,\fx_1-\ell}(z,z_1) \uno{z_1\le0}  \bar\fT^{\ep}_{-1,-\fx_2+\ell-\bar \fs} (\bar \fb -z_2) \uno{z_2\le 0}\|_1,\qquad\mbox{}
\end{align}
where the trace norms are as operator kernels in $z_1,z_2$ (with $z$ fixed) and $p^\ep_z(\d\s,\d\bar \s,\d\b,\d\bar \b)$ is the probability measure for the diffusively rescaled random walk starting at $z$ to hit the lower curve at $(\s,\b)$ and the upper (cutoff) curve at $(\bar \s,\bar \b)$.
Note again that the centred random walk only takes jumps upwards of at most one step
and therefore $\fB_\ep(\fx) \le z+\ep^{-1/2}\fx$.
Since $(\fg^\ep_{\ell})^{\ell,-}(\fx) \ge -\gga - \g|\ell-\fx|$ we simply cannot have $\ftau_\ep<1$ if $z< -N\coloneqq-\ep^{-1/2} -\gga - \g|\ell-1|$, i.e. the integrations are restricted to $z\ge  -N$.

The operators inside the norms above are rank one, so the trace norms are the product of $L^2$ norms and we get for \eqref{traceclassrankone'} a bound of a constant multiple of
\begin{equation}\label{traceclassrankone'2}
\int p^\ep_z(\d\fs,\d\bar \fs,\d\fb,\d\bar \fb)\d z\,\uno{\fs<1,\bar \fs<\ep n+\ell}  \exp\!\left\{ F_\ep(\fx_{1}-\ell,z)+ \bar F_\ep(-\fx_2+\ell-\bar \fs,\bar\fb)\right\}.
\end{equation}
We will proceed to estimate this one; it will be clear from the proof that the same argument works for \eqref{traceclassrankone} with a few simplifications since $\ftau_\ep,\fb$ are under better control than $\bar\ftau_\ep,\bar\fb$.

Note first of all that, in the same way as in \eqref{eq:B9a}, for $z\leq-\gga-\g\ell$ we have
\begin{equation}\label{thepbd}
p^\ep_z(\fs<1)\le \exp  \{  -\kappa (z+ \gga + \g\ell)^2  \} 
\end{equation}
(and furthermore it simply vanishes if $z<-N$ as we just argued).
We have the lower bound $\bar\b \ge -\gga - \g|\ell-\bar \fs|$, and we claim that this is enough to get a bound 
\begin{equation}\label{eq:B13a}
\bar F_\ep(-\fx_2+\ell-\bar \fs,\bar\fb)\le C(1+ \log \ell)
\end{equation}
 with $C<\infty$ independent of $\ep$ or $\ell$.  
Call  $\bar\fx_\ep=-\fx_2+\ell-\bar\s +\ep^{1/2}\bar\b/2 +3\ep /2 $ and $\bar u_\ep = \bar\b +\ep^{1/2}$.
Usually we have been breaking into cases $\bar\fx^2_\ep + \bar u_\ep \ge 0$ or not, but now let's suppose we have the stronger condition $\tfrac12\bar\fx^2_\ep + \bar u_\ep \ge 0$. Of course this implies $\bar\fy_\ep =\bar\fx_\ep^2 +\bar u_\ep \ge 0$ so $\nu_1,\nu_2\ge 0$.  The form $ \bar\fx_\ep\bar\fy_\ep - \tfrac12 \bar\fx_\ep^2\bar\fy_\ep^{1/2}-\tfrac12\bar\fy_\ep^{3/2}\le 0$ by Young's inequality and the fact that $\bar\fy_\ep\ge 0$  but now the extra condition $\tfrac12\bar\fx^2_\ep + \bar u_\ep \ge 0$ allows us to check the non-obvious fact that
$ \bar\fx_\ep  \bar\fy_\ep -    \tfrac13 \bar\fx_\ep ^3 -\tfrac23\bar\fy_\ep^{3/2}\le 0$ as well.  Therefore under these conditions we have $ \hat F_\ep(\bar\fx_\ep,\bar u_\ep)\le 0$.
Now there exists a $C$ depending on $\fx_2$, $\gga$, $\g$ such that if $\bar\fx_\ep\ge C\ep^{1/2}$ then, since $-\gga-\g\ell$ $\bar u_\ep \ge - \ep^{-1/2} \bar\fx_\ep$ we can use Lem.\,\ref{lem:B1}\ref{item1} to get the bound $C(1+\log\ell)$.  On the other hand, if $\bar\fx_\ep<C\ep^{1/2}$ we estimate $\bar F_\ep(-\fx_2+\ell-\bar \fs,\bar\fb)\le \hat F(\bar\fx_\ep,\bar u_\ep)\le 0$ using the above argument for the last inequality and Lem.\ref{lem:B1}\ref{item3} for the first, which  we are allowed to use because of the condition $\bar\fs\le \ep n+\ell$.

Alternatively we have  $\tfrac12\bar\fx^2_\ep + \bar u_\ep <0$.
Because of the lower bound on $\bar\b$, there is a $C<\infty$, depending only on $\gga$ and $\g$ such that $|\bar\fx_\ep|^3 + |\bar u_\ep|^{3/2} \le C$.
Then we can use Lem.\,\ref{lem:B1}\ref{item1}  when $\bar\fx_\ep\ge 0$, or, when $\bar\fx_\ep< 0$,  Lem.\,\ref{lem:B1}\ref{item5} if  $\bar\fx^2_\ep + \bar u_\ep <0$ or  Lem.\,\ref{lem:B1}\ref{item3} if $\bar\fx^2_\ep + \bar u_\ep \ge0$ to prove \eqref{eq:B13a}.

Now we consider  $ F_\ep(\fx_{1}-\ell,z)$.   We claim that
\begin{equation}\label{eq:theFepbd}
F_\ep(\fx_{1}-\ell,z)\le - (\tfrac23-\delta) \ell^3 + C\delta^{-1} |z|^{3/2}\uno{z\le -\gga-\g\ell}-
  C |z|^{3/2}\uno{z\ge -\gga-\g\ell} +C 
\end{equation}
Let $\fx_\ep=\fx_1-\ell-\ep^{1/2} z/2 - \ep/2$ and $u_\ep = z-\ep^{1/2}$.
Consider first the case $\tfrac12 \fx_\ep^2 + u_\ep \le 0$.
Since $z\ge -\gga - \g\ell -\ep^{-1/2}$, this can only happen if $\ell\le C\ep^{-1/4}$ and we have $ |\fx_\ep|\le C\ep^{-1/4}$, $|u_\ep|\le C\ep^{-1/2}$ and $|u_\ep|^{1/2}/|\fx_\ep|$ bounded so, if $\fx_\ep<0$ we can use Lem.\,\ref{lem:B1}\ref{item4} to get \eqref{eq:theFepbd}.
Suppose on the other hand that 
$\tfrac12 \fx_\ep^2 + u_\ep \ge 0$.  If $-\fx_\ep+\sqrt{\fx_\ep^2+ u_\ep  }\ge \ep^{-1/2}$, by Lem \ref{lem:B1}\ref{item0} there is nothing to estimate.
Otherwise we have $-\fx_\ep\le -\fx_\ep+\sqrt{\fx_\ep^2+ u_\ep  }<\ep^{-1/2}$, which together with the lower bound on $z$ gives $\ep^{1/2}\fx_\ep\in(1-\sqrt{5},1+\sqrt{5})$ independent of $\ell$, so we can use Lem.\,\ref{lem:B1}\ref{item3}  and \eqref{eq:Fdecomp} together with the facts that $\nu_1,\nu_2$ are non-negative and bounded independent of $\ell$  and that $\fx_\ep\fy_\ep  - \tfrac12 \fx_\ep^2\fy_\ep^{1/2}-\tfrac12\fy_\ep^{3/2}\leq0$ to get $F_\ep(\fx_{1}-\ell,z)\le  \fx_\ep u_\ep +   \tfrac23 \fx_\ep^3 -\tfrac23 (\fx_\ep^2+u_\ep)^{3/2}$.
Using $\fx_\ep u_\ep \le \delta |\fx_\ep|^3 + C\delta^{-1} |u_\ep|^{3/2}$ and $z\ge -\gga - \g\ell -\ep^{-1/2}$,  this is bounded above by $- (\tfrac23-\delta) \ell^3 + C\delta^{-1} |z|^{3/2}\uno{z\le -\gga-\g\ell}- C |z|^{3/2}\uno{z\ge -\gga-\g\ell} +C $ for sufficiently small $\ep$.  There is finally the case $\fx_\ep\ge 0$ but $\tfrac12 \fx_\ep^2 + u_\ep < 0$.  Here $\ell$ is bounded and we can estimate by $C\log(|z|+C)$ by  Lem.\,\ref{lem:B1}\ref{item1} which can be absorbed into the right hand side of \eqref{eq:theFepbd}.
This completes the proof of \eqref{eq:theFepbd}.

Now we can use \eqref{thepbd}, \eqref{eq:B13a} and \eqref{eq:theFepbd} to see that
\begin{align} \eqref{traceclassrankone}& \le 
C\tts e^{-(\frac23-\delta) \ell^3}\int\!\d z\, e^{ C\delta^{-1} |z|^{3/2}\uno{z\le -\gga-\g\ell} - C |z|^{3/2}\uno{z\ge -\gga-\g\ell}-\kappa (z+ \gga + \g\ell)^2\uno{ z+ \gga + \g\ell\le 0}} \uno{z\geq-N}.
\end{align}
It is not hard to see that this last term is bounded by $C'e^{- (\frac23-\delta) \ell^3}$, completing the proof.
\end{proof}

\begin{rem}
 Examining the argument, it is not hard to see that one should be able to get away with $\g(\fx) \ge -\gga - \g |\fx|^2$ if $\g$ is sufficiently small, depending on $\ft$.
From the variational formula \eqref{eq:var} it is clear that $\g=1/\ft$ is the physical barrier, but in fact the above argument breaks down at $\g= c/\ft$ with $c\approx 0.9$.
In fact, for $c/\ft <\g<1/\ft$ one has to do a very fine estimate on an oscillatory integral in order to control things.  
We do not pursue it here.
\end{rem}
 
\begin{rem}\label{rem:rwtoBMconjug}
There is still one term not taken care of by the preceeding argument; we need to show that for  $i<j$, $\fx_i>\fx_j$ and $\fa_i,\fa_j$ and the scaling introduced in \eqref{eq:KPZscaling} as well as the scaled variables $z_i=2\ep^{-1}\fx_i+\ep^{-1/2}(u_i+\fa_i)-2$ introduced in Lem. \ref{lem:KernelLimit1},
\begin{equation}
\label{convergence1a}
\Big\|{\vartheta_i( u_i)}\bar\chi_{\fa_i}(u_i)\!\left(\ep^{-1/2}Q^{n_j-n_i}(z_i,z_j)- e^{(\fx_i-\fx_j)\p^2}(u_i,u_j) \right)\!\bar\chi_{\fa_j}(u_j)\vartheta_j^{-1}( u_j) \Big\|_1\xrightarrow[\ep\to0]{}0.
\end{equation}
The pointwise convergence $\ep^{-1/2}Q^{n_j-n_i}(z_i,z_j)\longrightarrow e^{(\fx_i-\fx_j)\p^2}(u_i,u_j)$ is just a standard convergence of random walk transition probabilities
to those of Brownian motion.
To see that with these conjugations the convergence holds in trace norm, we write $\ep^{-1/2}Q^{n}- e^{\fx\p^2}$ as
\begin{equation}
\tfrac12( \ep^{-1/2}Q^{\lfloor \frac{n}2\rfloor}+e^{\frac{\fx}2\p^2})( \ep^{-1/2}Q^{\lceil \frac{n}2\rceil}- e^{\frac{\fx}2\p^2})+\tfrac12( \ep^{-1/2}Q^{\lfloor \frac{n}2\rfloor}-e^{\frac{\fx}2\p^2})( \ep^{-1/2}Q^{\lceil \frac{n}2\rceil}+ e^{\frac{\fx}2\p^2}).
\end{equation}
We can estimate the trace norm of each of the two terms as in the proof of Lem. A.2 of \cite{bfp}: Introducing the factors $\vartheta_i$ and $\vartheta_j^{-1}$ and bounding the trace norm of the product by the product of the Hilbert-Schmidt norms we get, for twice the first term, a bound by the square root of
$
\int_{-\infty}^\infty\d u_i\,\d z(  \ep^{-1/2}Q^{\lfloor \frac{n_j-n_i}2\rfloor}+e^{\frac{\fx_i-\fx_j}2\p^2})^2(u_i-z)\frac{(1+u_i^2)^{4i}}{ (1+z^2)^{2(i+j)}}
\int_{-\infty}^\infty\d u_i\,\d z(  \ep^{-1/2}Q^{\lceil \frac{n_j-n_i}2\rceil}-e^{\frac{\fx_i-\fx_j}2\p^2})^2(z-u_j)\frac{(1+z^2)^{2(i+j)}}{ (1+u_j^2)^{4j}}
$,
which goes to zero as $\ep\searrow0$; the other term is essentially the same.
\end{rem}

\subsection{Convergence}\label{sec:cvgce}

We now explain how the above trace estimates prove the convergence of the TASEP approximations to their continuum versions.
The same argument shows the continuity of the limiting kernels from $\UC$ to the trace class.

We want to show that 
\begin{equation}\label{precvgcethng}
(\fT^\ep_{-1,\fx} )^* \bar\fT^{\ep,\epi(\fg^{\ep})}_{-1,-\fx}\xrightarrow[\ep\to0]{}( \fT_{-1,\fx} )^* \bar\fT^{\epi(\fg)}_{-1,-\fx}
\end{equation}
in trace norm as $\fg^\ep\longrightarrow \fg$ in $\LC[0,\infty)$.
More explicitly, if $\ftau^\ep$ is the hitting time of $\fg^\ep$ by $\fB_\ep$, we want
\begin{multline}
\int_{\fs\in[0,\infty)\atop\fb,z\in\rr}\!\d z\,\pp_{\fB_\ep(0)=z}(\ftau_\ep\in \d\fs,\fB_\ep(\tau_\ep)\in \d\fb)\,\fT^\ep_{-1,\fx}(z,z_1)\bar\fT^\ep_{-1,-\fx-\fs}(\fb,z_2)\\ 
\longrightarrow \int_{\fs\in[0,\infty)\atop\fb,z\in\rr}\!\d z\,\pp_{\fB(0)=z}(\ftau\in \d\fs,\fB(\tau)\in \d\fb)\,\fT_{-1,\fx}(z,z_1)\bar\fT_{-1,-\fx-\fs}(\fb,z_2)\label{convergencething}
\end{multline}
as integral operators in the trace class.  First we claim that
\begin{equation}\label{asmeasures}\pp_{\fB_\ep(0)=z}(\ftau_\ep\in \d\fs,\fB_\ep(\ftau_\ep)\in \d\fb) \longrightarrow \pp_{\fB(0)=z}(\ftau\in \d\fs,\fB(\ftau)\in \d\fb)
\end{equation}
as measures.
By Donsker's invariance principle $\fB_\ep\longrightarrow \fB$ uniformly on compact sets.  By Prop.\,\ref{keyconvoftau} if $z<\fg(0)$  we have \eqref{asmeasures}, and furthermore the convergence is uniform over sets of locally bounded H\"older $\beta$-norm.  On the other hand if $z>\fg(0)$ then, from our assumption, $z>\fg^\ep(0) $ for sufficiently small $\ep$, and then $\ftau_\ep=\ftau=0$, so the convergence as measures also happens for $z>\fg(0)$. 

By the estimates proved in the previous section, one can restrict the integration in \eqref{convergencething} to $\fs,\fb$, and $z$ in compact intervals, the complementary integration being uniformly small in $\ep$ in trace norm.  Since we are now on a compact interval of $z$, we can normalize to make $\mu_\ep = c_\ep\pp_{\fB_\ep(0)=z}(\ftau_\ep\in\d\fs,\fB_\ep(\tau_\ep)\in\d\fb)\d z$ and
$\mu = c\tts\pp_{\fB(0)=z}(\ftau\in\d\fs,\fB(\tau)\in\d\fb)\d z$ probability measures.
Taking $\mathscr{B}$ as the trace class, the convergence is then a consequence of $\|\tsm\int (f_\ep-f) \d\mu_\ep\|\le \int\tsm\|f_\ep-f \|\d\mu_\ep$ and the following fact, which is presumably well known, but whose proof is easier than searching for a reference.

\begin{prop}\label{banach}
Let $\mu_n$, $n=1,2,\ldots$, and $\mu$ be probability measures on a Polish space $\Omega$ with $\mu_n$ converging weakly to $\mu$ as $n\to\infty$.
Let $f$ be a bounded continuous function from $\Omega$ into a Banach space $\mathscr{B}$.
Then $\int\tsm f\ts\d\mu_n\longrightarrow \int\tsm f\ts\d\mu$.
\end{prop}

\begin{proof}
The $\mu_n$ are tight, so, by throwing away a set of uniformly small measure, we can assume $\Omega$ is compact.
By the Skorokhod representation theorem, there exist random variables $X_n$, $n=1,2,\ldots,$ and $X$, distributed according to $\mu_n$ and $\mu$ with $X_n\longrightarrow X$ almost surely.
Let $\ep>0$.
Since $\Omega$ is compact, there is a $\delta>0$ so that $\| f (\omega)- f(\omega')\|< \ep/2$ whenever $d(\omega,\omega')< \delta$.
So $\|\ee[ f(X_n)-f(X) ] \| \le \|f\|_\infty \pp(d(X_n,X) \ge \delta) + \ep/2 <\ep$ for $n$ sufficiently large.
\end{proof}

This proves \eqref{precvgcethng} which, in view of the discussion around \eqref{eq:Klim}, yields Prop. \ref{prop:Kfixedpthalf}, i.e. the convergence of the TASEP finite dimensional distributions to those of the fixed point in the case of one-sided initial data.
The extension two-sided initial data (i.e. Thm. \ref{tm:2}) is done as explained in Sec. \ref{sec:1to2sided}, the main tool being the finite propagation speed result, Lem. \ref{cutofflemma}, proved in the next section.

\section{Finite propagation speed and regularity}\label{sec:finitespeedandreg}

\subsection{Finite propagation speed}\label{app:cutoff}

We will now prove Lem. \ref{cutofflemma}.
First of all, note that by a union bound it suffices to prove the result for $m=1$.
The resulting constants will then depend on $m$.
In reality, the constants are independent of $m$ if all the points $\fx_i$ are bounded above.
But we do not use it anywhere, so we just give the simplest proof which suffices for the results of this article.
For simplicity, we will also set $\ft=1$, and by shifting the initial data we can set $\fa =0$.

Recall that we are given initial data $X_0^\ep$ satisfying \eqref{x0limit} in the $\UC$
topology, which in particular means that we have a fixed $\gga,\g<\infty$ such that 
$\fh^\ep_0(\fy) \le \gga+\g|\fy|$ for all $\ep>0$.  
Then we make the UC cutoff, replacing it by $X_0^{\ep,L}$ in which all the particles with label less than or equal to $- \ep^{-1}L'$ are moved to $\infty$.
The corresponding rescaled height function is denoted by $\fh_0^{\ep,L}$ and we choose $L'\sim L/2$ in $\ep\zz$ so that this function has been replaced by a straight line with slope $-2\ep^{-1/2}$ to the right of $L$  (see Definition \ref{def:cutoff}).  We will prove that the difference of \eqref{eq:TASEPtofp} with $m=1$, that is $\pp\!\left(X_{2\ep^{-3/2}\ft}(\tfrac12\ep^{-3/2}\ft-\ep^{-1}\fx-\tfrac12\ep^{-1/2}\fa+1)>2\ep^{-1}\fx-2\right)$ for some fixed $\fa$, $\fx$, computed with initial data $X_0^{\ep,L-1}$ and initial data $X_0^{\ep,L}$, is less than $Ce^{-cL^3}$.
Then one just sums over integers $L\ge L_0$  to get Lem. \ref{cutofflemma}.

Note that to keep the notation as simple as possible we can assume in this subsection that $\fx=0$, or to make the notation even simpler $\fx=\ep$ because this makes $a=0$ in \eqref{eq:KPZscaling}.
From translation invariance we don't actually lose generality by assuming this.

To bound the difference of probabilities we use the Fredholm determinant formula and \eqref{eq:detBd}, which reduces the problem to estimating the trace norm of $\bar\chi_0(\SM_{-t,-{n}})^*\big( {\SN}_{-t,n}^{\oepi(X^{\ep,L-1}_0)}-{\SN}_{-t,n}^{\oepi(X^{\ep,L}_0)}\big)\bar\chi_0 $ with $t=2\ep^{-3/2}$ and $n=\frac12\ep^{-3/2}$.
It is more convenient to use the cutoff position as a frame of reference so we translate 
the cutoff rescaled height configurations to the left by a macroscopic distance $L$ so that the rightmost (non-infinite) particle of $\theta_{-\ep^{-1}L'}X_0^{\ep,L}$ has label $1$ and $\theta_{-\ep^{-1}L'}X_0^{\ep,L-1}$ is the same configuration, except that about $\ep^{-1}/2$ of the rightmost particles have been moved to $+\infty$.
The shifted macroscopic height functions have been replaced by a line of slope $-2\ep^{-1/2}$ to the right of $0$ in the first case, and to the right of $-1$ in the second case.
To view the system from the same position as before, we have to replace $n$ by $n+\ep^{-1}L'$ and so the problem comes down to bounding the trace norm of $\bar\chi_0(\SM_{-t,-n-\ep^{-1}L'})^*\big( {\SN}_{-t,n+\ep^{-1}L'}^{\oepi(\theta_{-\ep^{-1}L'}X^{\ep,L-1}_0)}-{\SN}_{-t,n+\ep^{-1}L'}^{\oepi(\theta_{-\ep^{-1}L'}X^{\ep,L}_0)}\big)\bar\chi_0 $.
In the language of the rescaled kernels from Lem. \ref{lem:KernelLimit1}, we need to bound the trace norm of 
\begin{equation}\label{eq:trnweneed'}
\bar\chi_0 (\fT^{\ep}_{ -1,-L})^* \big(\bar\fT^{\ep,\epi(-(\theta_{L}\fh_0^{\ep,L-1})^-)}_{-1,L}-\bar\fT^{\ep,\epi(-(\theta_{L}\fh_0^{\ep,L})^-)}_{-1,L}\big)\bar\chi_0,
\end{equation}
but this is exactly Lem. \ref{balkk}.

\begin{rem}\label{rem:airy-var-fps}
At the level of the fixed point, one has immediately from the one point version of the variational formula,
\begin{equation}\label{eq:var2}
\fh(1,0)  \stackrel{\uptext{dist}}{=}  \sup_{\fy\in\rr}\big\{\aip_2(\fy)- \fy^2 + \fh_0(\fy)\big\}
\end{equation} 
with $\aip_2(\fy)$ the Airy process (see Rem. \ref{rem:fullvspoint}), that replacing $ \fh_0$ by the cutoff $ \fh^L_0(\fy) =\fh_0(\fy) - \infty\uno{\fy>L} $,  affects the value of $\fh(1,0)$ only if the supremum is achieved at $\fy>L$.
Since $\fh_0(\fy) \le \gga + \g |\fy|$ this is essentially controlled by the probability that $\aip_2(L)\ge L^2 -\g L -\gga$.
Since $\aip_2(L)$ has the GUE Tracy-Widom distribution, this is roughly $\exp\{-\frac23 L^3\}$.
It is not hard to make this argument rigorous.  
Since TASEP satisfies a microscopic version of the variational formula, one could provide an alternate proof of the finite propagation speed following the same argument.  
It reduces to a large deviation bound for the tail probability in the microscopic analogue of the Airy process.
This can be proved with our formulas, and leads to similar computations as our proof of the cutoff.
We did it via Lem. \ref{balkk} because we need it later as well in the proof of regularity.
\end{rem}

\subsection{Regularity}\label{sec:reg}

Next we obtain the necessary tightness on $\fh^\ep(\ft,\fx)$ by obtaining uniform bounds on the local H\"older norm $\beta<1/2$.  
Note we are working  at $\ft$ fixed and the bounds are as functions of $\fx$, so we will assume in the rest of the section that $\ft=1$; other times can be obtained analogously, or alternatively by scaling.   We start with a well known version of the Kolmogorov continuity 
theorem.

\begin{lem}\label{lem:kolm}  
Let $\fh(\fx)$ be a stochastic process defined for $\fx$ in an interval $[-M,M]\subseteq\rr$, such that for some $p>1$ and $\alpha>0$,
\begin{equation}
\ee\! \left[ \bigl|\fh(\fx)-\fh(\fy)\bigr|^{ p}\right] \le C |\fx-\fy|^{1+\alpha} .\label{eq:hoelder0}
\end{equation}
Then for every $\beta< \alpha/p$ there is a $\bar{C}=\bar{C}(p,\alpha,\beta,C)$  such that for the local H\"older norm defined in \eqref{eq:defHolderNorm},
\begin{equation}\label{holderbd1}
\pp\!\left( \|\fh\|_{\beta, [-M,M]}\ge R \right) \leq \bar CR^{-p}.
\end{equation}
\end{lem}

We want to obtain an estimate like \eqref{eq:hoelder0} for our process, but we have only have access to \emph{cumulative} distribution functions.
We can use the following
\begin{lem} \label{lem:b2} Let $H_1,H_2$ be random variables.  Then, for $p\ge 2$, 
\begin{equation}
\ee\!\left[ |H_1-H_2|^p\uno{H_1>H_2}\right] = p(p-1) \int_{-\infty}^\infty\int_{-\infty}^\infty\d\fa\, \d\fb\,|\fa-\fb|^{p-2}\uno{\fa>\fb} \tts\pp(H_1>\fa,H_2\le \fb).\label{eq:hoelder25}
\end{equation}
\end{lem}\begin{proof}  Since the integrand is positive, the right hand side of \eqref{eq:hoelder25} can be rewritten using Fubini's theorem as
$\ee\!\left[\int_{H_2}^{H_1}\d h_2 \int_{h_2}^{H_1}\d h_1 ~ p(p-1)(h_1-h_2)^{p-2}\uno{H_1>H_2}\right]$.
Performing the integrations inside the expectation gives the left hand side.
\end{proof}

\begin{proof}[Proof of Theorem \ref{reg}]  We want to use the previous lemma, but there's a little problem.  We have  formulas for  $\pp(\fh(\fx) >\fa, \fh(\fy)\le \fb) = \pp( \fh(\fy)\le \fb)-\pp(\fh(\fx) \le \fa, \fh(\fy)\le\fb)$
in terms of differences of Fredholm determinants, which can be estimated by the trace norm of the difference of kernels.  If $\fx$ is close to $\fy$ the difference is small, as desired.
However, it is not straightforward to get the needed decay as $\fa,\fb\to \pm \infty$ (even if one knows the tails $\pp(\fh(\fx) >\fa)$ and $\pp( \fh(\fy)\le \fb)$ decay exponentially, it would still mean one had to control the difference of determinants on a range 
$c\log |\fa-\fb|^{-1}$ for $|\fa-\fb|$ small, which is non-trivial.)
We get around this with a simple trick.  Let $\fh_N(\fx) = (\fh(\fx)\wedge N)\vee (-N)$ be the cutoff of $\fh$ at $\pm N$.
Applying  Lem. \ref{lem:b2} to $H_1=\fh^\ep_N(\fx)$, $H_2=\fh^\ep_N(\fy)$ we have for fixed $\ft$ and $p\ge 2$,
\begin{multline}
\ee\! \left[ \bigl|\fh^\ep_N(\fx)-\fh^\ep_N(\fy)\bigr|^{ p}\right] \\
= p(p-1) \int_{-N}^N \int_{-N}^N \d\fa\,\d\fb\,|\fa-\fb|^{p-2} \bigl[F^\ep_\fx(\fa)\uno{\fa<\fb}+F^\ep_\fx(\fb)\uno{\fa\ge\fb} - F^\ep_{\fx, \fy}(\fa,\fb)\bigr],\label{eq:hoelder}
\end{multline}
where $F^\ep$ are the one and two point cumulative distribution functions of the 1:2:3 rescaled TASEP height functions.
Now suppose we find some $p>1$, $\alpha>0$ and $C=C(N)$ so that the right hand side is bounded by $C |\fx-\fy|^{1+\alpha}$
independent of $\ep>0$.  Suppose also that,
\begin{equation}\label{supbound}
\limsup_{N\to\infty}\limsup_{\ep\to0}\pp\bigg( \sup_{\fx\in [-M,M]}|\fh^\ep(\fx)|\ge N\bigg)=0.
\end{equation}
Then we have \eqref{tight1} because if $\sup_{\fx\in [-M,M]}|\fh^\ep(\fx)|\le N$, $\fh^\ep_N=\fh^\ep$ on $[-M,M]$.

To see that \eqref{supbound} holds, we note that we have assumed that $\fh^\ep(0,\fx) \le C(1+|\fx|)$.
We can also assume without loss of generality that $\fh^\ep(0, \fx_\ep) \ge \ell>-\infty$  independent of $\ep$ for some $\fx_\ep\longrightarrow\fx_0$, since the $\fh^\ep$ have been assumed to converge to some $\fh$ in $\UC$, and we can assume that $\fh(\fx_0)>-\infty$ for some $\fx_0$.
Therefore we can bound $\fh^\ep(1,\fx)$ above by the maximum of two rescaled TASEP height functions, one starting with $C(1+x)$ and one starting with $C(1-x)$.  For each we can estimate the probability that the height profile is  greater than $N$ anywhere on $[-M,M]$ by cutting the initial data at $L$, using the exact formulas (e.g. from Ex. \ref{example210}), and showing the bound does not depend on $L$. This is fairly standard, so we omit the details.
This shows that $\limsup_{N\to\infty}\limsup_{\ep\to0}\pp\!\left( \sup_{\fx\in [-M,M]}\fh^\ep(\fx)\ge N\right)=0$.  For the other direction, we bound $\fh^\ep(\ft, \fx)$ below by the narrow wedge solution $\underline{\fh}^\ep(\ft, \fx)$ centered at $\fx_\ep$. 
By results of Johansson \cite{johansson}, $\underline{\fh}^\ep(1, \fx) +(\fx-\fx_\ep)^2$ converges uniformly on compact sets to the (rescaled) Airy$_2$ process.
In particular, $\limsup_{N\to\infty}\limsup_{\ep\to 0}\pp\!\left( \sup_{\fx\in [-M,M]}\fh^\ep(\fx)\le - N\right)=0$. 
 
This reduces the problem to showing that for  $\fx_1<\fx_2$, 
\begin{equation}\label{eq:neededestimate}
\int_{-N\le \fa_1<\fa_2\le N}\d\fa_1\,\d\fa_2\,|\fa_1-\fa_2|^{p-2} \bigl[F^\ep_{\fx_1}(\fa_1) - F^\ep_{\fx_1, \fx_2}(\fa_1,\fa_2)\bigr]\le C(N)  |\fx_1-\fx_2|^{1+\alpha}.
\end{equation}
It may appear this is not general enough because $\fx_1<\fx_2$ and $\fa_1<\fa_2$, but our proof will apply equally well to the process with spatially flipped initial data $\fh_0(-\fx)$, which restores the symmetry.
There is one last issue which is that we only have nice formulas for these cumulative distributions when there is a rightmost particle.
So we move all the particles to the right of $\ep^{-1}L$ to $\infty$ giving a cutoff height function $\fh_0^{\ep,L}$ and corresponding one and two point distribution functions $F^{\ep,L}_\fx, F^{\ep,L}_{\fx, \fy}$ and we need the constant $C(N)$ in \eqref{eq:neededestimate} to be independent of $L$.

From \eqref{eq:TASEPtofp}, $F^{\ep,L}_{\fx_1,\fx_2}(\fa_1,\fa_2)=\pp_{X_0^{\ep,L}}\!\left(X_{2\ep^{-3/2}\ft}(n_i)>2\ep^{-1}\fx_i-2,\;i=1,2\right)$ with $n_i=\tfrac12\ep^{-3/2}\ft-\ep^{-1}\fx_i-\tfrac12\ep^{-1/2}\fa_i+1$, and where  $X_0^{\ep,L}$ is the cutoff TASEP initial data as in Def. \ref{def:cutoff}.
By translation invariance \eqref{eq:transInv} this is the same as
\begin{align}
&\pp_{\wt X_0^{\ep,L}}\!\left(X_{2\ep^{-3/2}\ft}(\tilde n_i)>2\ep^{-1}\fx_i-2,\;i=1,2\right)\\
&\hspace{0.5in}=\det\!\left(\tsm I-(\SM_{-t,-\tilde n_1})^*\SN^{\epi(\wt X_0^{\ep,L})}_{-t,\tilde n_1}\big(I-Q^{n_2-n_1}\P_{2\ep^{-1}\fx_2-2}Q^{n_1-n_2}\P_{2\ep^{-1}\fx_1-2}\big)\right)\\
&\hspace{0.5in}=\det\!\left(\tsm I-(\SM_{-t,-\tilde n_1})^*\SN^{\epi(\wt X_0^{\ep,L})}_{-t,\tilde n_1}\big(Q^{n_2-n_1}\bar\P_{2\ep^{-1}\fx_2-2}Q^{n_1-n_2}\P_{2\ep^{-1}\fx_1-2}+\bar\P_{2\ep^{-1}\fx_1-2}\big)\right),
\end{align}
where $\wt X_0^{\ep,L}=\theta_{-\ep^{-1}L}X^{\ep,L}_0$, $\tilde n_i=n_i+\ep^{-1}L$, and the first equality is from Thm.\,\ref{thm:tasepformulas} and \eqref{eq:path-int-kernel-TASEPgem} (note $n_1>n_2$ for small $\ep$).
Since $\SN^{\epi(\wt X_0^{\ep,L})}_{-t,\tilde n_1}Q^{n_2-n_1}=\SN^{\epi(\wt X_0^{\ep,L})}_{-t,\tilde n_2}$, scaling the variables in the determinant as in Sec.\,\ref{sec:one-sided} and noting that the scaled version of $\tilde X_0^{\ep,L}$ converges to $\theta_{L}\fh^L_0(\fx)$ yields 
$
F^{\ep,L}_{\fx_1,\fx_2}(\fa_1,\fa_2)=\det\big(\fI-\fA_{\fx_1,\fa_1;\fx_2,\fa_2}\big),
$
where
\[\fA_{\fx_1,\fa_1;\fx_2,\fa_2}=\bar\P_{-\fa_1}\Big[(\fT^{\ep}_{-\ft,\fx_1-L})^*\bar\fT^{\ep,\epi(-(\theta_{L}\fh_0)^-)}_{-\ft,-\fx_2+L}\bar\P_{-\fa_2}\wt Q^\ep_{\fx_2-\fx_1}\P_{-\fa_1}+(\fT^{\ep}_{-\ft,\fx_1-L})^*\bar\fT^{\ep,\epi(-(\theta_{L}\fh_0)^-)}_{-\ft,-\fx_1+L}\bar\P_{-\fa_1}\big]\]
with $\tilde Q^\ep_{\fx_2-\fx_1}(u_2,u_1)=\ep^{-1/2}Q^{n_1-n_2}(z_2,z_1)$, $z_i= 2\ep^{-1}\fx_i+\ep^{-1/2}(u_i+\fa_i)-2$. Note $\tilde Q^\ep_{\fx_2-\fx_1}(u_2,u_1)= \tilde Q^\ep_{\fx_2-\fx_1}(u_2-u_1)$ are transition densities of a centred, diffusively rescaled geometric random walk $\fB_\ep$.  

We need to estimate $F^{\ep,L}_{\fx_1}(\fa_1)-F^{\ep,L}_{\fx_1,\fx_2}(\fa_1,\fa_2)$, which by \eqref{eq:detBd} is controlled by the trace norm 
\[\big\|\fA_{\fx_1,\fa_1;\fx_1,\fa_1}-\fA_{\fx_1,\fa_1;\fx_2,\fa_2}\big\|_1=\big\|\bar\P_{-\fa_1}(\fT^{\ep}_{-\ft,\fx_1-L})^*\bar\fT^{\ep,\epi(-(\theta_{L}\fh_0)^-)}_{-\ft,-\fx_2+L}\bar\P_{-\fa_2}\wt Q^\ep_{\fx_2-\fx_1}\P_{-\fa_1}\big\|_1.\]
From Prop. \ref{prop:epsilonkernel} we have $\big\|\bar\P_{-\fa_1}(\fT^{\ep}_{-\ft,\fx_1-L})^*\bar\fT^{\ep,\epi(-(\theta_{L}\fh_0)^-)}_{-\ft,-\fx_2+L}\bar\P_{-\fa_2}\big\|_1 \le C$. Since $-\fa_2<-\fa_1$,  we also have by Schwartz's inequality,
\begin{equation}\label{heatkernelestimate}
\big\|\bar\P_{-\fa_2}\wt Q^\ep_{\fx_2-\fx_1}\P_{-\fa_1}\big\|_\uptext{op} \le  \int_{|y|\ge|\fa_2-\fa_1|} \wt Q^\ep_{\fx_2-\fx_1}(y)\d y,
\end{equation}
where $\|\cdot\|_\uptext{op}$ is the operator norm and we are using $\|AB\|_1\le \|A\|_\uptext{op}\|B\|_1$.  
When we plug this estimate into the left hand side of \eqref{eq:neededestimate}, we get $2N \ee\left[ |\fB_\ep( \fx_2-\fx_1)|^{p-1}\right]$.
The $2N$ comes from the integral over $\fa_1$ and we lose a power in the exponent because of the estimate in \eqref{heatkernelestimate} which is by a transition probability instead of the cumulative functions.
At any rate, we have $2N \ee\left[ |\fB_\ep( \fx_2-\fx_1)|^{p-1}\right]\le C(N,p) |\fx_2-\fx_1|^{\frac{p-1}2}$.
So we have proved \eqref{eq:neededestimate} with $\alpha = \frac{p-3}2$, which means that we have \eqref{holderbd1} with $\beta = \frac12 - \frac3{2p}$, for any $p\ge 2$ (though with a constant
$\bar{C}$ getting worse as $p\to \infty$.) \end{proof}

\section{Path integral formulas}

\settocdepth{section}

\subsection{An alternative version of [BCR15, Thm. 3.3]}\label{app:altBCR}

Consider a measure space $(X,\mu)$ and fix $t_1<\dotsm<t_n$.
In \cite[Sec. 3]{bcr} a very general setting is described on which the Fredholm determinant of certain \emph{extended kernels}, acting on $L^2(\{t_1,\dotsc,t_n\}\times X)$, can be turned into the Fredholm determinant of what they call a \emph{path integral kernel}, which acts on $L^2(X)$.
The result proved in that paper can be applied to a variety of processes, and in particular yields the path integral formula for the KPZ fixed point given in Prop.\,\ref{prop:pathint-fixedpt}.
However, the TASEP extended kernel does not fit into the setting of that paper.
The purpose of this subsection is thus to provide a suitable version of \cite[Thm. 3.3]{bcr}.
We do not strive for the greatest possible generality, but instead state a relatively simple variation of \cite{bcr} which is enough to obtain the TASEP path integral formula, which in turn is proved in the next subsection.

Let $\cm(X)$ be the space of real-valued measurable functions on $X$.
We are given a collection of integral operators on subspaces $\cD_\cw^+,\cR_\cw^+,\cD_\cw^-,\cR_\cw^-,\cD_K,\cR_K$ of $\cm(X)$:
\renewcommand{\labelitemi}{$\vcenter{\hbox{\tiny$\bullet$}}$}
\begin{itemize}
\item $\cw_{t_i,t_j}\!:\cD_\cw^+\longrightarrow\cR_\cw^+$, for $1\leq i<j\leq n$.
\item $\cw_{t_i,t_j}\!:\cD_\cw^-\longrightarrow\cR_\cw^-$, for $1\leq j\leq i\leq n$.
\item $K_{t_i}\!:\cD_K\longrightarrow\cR_K$, for $1\leq i\leq n$, with $\cD_K,\cR_K$ satisfying\\[-16pt] 
\begin{multline}
\hskip0.5in\cD_K,\cR_K\subseteq\cD_\cw^+\cap\cD_\cw^-,\qquad\cw_{t_i,t_j}(\cD_K)\subseteq\cD_\cw^+\cap\cD_\cw^-\quad\uptext{for all $i,j$},\\
\uptext{and}\qquad\cw_{t_i,t_j}(\cD_K)\subseteq\cD_K\quad\uptext{for all $i>j$}.\hskip0.5in\label{eq:inclusions}
\end{multline}
\end{itemize}
\renewcommand{\labelitemi}{\textendash}
Given these operators we construct an integral operator $K^\uptext{ext}$ acting on the space $\cD_K^{(t_1,\dotsc,t_n)}$ of functions $f\!:\{t_1,\dotsc,t_n\}\times X\longrightarrow\rr$ such that $f(t_i,\cdot)\in\cD_K$ for $1\leq i\leq n$ through the following integral kernel:
\begin{equation}\label{eq:generalExt}
  K^\uptext{ext}(t_i,x;t_j,y)=
  \begin{dcases*}
    \cw_{t_i,t_j}K_{t_j}(x,y) & if $i\geq j$,\\
    -\cw_{t_i,t_j}(I-K_{t_j})(x,y) & if $i<j$
  \end{dcases*}
\end{equation}
(that is, for $f\in \cD_K^{(t_1,\dotsc,t_n)}$ we set $K^\uptext{ext}f(t_i,x)=\sum_{j=1}^n\int_X\d\mu(y)\,K^\uptext{ext}(t_i,x;t_j,y)f(t_j,y)$).
Our assumptions on the domains and ranges of our operators ensure that these (and later) compositions are always well defined.

We will suppose that our operators satisfy the following algebraic assumptions: for all $f\in\cD_K$,
\begin{subequations}\label{eq:altBCR}
\begin{align}
\cw_{t_i,t_i}f&=f& &\uptext{for all $1\leq i\leq n$},
\label{eq:altBCR1}\\
\cw_{t_i,t_j}\cw_{t_j,t_k}f&=\cw_{t_i,t_k}f& &\uptext{for all $1\leq i,j,k\leq n$},
\label{eq:altBCR2}\\
\hskip1in\cw_{t_i,t_j}K_{t_j}\cw_{t_j,t_i}f&=K_{t_i}f& &\uptext{for all $1\leq i<j\leq n$}.\hskip1in
\label{eq:altBCR3}
\end{align}
\end{subequations}
These three assumptions are our replacement for Assumption 2 of \cite{bcr}.
There are two main differences with that paper: First, we are assuming here that the $\cw_{t_i,t_j}$'s are properly defined operators for $i>j$, and that $\cw_{t_i,t_j}$ and $\cw_{t_j,t_i}$ are essentially inverses of each other; and second, the reversibility relation $\cw_{t_i,t_j}K_{t_j}=K_{t_i}\cw_{t_i,t_j}$ which is assumed in \cite{bcr} is replaced here by \eqref{eq:altBCR3}\footnote{It is possible to state a version of Thm.\,\ref{thm:alt-extendedToBVP} where $\cw_{t_j,t_i}$ is not necessarily well defined for $i<j$, but instead one assumes that $\cw_{t_j,t_i}K_{t_i}$ and $K_{t_j}\cw_{t_j,t_i}$ are well defined and satisfy the necessary algebraic assumptions.}.

Additionally, we consider multiplication operators $\Ml_{t_i}$ acting on $\cm(X)$ as $\Ml_{t_i}f(x)=\varphi_{t_i}(x)f(x)$ for some $\varphi_{t_i}\in\cm(X)$.
We define $N$ to be the diagonal operator acting on functions $f\!:\{t_1,\dotsc,t_n\}\!\times\!X\longrightarrow\rr$ as $Nf(t_i,\cdot)=\Ml_{t_i}f(t_i,\cdot)$.
We also introduce the notation
\[\oM_{t_i}=I-\Ml_{t_i}.\]
We need to make some additional assumptions on our operators (the first one is convenient in light of our setting in the next subsection, but could be relaxed as in \cite{bcr}; the last four replace Assumptions 1 and 3 in \cite[Thm. 3.3]{bcr}):
\begin{enumerate}[label=(\roman*)]
\item $\varphi_{t_i}(x)\geq0$ for all $x\in X$.
We let $\Ml_{t_i}^{1/2}$ denote the operator of multiplication by $\varphi_{t_i}(x)^{1/2}$ and define $N^{1/2}$ in the same way.
\item For all $1\leq i\leq n$, $N_{t_i}^{1/2}$ maps $L^2(X)$ to $\cD_K$.
\item For all $1\leq i,j\leq n$, $N_{t_i}^{1/2}$ maps $\cw_{t_i,t_j}(\cR_K)$ to $L^2(X)$, and if additionally $i\geq j$ then it also maps $\cw_{t_i,t_j}(\cD_K)$ to $L^2(X)$.
\item $\cR_K\cup\cD_\cw^+\subseteq L^2(X)$.
\item There exist multiplication operators $V_{t_i}$, $V_{t_i}'$ satisfying $V_{t_i}'V_{t_i}N_{t_i}^{1/2}=N_{t_i}^{1/2}$ and such that for all $i<j$ the operator $V_{t_i}N_{t_i}^{1/2}\cw_{t_i,t_j}N_{t_j}^{1/2}V_{t_j}'$ is trace class in $L^2(X)$ and for all $i,j$ the operators $V_{t_i}N_{t_i}^{1/2}\cw_{t_i,t_j}K_{t_j}N_{t_j}^{1/2}V_{t_j}'$ are trace  class in $L^2(X)$.
\item There exist multiplication operators $U_{t_i}$, $U_{t_i}'$ satisfying $U_{t_i}U_{t_i}'K_{t_i}\cw_{t_i,t_j}=K_{t_i}\cw_{t_i,t_j}$ for all $j\leq i$ and such that the operator 
\[\mbox{}\qquad\quad U_{t_i}\!\left[K_{t_i}\cw_{t_{i},t_{i+1}}\oM_{t_{i+1}}
    \dotsm\cw_{{t_{n-1}},{t_{n}}}\oM_{{t_{n}}}K_{t_n}-K_{t_i}\cw_{t_{i},{t_1}}\oM_{{t_1}}\cw_{{t_1},{t_2}}\oM_{{t_2}}\dotsm\cw_{{t_{n-1}},{t_{n}}}\oM_{{t_{n}}}\right]\!U_{t_i}'\]
 is trace class in $L^2(X)$.
\end{enumerate}

\begin{theorem}\label{thm:alt-extendedToBVP}
  In the above setting, assume that \eqref{eq:inclusions}, \eqref{eq:altBCR} and assumptions (i)-(vi) are satisfied.
  Then
  \begin{multline}\label{eq:extToBVP}
    \det\!\big(I-\Ml^{1/2} K^\uptext{ext}\Ml^{1/2}\big)_{L^2(\{t_1,\dots,t_n\}\times X)}\\
    =\det\!\big(I-K_{t_n}+K_{t_n}\cw_{t_n,t_1}\oM_{t_1}\cw_{t_1,t_2}\oM_{t_2}\dotsm\cw_{t_{n-1},t_n}\oM_{t_n}\big)_{L^2(X)}.
  \end{multline}
\end{theorem}

Note that by (ii) above, the operator inside the determinant on the left hand side acts on $L^2(\{t_1,\dots,t_n\}\times X)$.
Similarly, by (iii) above and the fact that
\begin{multline}
K_{t_n}-K_{t_n}\cw_{t_n,t_1}\oM_{t_1}\cw_{t_1,t_2}\oM_{t_2}\dotsm\cw_{t_{n-1},t_n}\oM_{t_n}\\
=\sum_{j=1}^n\sum_{k=0}^{n-j}(-1)^{k}\quad\smashoperator{\sum_{j=\ell_0<\ell_1<\dots<\ell_k\leq
        n}}\quad K_n\cw_{n,j}\Ml_{j}\cw_{{j},{{\ell_1}}}\Ml_{{\ell_1}}\cw_{{\ell_1},{\ell_2}}
    \Ml_{{\ell_{k-1}}}\cw_{{\ell_{k-1}},{\ell_k}}\Ml_{{\ell_k}}\cw_{{\ell_k},{n}}
\end{multline}
(see \eqref{eq:tswti1} below), the operator inside the determinant on the right hand side acts on $L^2(X)$.
Moreover, thanks to (iv) and (v), after conjugating appropriately by the operators $V_{t_i}$, $V_{t_i}'$ on the left and $U_{t_n}$, $U_{t_n}'$ on the right, which does not change the value of the Fredholm determinants, the operators on both sides become trace class, which shows that the two Fredholm determinants are finite.

\begin{proof}[Proof of Theorem \ref{thm:alt-extendedToBVP}]
  The proof is a minor adaptation of the arguments in \cite[Thm. 3.3]{bcr}, and we will use throughout it all the notation and conventions of that paper.
  We will just sketch the proof, skipping some of the technical details.
  In particular, we will completely omit the need to conjugate by the operators $U_{t_i}$ and $V_{t_i}$, since this aspect of the proof can be adapted straightforwardly from \cite{bcr} (such conjugations are used to justify the operations involving the multiplicativity and the cyclic property of the Fredholm determinant).

  In order to simplify notation throughout the proof we will replace subscripts of the form $t_i$ by $i$, so for example $\cw_{i,j}=\cw_{t_i,t_j}$.
  Let $\sK=\Ml^{1/2}K^\uptext{ext}\Ml^{1/2}$. Then $\sK$ can
  be written as
  \begin{equation}\label{eq:sK}
    \sK=\sQ^{1/2}(\sW^{-}\sK^\uptext{d}+\sW^{+}(\sK^\uptext{d}-\sI))\sQ^{1/2}\quad\text{with}\quad
    \sK^\uptext{d}_{ij}=K_{i}\uno{i=j},\quad \sQ_{i,j}^{1/2}=\Ml_{{i}}^{1/2}\uno{i=j},
  \end{equation}
  where $\sW^{-}$, $\sW^{+}$ are lower triangular, respectively strictly upper triangular, and defined by
  \begin{equation*}
    \sW^{-}_{ij} = \cw_{i,j}\uno{i\geq j},\qquad
    \sW^{+}_{ij}=\cw_{{i},{j}}\uno{i < j}.
  \end{equation*}
  
  The key to the proof in \cite{bcr} was to observe that $\big[(\sI+\sW^{+})^{-1})\big]_{i,j}=I\uno{j=i}-\cw_{{i},{i+1}}\uno{j=i+1}$, which then implies that $\big[(\sW^{-}+\sW^{+})\sK^\uptext{d}(\sI+\sW^{+})^{-1}\big]_{i,j}=\cw_{{i},{1}}K_{1}\uno{j=1}$.
  The fact that only the first column of this matrix has non-zero entries is what ultimately allows one to turn the Fredholm determinant of an extended kernel into one of a kernel acting on $L^2(X)$.
  However, the derivation of this last identity uses $\cw_{i,{j-1}}K_{j-1}\cw_{{j-1},j}=\cw_{i,j}K_{j}$, which is a consequence of Assumptions 2(ii) and 2(iii) in \cite{bcr}, and thus is not available to us. In our case we may proceed similarly by observing that
  \begin{equation}\label{eq:IT+}
    \big[(\sW^{-})^{-1})\big]_{i,j}=I\uno{j=i}-\cw_{{i},{i-1}}\uno{j=i-1},
  \end{equation}
  as can be checked directly using \eqref{eq:altBCR1} and \eqref{eq:altBCR2}.
  Now using the identity $\cw_{i,{j+1}}K_{j+1}\cw_{{j+1},j}=\cw_{i,j}K_{j}$ (which follows from our assumption \eqref{eq:altBCR3} together with \eqref{eq:altBCR2}) we get
  \begin{equation}
      \big[(\sW^{-}+\sW^{+})\sK^\uptext{d}(\sW^{-})^{-1}\big]_{i,j}
      =\cw_{i,j}K_{j}-\cw_{i,{j+1}}K_{j+1}\cw_{{j+1},j}\uno{j<n}
      =\cw_{{i},{n}}K_{n}\uno{j=n}.\label{eq:T-T+}
  \end{equation}
  Note that now only the last column of this matrix has non-zero entries, which accounts for the different expression on the right hand side of \eqref{eq:extToBVP} when compared to \cite{bcr}. 
  To take advantage of \eqref{eq:T-T+} we write $\sI-\sK=(\sI+\sQ^{1/2}\sW^+\sQ^{1/2})\big[\sI-(\sI+\sQ^{1/2}\sW^+\sQ^{1/2})^{-1}\sQ^{1/2}(\sW^-+\sW^+)\sK^\uptext{d}(\sW^-)^{-1}\sW^-\sQ^{1/2}\big]$.
  Since $\sQ^{1/2}\sW^+\sQ^{1/2}$ is strictly upper triangular, $\ts\det(\sI+\sQ^{1/2}\sW^+\sQ^{1/2})=1$, which in particular shows that $\sI+\sQ^{1/2}\sW^+\sQ^{1/2}$ is invertible.
  Thus, using \eqref{eq:T-T+}, we deduce that $\det\!\big(\sI-\sK\big)$ is the same as $\det\!\big(\sI-(\sI+\sQ^{1/2}\sW^+\sQ^{1/2})^{-1}\sQ^{1/2}\sW^{(n)}\sK^\uptext{d}\sW^-\sQ^{1/2}\big)$
  with $\sW^{(n)}_{i,j}=\cw_{i,n}\uno{j=n}$.
  Using the cyclic property of the Fredholm determinant we deduce now that $\det(\sI-\sK)=\det(\sI-\wt\sK)$ with
  \[\wt\sK=\sK^\uptext{d}\sW^-\sQ^{1/2}(\sI+\sQ^{1/2}\sW^+\sQ^{1/2})^{-1}\sQ^{1/2}\sW^{(n)}.\]
  Since only the last column of $\sW^{(n)}$ is non-zero, the same holds for $\wt\sK$, and thus 
  \begin{equation}
  \ts\det(\sI-\sK)=\det(\sI-\wt\sK_{n,n})_{L^2(X)}.\label{eq:detKKnn}
  \end{equation}

  Our goal is thus to compute $\wt\sK_{n,n}$.
  We have, for $0\leq k\leq n-i$,
  \begin{equation*}
    \left[(\sQ^{1/2}\sW^+\sQ^{1/2})^k\sQ^{1/2}\sW^{(n)}\right]_{i,n}
    =\qquad\smashoperator{\sum_{i<\ell_1<\dots<\ell_k\leq n}}\qquad
    \Ml_{{i}}^{1/2}\cw_{i,{\ell_1}}\Ml_{{\ell_1}}\cw_{{\ell_{1}},{\ell_2}} \dotsm
    \Ml_{{\ell_{k-1}}}\cw_{{\ell_{k-1}},{\ell_k}}\Ml_{{\ell_k}}\cw_{{\ell_k},{n}},
  \end{equation*}
  while for $k>n-i$ the left-hand side above equals 0 (the case $k=0$ is interpreted as $\Ml_{i}^{1/2}\cw_{i,n}$).
  As in \cite{bcr} this leads to
  \begin{equation}
    \wt\sK_{i,n}
    =\sum_{j=1}^i\sum_{k=0}^{n-j}(-1)^{k}\quad\smashoperator{\sum_{j=\ell_0<\ell_1<\dots<\ell_k\leq
        n}}\quad K_i\cw_{i,j}\Ml_{j}\cw_{{j},{{\ell_1}}}\Ml_{{\ell_1}}\cw_{{\ell_1},{\ell_2}}
    \Ml_{{\ell_{k-1}}}\cw_{{\ell_{k-1}},{\ell_k}}\Ml_{{\ell_k}}\cw_{{\ell_k},{n}}.\label{eq:tswti1}
  \end{equation}
  Replacing each $\Ml_{\ell}$ by $I-\oM_{\ell}$ except for the first one and simplifying as in \cite{bcr} leads to
  \[\wt\sK_{i,n}=K_i\cw_{{i},{i+1}}\oM_{{i+1}}\cw_{{i+1},{i+2}}\oM_{{i+2}}
    \dotsm\cw_{{{n-1}},{{n}}}\oM_{{{n}}}-K_i\cw_{{i},{1}}\oM_{{1}}\cw_{{1},{2}}\oM_{{2}}\dotsm\cw_{{{n-1}},{{n}}}\oM_{{{n}}}.\]
  Setting $i=n$ yields $\wt\sK_{n,n}=K_{n}-K_n\cw_{{n},{1}}\oM_{{1}}\cw_{{1},{2}}\oM_{{2}}\dotsm\cw_{{{n-1}},{{n}}}\oM_{{{n}}}$, which in view of \eqref{eq:detKKnn} yields the result.
\end{proof}

\subsection{Proof of the TASEP path integral formula}\label{app:proofPathIntTASEP}

Recall that $Q^{n-m}\Psi^{n}_{n-k}=\Psi^{m}_{m-k}$.
Then for $K^{(n)}_{t}=K_t(n,\cdot;n,\cdot)$ we may write
\begin{equation}\label{eq:checkQKK}
  Q^{n_j-n_i}K^{(n_j)}_t=\sum_{k=0}^{n_j-1}Q^{n_j-n_i}\Psi^{n_j}_{k}\otimes \Phi^{n_j}_{k}=\sum_{k=0}^{n_j-1}\Psi^{n_i}_{n_i-n_j+k}\otimes \Phi^{n_j}_{k}=K_t(n_i,\cdot;n_j,\cdot)+Q^{n_j-n_i}\uno{n_i<n_j}.
\end{equation}
This means that the extended kernel $K_t$ defined in \eqref{eq:Kt} has exactly the structure specified in \eqref{eq:generalExt}, taking (for fixed $t\geq0$) 
\[t_i=n_i,\qquad K_{t_i}=K^{(n_i)}_t\qqand \mathcal{W}_{t_i,t_j}=Q^{n_j-n_i}.\] 
As a consequence, the path integral version \eqref{eq:path-int-kernel-TASEPgem} of the TASEP formula will follow from Thm.\,\ref{thm:alt-extendedToBVP}, with $N_{t_i}=\bar\P_{a_i}$, once we check all the necessary assumptions.
We do this next.

First we need to specify the domains and ranges of our operators:
\begin{gather}
\cD_K=\{f\in\ell^2(\zz)\!:\sum_{x\in\zz}2^x|x^kf(x)|<\infty~~\forall~k\geq0\},\qquad\cR_K=\ell^2(\zz),\\
\cD_\cw^+=\cR_\cw^+=\{f\in\ell^2(\zz)\!:\sum_{x\leq0}2^x|x^kf(x)|<\infty~~\forall~k\geq0\},
\qquad\cD_\cw^-=\cR_\cw^-=\cm(\zz).
\end{gather}
It is easy to check that our operators are indeed well defined with these domains and ranges, and also that \eqref{eq:inclusions} and assumptions (i)-(iv) are satisfied (essentially all one needs is to observe that $Q$ and $Q^{-1}$ preserve $\ell^2(\zz)$).
Before checking assumptions (v) and (vi), let us turn to \eqref{eq:altBCR}.

Conditions \eqref{eq:altBCR1} and \eqref{eq:altBCR2} are clearly satisfied.
We note at this stage that Assumptions 2(i) and 2(ii) in \cite[Thm. 3.3]{bcr} (the semigroup property and the right-invertibility condition) hold in our setting, but their Assumption 2(iii), which translates into $Q^{n_j-n_i}K_t^{(n_j)}=K_t^{(n_i)}Q^{n_j-n_i}$ for $n_i\leq n_j$, does not (in fact, the right hand side does not even make sense as the product is divergent, as can be seen by noting that $\Phi^{(n)}_0(x)=2^{x-X_0(n)}$),
which is why we need Thm.\,\ref{thm:alt-extendedToBVP}.
To use it, we need to check \eqref{eq:altBCR3}, that is
\begin{equation}
  Q^{n_j-n_i}K_t^{(n_j)}Q^{n_i-n_j}=K_t^{(n_i)}\label{eq:secondExtKernAssmp}
\end{equation}
for $n_i<n_j$.
In fact, if $0\leq k<n_i$ then \eqref{bhe1} together with the easy fact that $h^n_k(\ell,z)=h^{n-1}_{k-1}(\ell-1,z)$ imply that $(\Qt)^{n_i-n_j}h^{n_j}_{k+n_j-n_i}(0,z)=h^{n_j}_{k+n_j-n_i}(n_j-n_i,z)=h^{n_i}_{k}(0,z)$, so that
$(\Qt)^{n_i-n_j}\Phi^{n_j}_{k+n_j-n_i}=\Phi^{n_i}_k$.
Therefore, proceeding as in \eqref{eq:checkQKK}, the left hand side of \eqref{eq:secondExtKernAssmp} equals
\begin{equation}\label{eq:extKernAssumCheck}
 \begin{split}
  \sum_{k=0}^{n_j-1}Q^{n_j-n_i}\Psi^{n_j}_{k}\otimes(\Qt)^{n_i-n_j}\Phi^{n_j}_{k}
   &=\sum_{k=n_i-n_j}^{n_i-1}Q^{n_j-n_i}\Psi^{n_j}_{k+n_j-n_i}\otimes(\Qt)^{n_i-n_j}\Phi^{n_j}_{k+n_j-n_i}\\
  &=\sum_{k=n_i-n_j}^{-1}\Psi^{n_i}_{k}\otimes(\Qt)^{n_i-n_j}\Phi^{n_j}_{k+n_j-n_i}
  +\sum_{k=0}^{n_i-1}\Psi^{n_i}_{k}\otimes\Phi^{n_i}_{k}
 \end{split}
\end{equation}
(note the crucial fact, while $\Qt\Phi^n_\ell$ is divergent, in this argument $(\Qt)^m$ is applied to these functions only for negative $m$).
The last sum is exactly the right hand side of \eqref{eq:secondExtKernAssmp}, so we need to check that the first sum on the right hand side above vanishes.
For this we note that, if $n_i-n_j\leq k<0$, then we have 
$(\Qt)^{n_i-n_j}h^{n_j}_{k+n_j-n_i}(0,z)=(\Qt)^kh^{n_j}_{k+n_j-n_i}(k+n_j-n_i,z)=0$ thanks to \eqref{bhe2} and the fact that $2^z\in\ker(\Qt)^{-1}$, which gives $(\Qt)^{n_i-n_j}\Phi^{n_j}_{k+n_j-n_i}=0$ as desired.

To finish our proof of \eqref{eq:path-int-kernel-TASEPgem} we need to check that conditions (v) and (vi) in Thm.\,\ref{thm:alt-extendedToBVP} are satisfied.
For (v), we let $V_{n_i}=\wt V_i$ and $V'_{n_i}=\wt V'_i$ with $\wt V_{i}f(x)=(1+x^2)^{i}$, $\wt V_{i}'(x)=(1+x^2)^{-i}$.
For $i<j$ we need to check that $\bar\chi_{a_i}V_{n_i}Q^{n_j-n_i}V_{n_j}'\bar\chi_{a_j}$ is trace class.
The $\bar\chi_{a_i}$'s are projections, so $\|\bar\chi_{a_i}\|_\infty=1$ and thus it is enough to show that $V_{n_i}Q^{n_j-n_i}V_{n_j}'$ is trace class.
Assume first that $n_j-n_i>1$.
Then we may write $V_{n_i}Q^{n_j-n_i}V_{n_j}'$ as $\prod_{\ell=n_i}^{n_j-1}\wt V_{\ell}Q\wt V_{\ell+1}'$ and it is enough to show that each factor is Hilbert-Schmidt, which is clear:
\[\|\wt V_{\ell}Q\wt V_{\ell+1}'\|_2^2=\sum_{\substack{x,y\in\zz\\y<x}}2^{2(y-x)}\frac{(1+x^2)^{2\ell}}{(1+y^2)^{2(\ell+1)}}
=\sum_{y\in\zz}\inv{(1+y^2)^{2(\ell+1)}}\sum_{x>0}(1+(x+y)^2)^{2\ell}\ts2^{-2x}<\infty.\]
If $n_j-n_i=1$ then we write $Q(x,y-1)=A^2(x,y)$ with $A(x,y)=\frac{\Gamma(x-y+\frac12)}{\sqrt{\pi}(x-y)!}2^{-(x-y+1/2)}\uno{x\geq y}$ (this can be checked directly, but it just corresponds to writing a Geom$[\frac12]$ random variable as the sum of two independent negative binomial random variables with parameters $(\frac12,\frac12)$) and thus it is enough to show that $V_iA^2V'_{i+1}$ is trace class.
By Stirling's formula, $A(x,y)\sim c(x-y)^{-1/2}2^{-(x-y)}$ as $x-y\to\infty$, so factoring the kernel as $\big(V_iAV'_{i+1/2}\big)\big(V_{i+1/2}AV'_{i+1}\big)$ the above argument still applies.
Next for all $i,j$ we need to check that $\bar\chi_{a_i}V_{n_i}Q^{n_j-n_i}K_t^{(n_i)}\bar\chi_{a_j}V_{n_j}'$ is trace class.
Since $Q^{n_j-n_i}K_t^{(n_j)}=\sum_{k=1}^{n_j}Q^{n_j-n_i}\Psi^{n_j}_{n_j-k}\otimes\Phi^{n_j}_{n_j-k}$, it is enough to show (using also \eqref{eq:QmnPsi}) that $\bar\chi_{a_i}V_{n_i}\Psi^{n_i}_{n_i-k}\otimes \bar\chi_{a_j}V'_{n_j}\Phi^{n_j}_{n_j-k}$ is trace class for each $k=1,\dotsc,n_j$.
This operator is rank one and its unique eigenvalue is $\lambda=\sum_{x\in\zz}\bar\chi_{a_i}V_{n_i}\Psi^{n_i}_{n_i-k}(x)\bar\chi_{a_j}V'_{n_j}\Phi^{n_j}_{n_j-k}(x)$, which satisfies
\[|\lambda|^2\leq\sum_{x\leq a_i}(1+x^2)^{2i}\Psi^{n_i}_{n_i-k}(x)^2\sum_{x\leq a_j}\frac{2^{2x}}{(1+x^2)^{2j}}(2^{-x}\Psi^{n_i}_{n_i-k}(x))^2<\infty\]
because $\Psi^{n_i}_{n_i-k}(x)=0$ for $x<X_0(k)-n_i+k$ and $2^{-x}\Psi^{n_i}_{n_i-k}(x)$ is a polynomial.
Therefore $\|\bar\chi_{a_i}V_{n_i}\Psi^{n_i}_{n_i-k}\otimes \bar\chi_{a_j}V'_{n_j}\Phi^{n_j}_{n_j-k}\|_1=|\lambda|<\infty$ as desired.

We turn finally to (vi).
By \eqref{eq:tswti1} it is enough to show that
\begin{equation}
U_{n_i}K^{(n_i)}_tQ^{n_{\ell_0}-n_i}\bar\chi_{a_{\ell_0}}Q^{n_{\ell_1}-n_{\ell_0}}\bar\chi_{a_{\ell_1}}\dotsm Q^{n_{\ell_k}-n_{\ell_{k-1}}}\bar\chi_{a_{\ell_k}}Q^{m-n_{\ell_k}}U'_{n_m}\label{eq:productUdotsU'}
\end{equation}
is trace class for all $\ell_0<\ell_1<\dotsm<\ell_k\leq m$ with $\ell_0=1,\dotsc,i$ and $k=0,\dotsc,m-\ell_0$, where we choose $U_{n_i}=V_{n_i}$, $U'_{n_i}=V'_{n_i}$.
Now replace each $\bar\chi_{a_\ell}$ in the above operator by $V'_{n_{\ell}}\bar\chi_{a_\ell}V_{n_\ell}$.
The argument in the last paragraph can be used to show again that $U_{n_i}K^{(n_i)}_tQ^{n_{\ell_0}-n_i}V'_{n_{\ell_0}}\bar\chi_{a_{\ell_0}}$ is trace class (now using the decay of $\Psi^{n_i}_{n_i-k}(x)$ as $x\to\infty$ since we have no projection on the left).
As shown above, the rest is the product of trace class operators, so the whole product \eqref{eq:productUdotsU'} is trace class as needed.

\vs

\noindent{\bf Acknowledgements.}  
JQ and KM were supported by the Natural Sciences and Engineering Research Council of Canada.
JQ was also supported by a Killam research fellowship.
DR was supported by Conicyt Basal-CMM Proyecto/Grant PAI AFB-170001, by Programa Iniciativa Cient\'ifica Milenio grant number NC120062 through Nucleus Millenium Stochastic Models of Complex and Disordered Systems, and by Fondecyt Grants 1160174 and 1201914.

\printbibliography[heading=apa]

\end{document}